\documentclass[10pt, a4paper, reqno]{amsart}

\usepackage[T1]{fontenc}
\usepackage[utf8]{inputenc}
\usepackage{amsmath, amssymb, amsfonts, amsthm}
\usepackage[margin=2.0cm]{geometry}
\usepackage[]{setspace}%\onehalfspacing
\usepackage[usenames,dvipsnames]{color}
\usepackage[, unicode]{hyperref}
\usepackage{float}
\hypersetup{
    urlcolor=blue,
    colorlinks=true,
   linkcolor= blue,
   citecolor=blue
}
\usepackage{bm}
\usepackage{paralist}
\usepackage{amsmath}
\numberwithin{equation}{section}
\usepackage{dirtytalk}
\usepackage{amssymb}
\usepackage{am
sthm}
\usepackage{enumitem}
\setlist[itemize]{leftmargin=2em}
\setlist[enumerate]{leftmargin=2em}
\usepackage{tikz-cd}
\usepackage[T1]{fontenc}
\usepackage[utf8]{inputenc}
\usepackage{mathrsfs}
\usepackage{xcolor}
\usepackage{hyperref}
\usepackage[title]{appendix}
\hypersetup{
   colorlinks,
   citecolor=olive,
   filecolor=black,
   linkcolor=blue,
   urlcolor=blue
 }%
\usepackage{mathtools,halloweenmath}
\usepackage{mathbbol}
\theoremstyle{plain}
\newtheorem{propo}{Proposition}[section]
\newtheorem{lema}[propo]{Lemma}
\theoremstyle{plain}
\newtheorem{teorema}[propo]{Theorem}
\newtheorem{cor}[propo]{Corollary}
\newtheorem{fact}[propo]{Fact}
\newtheorem{claim}[propo]{Claim}

\theoremstyle{definition}

\newtheorem{defin}[propo]{Definition}

\newtheorem{remark}[propo]{Remark}
\newtheorem{condition}[propo]{Condition}

\newcommand{\ol}[1]{\overline{#1}}
\newcommand{\id}{id}

\renewcommand{\.}{\ldots}

\newcommand{\dl}{ \Big( \kern-0.6em{ \Big( }}
\newcommand{\dr}{ \Big) \kern-0.6em{ \Big) }}
\newcommand{\hs}[2]{#1\hspace{0.03cm}\dl\hspace{0.05cm}t^{#2}\hspace{0.05cm}\dr}

\newcommand{\sqdl}{ \Big[ \kern-0.55em{ \Big[ }}
\newcommand{\sqdr}{ \Big] \kern-0.55em{ \Big] }}

\def\aa{\textbf{a}}
\def\bb{\textbf{b}}

\DeclareMathOperator{\tp}{tp} % TYPE
\DeclareMathOperator{\qftp}{qftp} % QUANTIFIER-FREE TYPE
\DeclareMathOperator{\dcl}{dcl} % DEFINABLE CLOSURE
\DeclareMathOperator{\acl}{acl} % ALGEBRAIC CLOSURE
\DeclareMathOperator{\Th}{Th} % THEORY
\DeclareMathOperator{\imdeg}{imp-deg} % IMPERFECTION DEGREE
\DeclareMathOperator{\Aut}{Aut} % AUTOMORPFISM GROUP
\renewcommand{\char}{char} % CHARACTERISTIC
\DeclareMathOperator{\ac}{ac} %ANGULAR COMPONENT
\DeclareMathOperator{\res}{res} % Residue map
\DeclareMathOperator{\Gal}{Gal} % Galois group
\DeclareMathOperator{\Mon}{Mon} % Monomials
\DeclareMathOperator{\cl}{cl} % Closure operator of a pregeometry
\DeclareMathOperator{\adj}{adj} % adjucate of a matrix

\def\a{\alpha}
\def\b{\beta}
\def\g{\gamma}

\def\d{\delta}
\def\D{\Delta}
\newcommand{\e}{\varepsilon}
\newcommand{\f}{\varphi}

\renewcommand{\th}{\theta}
\renewcommand{\t}{\tau}
 
\let\lw\l
\renewcommand\l{{\lambda}}
\def\k{\kappa}
\def\s{\sigma}
\def\S{\Sigma}

\def\E{\exists} 
\def\A{\forall}

\newcommand{\mirrorexists}{E}
\newcommand{\sii}{\longleftrightarrow}
\newcommand{\Sii}{\Longleftrightarrow}
\def\N{\mathbb{N}} 
\def\Z{\mathbb{Z}}

\def\F{\mathbb{F}}
\def\GG{\mathbb{G}}
\def\KK{\mathbb{K}} 
\def\kk{\mathbb{k}}
\def\so{\mathbb{S}}

\def\AA{\mathfrak{A}} % model
\def\BB{\mathfrak{B}} % model
\def\MM{\mathfrak{M}} % model
\def\NN{\mathfrak{N}} % model
\def\LL{\mathscr{L}} % language
\def\UU{\mathfrak{U}} % monster model
\def\p{\mathfrak{p}} % global type
\def\O{\mathcal{O}} % valuation ring
\def\M{\mathscr{M}} % max. ideal of valuation ring
\def\HH{\mathbb{H}}

\def\FF{\mathfrak{F}}

\def\Ind#1#2{#1\setbox0=\hbox{$#1x$}\kern\wd0\hbox to 0pt{\hss$#1\mid$\hss}
\lower.9\ht0\hbox to 0pt{\hss$#1\smile$\hss}\kern\wd0}

\def\Notind#1#2{#1\setbox0=\hbox{$#1x$}\kern\wd0\hbox to 0pt{\mathchardef
\nn=12854\hss$#1\nn$\kern1.4\wd0\hss}\hbox to
0pt{\hss$#1\mid$\hss}\lower.9\ht0 \hbox to
0pt{\hss$#1\smile$\hss}\kern\wd0}

\usepackage[utf8]{inputenc}
\usepackage{fancyhdr}
\usepackage{csquotes}
\usepackage{faktor}
\usepackage{enumitem}
\usepackage{dirtytalk}

\usepackage[english]{babel}
\newtheoremstyle{named}{}{}{}{}{}{}{.5em}{\normalfont{\thmnumber{#2}} \itshape\thmnote{#3.}}
\theoremstyle{named}
\newtheoremstyle{asd}{}{}{}{}{\scshape}{}{.5em}{\scshape{\thmname{#1}} \normalfont{\thmnumber{#2}}}
\theoremstyle{asd}
\fancypagestyle{mypagestyle}{%

\lhead{$\,$}
\rhead{\Cross}
}
\fancypagestyle{plain}{
\fancyfoot{}
\fancyfoot[LE,RO]{\thepage}
}
\usepackage{titlesec}
\titleformat{\section}
  {\centering\huge\scshape}{\scshape\thesection}{1em}{}
\titleformat{\subsection}
  {\Large\scshape}{\scshape\thesubsection}{1em}{}
\allowdisplaybreaks
\usepackage{caption}

\title[RUNNING TITLE]{\Large\rm On some NIP Fragments of Fields}
\author{PAULO ANDRÉS SOTO MORENO}
\address{Universit\'{e} Paris Cit\'{e}, Sorbonne Universit\'{e}, CNRS, IMJ-PRG, F-75013 Paris, France}
\email{paulo.soto@imj-prg.fr}

\begin{document}

\begin{abstract}
   In this document we study sets of NIP formulas in some theories of fields and valued fields, with a special focus on the sets of quantifier-free and existential formulas. First, we give a new proof of the fact that Separably Closed Valued Fields of any characteristic and any imperfection degree are NIP, and use this result to fill some gaps of a proof of the so-called NIP Transfer Theorem for henselian valued fields of equal characteristic. Second, we prove a variant of a theorem of Johnson: every positive characteristic valued field whose existential formulas are NIP is henselian, and generalize this result for finer sets of existential formulas, like positive existential with a given number of quantifiers. Finally, we set the ground for the finer question of transfer of NIP formulas of valued fields with bounded quantifier rank. Namely, we prove that for any henselian equicharacteristic valued field, any formula of quantifier rank at most $n\geq 1$ is NIP if and only if the same is true for the residue field and the value group, provided that the valued field is separably defectless Kaplansky and conditional on a multi-variable generalization of a well known statement about indiscernible sequences of singletons in ac-valued fields. 
\end{abstract}
\maketitle

\section{Introduction}

In this document we study the interplay between the \emph{independence property} and the \emph{quantifier complexity} of definable sets in some theories of valued fields of equal characteristic. These subjects have played central roles in the model theory of valued fields for the last years, with outstanding results covering, on the one hand, a classification of henselian valued fields without the independence property, cf.~\cite{js} and \cite{aj}, and on the other hand, a thorough analysis of quantifier complexity of definable henselian valuation rings, cf.~\cite{jf} for a survey. The goal of this document is to present refined versions of a number of well known results of the theory of NIP henselian valued fields of equal characteristic, improving on the quantifier rank of formulas without the independence property, thus setting the ground for the study of the aforementioned interplay.   
Let $\LL$ be a language and let $T$ be an $\LL$-theory. If $\f(x,y)$ is an $\LL$-formula, we say that $\f(x,y)$ has the \emph{independence property} (with respect to $T$) or that $\f(x,y)$ \emph{has IP} if there is some model $\MM$ of $T$ and some sequences $(a_i:i<\omega), (b_{S}:S\subseteq\omega)$ of $|x|$- and $|y|$-tuples of $M$ respectively such that $\MM\models\f(a_i,b_S)$ if and only if $i\in S.$ Equivalently, $\f(x,y)$ has IP if there is some model $\MM$ of $T$, and indiscernible sequence $(a_i:i<\omega)$ of $|x|$-tuples of $M$ and a $|y|$-tuple $b$ of parameters from $M$ such that $\MM\models\f(a_i,b)$ if and only if $i$ is even, cf.~\cite[Lemma 2.7]{simon}. A formula is \emph{NIP} if it does not have IP, and an $\LL$-theory is called \emph{NIP} if all $\LL$-formulas $\f(x,y)$ are NIP with respect to $T.$ Also, \cite[Lemma 2.5]{simon} states that a formula $\f(x,y)$ is NIP if and only if $\f^{op}(y,x)=\f(x,y)$ is NIP, which allows to interchange the position of the indiscernible sequence and the parameters. For more on NIP theories, see e.g.~\cite{simon}. If $(K,v)$ is a valued field, we will usually denote by $vK$ and $Kv$ its value group and residue field, $\O=\O_v$ will denote its valuation ring and $\M=\M_v$ its maximal ideal. Finally, since henselian NIP valued fields of equal characteristic are separable-algebraically maximal Kaplansky, we set this class of valued fields as our underlying object of study. 
Our main theorems are about families of first order formulas of fields and valued fields, including quantifier free formulas (Theorem \ref{t1}), existential formulas (Theorem \ref{t2} and Theorem \ref{t3}) and formulas with a given quantifier rank (Theorem \ref{t4}). These theorems are the following. 

\medskip
\emph{\underline{Quantifier-free Formulas}}

\begin{teorema}[Theorem \ref{lqfnip}, cf.~{\cite[Corollary 5.2.13]{ht}} for $e\in\N$ and Delon for $e=\infty$]
\label{t1}
    Quantifier-free formulas of the language of valued fields $\LL_{div}$ expanded by the parameterized $\l$-functions are NIP. As a consequence, the theory $\texttt{SCVF}_e$ of Separably Closed Valued Fields of Ershov degree $e\in\N\cup\{\infty\}$ is NIP.  
\end{teorema}

Theorem \ref{t1} is well stated in the literature for $e\in\N,$ see e.g.~\cite[Corollary 5.2.13]{ht}. The case $e=\infty$ is attributed to Delon, who identified 1-types in separably closed fields of infinite imperfection degree, in a language including the parameterized $\l$-functions, as seen in Section 4 of \cite{delscvf}. This allows to use the so-called \emph{Poizat's coheir counting} argument to conclude that $\texttt{SCVF}_{\infty}$ is NIP, as she did in \cite{delon} for henselian valued fields of residue characteristic 0. However, to the best of our knowledge, this result remains unpublished. 

We use Theorem \ref{t1} to revisit Anscombe-Jahnke and Jahnke-Simon NIP transfer principles for henselian valued fields of equal characteristic, Corollary \ref{fulltrans}. In particular, thanks to a suggestion from Itay Kaplan and Martin Bays, we adapt Jahnke-Simon's \textbf{SE}-\textbf{Im} strategy in Proposition \ref{niptransfer} to work in a general model-theoretic context, and we then apply this result to the theory of separable-algebraically maximal Kaplansky fields, in a language of valued fields expanded by symbols for the parameterized $\l$-functions and for an angular component map. 

\medskip
\emph{\underline{Existential Formulas}}

\begin{teorema}[Theorem \ref{eniphens}, cf.~{\cite[Theorem 2.8]{dp1a}}]
\label{t2}
    Let $(K,\O)$ be a valued field of positive characteristic whose existential formulas are NIP. Then $\O$ is henselian. 
\end{teorema}

For a given language $\LL$ and a natural number $n\geq 1,$ denote by $\E_+^{n}$ the set of positive existential formulas, in prenex normal form, that have at most $n$ existential quantifiers, see Definition \ref{deffr}. If $l\geq1,$ we say that a valued field $(K,\O)$ is \emph{$l$-henselian} if $\O$ extends uniquely to any separable extension of $K$ of degree at most $l.$ The following theorem lists a number of refinements of Theorem \ref{t2}. 

\begin{teorema}[Theorem \ref{finereniphens}]
\label{t3}
    Let $d$ and $l$ be some positive integers and let $n\in\N$. Let $\MM$ be some sufficiently saturated $\LL$-structure with universe $M$ and let $(K,\O)$ be a type-definable valued field of characteristic $p>0$, with $K\subseteq M^{d}.$
    \begin{enumerate}[label*={\arabic*.}]
        \item If $\MM$ is NIP, then $\O$ is henselian.

        \item If positive existential formulas are NIP in $\MM$ and $(K,\O)$ is type-definable with formulas that are positive existential, then $\O$ is henselian.

        \item Let $m=\max\{6l,p+1\}.$ If all formulas from $\E_{+}^{n+dm}$ are NIP in $\MM$ and $(K,\O)$ is type-definable with formulas from $\E_{+}^{n+dm}$, then $\O$ is $l$-henselian.
    \end{enumerate}
\end{teorema}

\emph{\underline{Formulas of a given Quantifier Rank}}

\medskip
Additionally, let $\mirrorexists_{n}$ be the set of $\LL$-formulas in prenex normal form, of quantifier rank at most $n$, whose block of quantifiers starts with an existential quantifier, cf.~Definition \ref{deffr}. Below, Corollary \ref{enniptr} states a resplendent version of the next theorem.

\begin{teorema}[Theorem \ref{nonrespltr}]
\label{t4}
    Let $\LL_{3s}$ be the usual three-sorted language of valued fields, let $n\geq 1$ and let $\MM$ be the $\LL_{3s}$-structure generated by a henselian equi-characteristic valued field $(M,v)$. If $\MM$ satisfies that all of its formulas from $\mirrorexists_{n}$ are NIP, then $(M,v)$ is separably-defectless Kaplansky and all $\LL_{ring}$-formulas from $\mirrorexists_{n}$ are NIP with respect to the theory of $Mv$. The converse implication holds if Condition \ref{acv} is satisfied.
\end{teorema}

This document is organized as follows. In Section \ref{2}, we recall some languages and theories that allow to axiomatize the class $\mathrm{SAMK}$ of separable-algebraically maximal Kaplansky fields of a given Ershov degree $e\in\N\cup\{\infty\}$. We introduce the \emph{parameterized $\l$-functions} and explain their role of detecting separable extensions of fields. Additionally, we introduce \emph{fragments} of an arbitrary language. We define the independence property \emph{with respect to} a given fragment, and give an account of some model theoretic results of NIP theories that hold in our setting, such as the existence of type-connected components in Proposition \ref{g00}.
We finish this section by presenting a general model theoretic NIP transfer principle, Proposition \ref{niptransfer}, building on Jahnke-Simon's NIP transfer for valued fields, where we restate the conditions \textbf{SE} and \textbf{Im} for our general model theoretic setting, cf.~Definition \ref{sideconds}.
In Section \ref{sectlres} we focus on the first of the fragments we will study: the fragment of quantifier-free formulas. In a language of valued fields expanded by the parameterized $\l$-functions, we prove that any such formula is NIP. Given that the theory $\texttt{SCVF}_e$ of separably closed valued fields of Ershov degree $e\in\N\cup\{\infty\}$ eliminates quantifiers in this language, cf.~\cite[Theorem 4.12]{h}, we conclude that it is a NIP theory, completing the proof of our first main theorem.  
As an application, we revisit in Section \ref{4} the so-called \emph{NIP transfer principle} for $\mathrm{SAMK}$ fields: if the residue field and the value group of any such valued field are NIP in any expansion of their languages, then the valued field in question is NIP, cf.~Corollary \ref{fulltrans}. Although the original statement of the theorem is correct, cf.~\cite[Proposition 4.1]{aj} for $e\in\N\cup\{\infty\}$ and \cite[Theorem 3.3]{js} for $e\in\N$, building on \cite[Lemma 3.2]{js}, we found that the proofs given in \cite[Proposition 4.1]{aj} and \cite[Lemma 3.2]{js} are not accurate as stated, and that a small but consequential modification of the proofs is required. For a short explanation of the gaps in these proofs, see Remark \ref{gaps}. We take the opportunity to record a corrected version in Corollary \ref{fulltrans} and notably in Lemma \ref{e+im}, to supply the missing details in a self-contained manner. 
In Section \ref{5} we revisit some of the most fundamental theorems of NIP fields, and see that the quantifier rank of the formulas witnessing IP in several of these theorems is one. Most notably, we improve on a result of Johnson \cite[Theorem 2.8]{dp1a} and prove that any valued field of positive characteristic \emph{whose existential formulas are NIP} is necessarily henselian, Theorem \ref{eniphens}, completing our second main theorem. In Theorem \ref{finereniphens}, we explore some variants of this result for positive existential formulas, of a given number of existential quantifiers, for valued fields that are type-definable in structures where these formulas are NIP. 
Finally, in Section \ref{6}, we address the question of whether theories of henselian valued fields of equal characteristic \emph{whose formulas of quantifier rank at most $n\in\N$ are NIP} can be characterized in a similar fashion as their NIP counterparts. We prove in Proposition \ref{enequiv} a (relative, resplendent) quantifier elimination principle up to formulas of quantifier rank at most $n\in\N,$ deduce the corresponding Ax-Kochen-Ershov principles in Corollary \ref{akesen}, and use these results to give a positive answer of the main question in Proposition \ref{enniptr}, conditional to a multi-variable generalization of a well known statement about indiscernible sequences of singletons in ac-valued fields, cf.~Condition \ref{acv}.

\section{Preliminaries}
\label{2}

\subsection{Languages and theories of \texorpdfstring{$\mathrm{SAMK}$}{} fields}

A valued field $(K,v)$ is called \emph{separable-algebraically maximal} if it does not admit any proper separable immediate algebraic extension, and it is called \emph{Kaplansky} if $\char(Kv)=0$ or $\char(K)=p>0,$ $vK$ is $p$-divisible and $Kv$ does not admit any finite extension of degree divisible by $p.$ We will use the notations introduced in Section 2 of \cite{sm}, which gathers a more detailed introduction to separable-algebraically maximal Kaplansky fields. Throughout this document, this class will be denoted by $\mathrm{SAMK}$. Unless stated otherwise, we will work with three sorted languages for valued fields, with sorts $\KK$ for the home field, $\kk$ for the residue field and $\GG$ for the value group. If $\LL$ is any such language, we write $\LL(\KK)$ (resp.~$\LL(\kk),$ $\LL(\GG)$) to mean the language associated to the sort $\KK$ (resp.~$\kk$ and $\GG$), and we always assume that any such $\LL$ contains a function symbol $\underline{v}:\KK\to\GG$ to be interpreted as a valuation, a function symbol $\underline{\res}:\KK\to\kk$ to be interpreted as the associated residue map, and we also assume that $\LL(\KK)$ and $\LL(\kk)$ contain the language of rings $\LL_{ring}$ and $\LL(\GG)$ contains the language of ordered groups $\LL_{og}$ together with a symbol $\infty$ for a point at infinity. In order to stress the dependence on the auxiliary languages $\LL(\kk)$ and $\LL(\GG),$ we will even write $\LL=\LL(\LL_\kk,\,\LL_\GG)$ to mean that $\LL(\kk)=\LL_\kk$ and $\LL(\GG)=\LL_\GG$ for a given couple of languages $\LL_\kk$ and $\LL_\GG$ containing $\LL_{ring}$ and $\LL_{og}\cup\{\infty\}$ respectively. 
Let $\LL_{3s}=\LL_{3s}(\LL_\kk,\,\LL_\GG)$ be the three sorted language for valued fields, \emph{possibly expanded in $\LL_{3s}(\GG)$ and $\LL_{3s}(\kk),$} i.e.~where $\LL_\kk$ and $\LL_\GG$ are allowed to have more symbols than $\LL_{ring}$ and $\LL_{og}\cup\{\infty\}$ respectively. 
Define $\LL_{\ac}$ as the expansion of $\LL_{3s}$ by a function symbol for an angular component $\underline{\ac}:\KK\to\kk$, and let $\LL$ be the expansion of $\LL_{\ac}$ by countably many function symbols $\{\underline{\l}_{n,m}(x_0;x_1,\.,x_n):n\geq1,m\in\Mon(n)\}$ where each $\underline{\l}_{n,m}$ is of sort $\KK^{n+1}\to\KK.$ 

As noted in Section 2.2 of \cite{sm}, the class $\mathrm{SAMK}$ of equi-characteristic separable-algebraically maximal Kaplansky fields is elementary in any such $\LL_{3s},$ axiomatized by the following sentences:
\begin{itemize}[wide]
    \item The axioms of henselian valued fields,

    \item The axioms of extremality for separable polynomials in one variable,

    \item Equi-characteristic: $$\left\{\sum_{i=1}^{p}1_{\KK}=0\sii\sum_{i=1}^{p}1_{\kk}=0:p\text{ prime}\right\},$$
    where $1_\KK$ is the unit symbol of the home field sort $\KK$ and $1_\kk$ is the unit of the residue field sort $\kk,$
    \item Kaplansky (uniform across all positive characteristics): $$\left\{\sum_{i=1}^{p}1_{\kk}=0\to\A x\E y\left(v(x)=pv(y)\right):p\text{ prime}\right\},$$ and $$\left\{\sum_{i=1}^{p}1_{\kk}=0\to\A x_0,\.,x_n,y\,\E z\left(\sum_{i=0}^{n}x_iz^{p^{i}}=y\right):p\text{ prime}, n<\omega\right\},$$ where all the mentioned variables are of sort $\kk.$
\end{itemize}
For $e\in\N\cup\{\infty\},$ let $\texttt{SAMK}_{e}^{\l,\ac}$ be the $\LL$-theory given by the following axioms:
\begin{itemize}[wide]
    \item The $\LL_{\ac}$-theory of non-trivial $\mathrm{SAMK}$ equi-characteristic valued fields admitting an angular component,

    \item If $e\in\N,$ the $\LL_{ring}$-axioms fixing the \emph{Ershov degree} of the home sort equal to $e:$ If $x_1,\.,x_n$ are variables from the sort $\KK,$ let $\Mon(x_1,\.,x_n)=\{\prod_{i=1}^{n}x_i^{\a_i}:0\leq\a_i<p\text{ for all }i\in\{1,\.,n\}\}$ be the set of monomials obtained with these variables. Then the axioms (one for each prime $p$) say that if $\sum_{i=1}^{p}1_\KK=0,$ then there are $e$ elements $x_1,\.,x_e$ whose set of monomials $\Mon(x_1,\.,x_e)$ is linearly independent over $\KK^{p}$ and generates $\KK$ over $\KK^{p}.$

    \item If $e=\infty,$ the axioms fixing the Ershov degree of the home sort equal to $e:$ with the same notation as above, the axioms (one for each $e\in\N$ and each prime $p$) say that if $\sum_{i=1}^{p}1_\KK=0,$ then there are $e$ elements $x_1,\.,x_e$ whose set of monomials $\Mon(x_1,\.,x_e)$ is linearly independent over $\KK^{p}.$  

    \item The axioms describing the \emph{parameterized $\l$-functions:} for any $n<\omega$ and any prime $p,$ if $\sum_{i=1}^{p}1_\KK=0,$ then $$t=\sum_{m\in\Mon(n)}(\l_{n,m}(t;x))^{p}\cdot x^{m}$$ for all tuples $x=(x_1,\.,x_n)$ whose monomial set $\Mon(x_1,\.,x_n)$ is linearly independent over $\KK^{p}$ and for all $t$ in the $\KK^{p}$-span of $\Mon(x_1,\.,x_n),$ and $\l_{n,m}(t;x)=0$ otherwise. 
\end{itemize}
Note that this axiomatization allows to define each parameterized $\l$-function in $\LL_{3s}.$ 

We will use the following fact that describes the structure generated by a field under the parameterized $\l$-functions of a given extension thereof. For now, let $\LL_{\l}$ denote the language $\LL_{ring}$ expanded by symbols for the parameterized $\l$-functions. If $M|F$ is a field extension and $b$ is any tuple from $M,$ we denote by $\langle F,b\rangle$ the universe of the $\LL_{\l}$-substructure of $M$ generated by $F$ and $b.$

\begin{fact}[Cf.~{\cite[Proposition 2.7]{at}}]
\label{sylvy}
Let $M|F$ be a field extension. The following statements are equivalent.
\begin{enumerate}[label*={\arabic*.}]
    \item $M|F$ is separable.

    \item For each tuple $b$ from $F$ which is $p$-independent in $M,$ we have that $M^{p}(b)\cap F=F^{p}(b).$

    \item For each tuple $b$ from $F$ which is $p$-independent in $M,$ each monomial $m\in\Mon(|b|)$ and each $a\in F\cap M^{p}(b),$ the element $\l_m^{b}(a)$ lies in $F.$
\end{enumerate}
\end{fact}

We say that $F$ is \emph{$\l$-closed in $M$} if Statement 3 of Fact \ref{sylvy} holds. Therefore, $F$ is $\l$-closed in $M$ if and only if $M|F$ is separable, by Fact \ref{sylvy}.

\begin{fact}[Cf.~{\cite[Proposition 2.28]{at}}]
\label{lcl}
    Let $M|F$ be a separable extension of fields, seen as structures in the language $\LL_{\l}$, and let $b\in M.$ Then $\langle F,b\rangle=\Lambda_M(F(b))$ is the lambda-closure of $F(b)$ in $M$. In particular, $M|\langle F,b\rangle$ is separable. 
\end{fact}

\subsection{Fragments and NIP Formulas}

\begin{defin} Let $\LL$ be a language.
\label{deffr}
\begin{itemize}[wide]
    \item  An $\LL$-\emph{fragment} is a set of $\LL$-formulas closed under $\wedge$ and $\vee,$ containing $\top$ and $\perp.$ If $\D_0$ is a set of $\LL$-formulas, the $\LL$-\emph{fragment generated by }$\D_0$ is the intersection of all the $\LL$-fragments that contain $\D_0.$ 

    \item If $\D_0$ is a set of $\LL$-formulas, we denote by $\pm\D_0$ the $\LL$-fragment generated by $\D_0\cup\{\neg\f:\f\in\D_0\}.$ 

    \item Given a language $\LL$, we denote by $\E$ and $\A$ (resp.~$\E_{+}$ and $\A_{-}$) the sets of existential and universal (resp.~positive existential and negative universal) $\LL$-formulas.

    \item Given a language $\LL$ and some $n\geq1,$ we denote by $\E^{n}$ and $\E^{n}_{+}$ (resp.~$\A^{n}$ and $\A^{n}_{-}$) the fragment generated by existential $\LL$-formulas, in prenex normal form, with at most $n$ existential quantifiers, see \cite[Definition 3.1]{af}. 
    
    \item Given a language $\LL$ and some $n\geq1,$ we denote by $\mirrorexists_{n}$ (resp.~$A_n$) the set of $\LL$-formulas in prenex normal form, of quantifier rank $n,$ whose outermost quantifier is $\E$ (resp.~$\A$). Note that the sets $\mirrorexists_{n}$ and $A_n$ are, up to logical equivalence, closed under conjunctions and disjunctions, i.e.~they form $\LL$-fragments for any language $\LL.$ We call these sets the \emph{classical fragments}.  
    
\end{itemize}
\end{defin}

Though a fragment is not formally a set of \emph{equivalence classes} of formulas, for most of our purposes we do not make the distinction between a formula $\f$ actually belonging to a fragment $\D$ and the formula $\f$ being logically equivalent to a formula from $\D.$ Note that if $\D$ is an $\LL$-fragment closed under negations, then $\pm\D=\D.$ 

\begin{defin}
    Let $\LL,\LL'$ be two languages, and let $\AA$ and $\BB$ be some $\LL$ and $\LL'$-structures respectively. We say that $\AA$ is \emph{interpretable} in $\BB$ if there is a number $n>0$ and a partial surjective function $\Phi:B^{n}\to A$ such that, for any $\LL$-formula $\f(x_1,\.,x_l,z),$ there is an $\LL'$-formula $\psi(y_1,\.,y_l,w)$ satisfying $|y_i|=n$ for all $i\in[l]$ and, for all $\a\in A^{|z|},$ there is some $\b\in B^{|w|}$ with $$\AA\models\f(\Phi(y_1),\.,\Phi(y_l),\a)\Sii\BB\models\psi(y_1,\.,y_l,\b).$$
    We also say that $\AA$ is interpretable in $\BB$ \emph{through $(n,\Phi).$}
    If $\D$ and $\D'$ are sets of $\LL$- and $\LL'$-formulas respectively, and if $\AA$ is interpretable in $\BB$ through $(n,\Phi),$ we say that $(n,\Phi)$ is \emph{$\D$-to-$\D'$} if for any formula $\f(x,y)\in\D$ there is some formula $\psi(z,w)\in\D'$ such that, if $D=\f(A,\a)$ for some tuple $\a$ of $A,$ then $\Phi^{-1}D=\psi(B,\b)$ for some tuple $\b$ of $B.$
\end{defin}

The following Lemma is well known, see e.g.~Section 5.3 of \cite{hodges} or Section 9.4 of \cite{poizmt}. It can be proved by induction on terms and formulas.

\begin{lema}
\label{intcrit}
    Let $\LL,\LL'$ be two languages, and let $\AA$ and $\BB$ be some non-empty $\LL$ and $\LL'$-structures respectively. Then $\AA$ is interpretable in $\BB$ if there is a number $n>0$ and a partial surjection $\Phi:B^{n}\to A$ for which the following sets are $\LL'$-definable:  
    \begin{enumerate}[label*={\arabic*.}]
        \item $\Phi^{-1}A,$ say, by $\f_A(B,\b_A),$

        \item $\{(x,y)\in B^{2n}:\AA\models\Phi(x)=\Phi(y)\},$ say, by $\f_x(B,\b_x),$

        \item If $c$ is a constant symbol in $\LL,$ the set $\{x\in B^{n}:\AA\models\Phi(x)=c\},$ say, by $\f_c(B,\b_c),$
        
        \item If $f$ is an $l$-ary function symbol in $\LL$, the set $\{(x_1,\.,x_l,y)\in B^{n(l+1)}:\AA\models f(\Phi(x_1),\.,\Phi(x_l))=\Phi(y)\},$ say, by $\f_f(B,\b_f),$

        \item If $R$ is an $l$-ary relation symbol in $\LL,$ the set $\{(x_1,\.,x_l)\in B^{ln}:\AA\models R(\Phi(x_1),\.,\Phi(x_l))\},$ say, by $\f_R(B,\b_R).$
    \end{enumerate}
\end{lema}

By tracking down quantifier complexity along the inductions of the proof of Lemma \ref{intcrit}, we get the following Corollary. 

\begin{cor}
\label{eint}
    Let $\LL,\LL'$ be two languages, and let $\AA$ and $\BB$ be some $\LL$ and $\LL'$-structures respectively. Suppose the highlighted sets of Lemma \ref{intcrit} are $\LL'$-definable. 
    \begin{enumerate}[label*={\arabic*.}]
        \item If the sets of items 1, 2, 3 and 4 are quantifier-free definable, then for any $\LL$-term $t,$ the set $S_t$ is both existentially and universally definable.

        \item Suppose moreover that there is some $m\geq 1$ such that the sets of item 5 are definable by formulas both from $E_m$ and $A_m.$ Then for any $E_m$-formula (resp.~any $A_m$-formula) $\f$ from $\LL$ there is some $E_m$-formula (resp.~some $A_m$-formula) from $\LL'$ that defines the set $S_\f.$ In particular, the resulting interpretation $(n,\Phi)$ is $E_m$-to-$E_m$ and $A_m$-to-$A_m.$
    \end{enumerate}
\end{cor}

For any subset $\D$ of $\LL$-sentences, and any $\LL$-theory $T,$ consider the space $S^{\D}(T)$ of complete $\D$-types. If $\Phi\in\D,$ define the subset $D(\Phi)=\{p\in S^{\D}(T):\Phi\not\in p\}.$ The family of subsets $\{D(\Phi):\Phi\in\D\}$ determines a sub-basis of a topology on $S^{\D}(T).$ 
The following lemma should be well known, as it is an application of the fact that $S^{\D}(T)$ is a spectral topological space, cf.~\cite[Theorem 14.2.5]{ss}.

The following notations are taken from Chapter 3 of \cite{tz}.

\begin{defin}
    Let $\LL$ be a language, $\D$ be a subset of $\LL$-formulas, and let $\AA,\BB$ be two $\LL$-structures. We write $\AA\Rightarrow_\D\BB$ if $\AA\models\f$ implies that $\BB\models\f$ for every sentence $\f\in\D,$ and if $f:A\to B$ is a function, we write $f:\AA\to_\D\BB$ if $\AA\models\f(a)$ implies that $\BB\models\f(f(a))$ for all formulas $\f(x)\in\D$ and all tuples $a$ from $A,$ i.e.~if $(\AA,a)_{a\in A}\Rightarrow_{\D}(\BB,f(a))_{a\in A}.$ We write $\AA\equiv_\D\BB$ if $\AA\Rightarrow_{\D}\BB$ and $\BB\Rightarrow_{\D}\AA.$ Finally, if $\AA\subseteq\BB,$ we write $\AA\preceq_\D\BB$ if $(\AA,a)_{a\in A}\equiv_\D(\BB,a)_{a\in A}.$ 
\end{defin}

The notation $\AA\Rightarrow_\D\BB$ is not uniform throughout the literature, where sometimes instead this relation is denoted $\Th_\D(\AA)\subseteq\Th_{\D}(\BB).$
The following fact is well known. The referenced result only shows (1) is equivalent to (3), but the equivalence of (2) follows from the proof.

\begin{fact}[Cf.~{\cite[Lemma 3.1.2]{tz}}]
\label{exembg}
Let $\AA$ and $\BB$ be two $\LL$-structures, and let $\D$ be a set of $\LL$-formulas closed under existential quantifications and conjunctions. The following statements are equivalent.
\begin{enumerate}[label*={\arabic*.}]
    \item $\AA\Rightarrow_\D\BB.$

    \item For any $|A|^{+}$-saturated elementary extension $\BB^{*}$ of $\BB,$ there is some function $f:\AA\to_\D\BB^{*}.$

    \item There is some $\BB^{*}\succeq\BB$ and some function $f:\AA\to_\D\BB^{*}.$ 
\end{enumerate}
\end{fact}

\begin{fact}[cf.~{\cite[Theorem 14.2.17]{ss}}] 
\label{delredeqf}
Let $T$ be an $\LL$-theory and let $\D$ be a set of $\LL$-formulas closed under conjunctions and disjunctions. Let $\Phi,\Psi$ be subsets of $\LL$-formulas. The following are equivalent.
\begin{enumerate}[label*={\arabic*.}]
    \item There is a set $\S\subseteq\D\cup\{\bot\}$ such that $T\cup\Phi\vdash\Sigma$ and $T\cup\Sigma\vdash\Psi,$

    \item If $p,q\in S(T)$ and $p\cap\D\subseteq q,$ then $\Phi\subseteq p$ implies $\Psi\subseteq q.$
\end{enumerate}
If $\Psi$ is finite, $\Sigma$ can be chosen to be a single formula.
\end{fact}

\begin{cor}
\label{delredeq}
Let $T$ be an $\LL$-theory and let $\D'$ be a set of $\LL$-formulas closed under conjunctions and disjunctions. 
\begin{enumerate}[label={\arabic*.}]
    \item Let $\f(x)$ be an $\LL$-formula. The following statements are equivalent.
        \begin{enumerate}
            \item There is a formula $\psi(x)\in\D'\cup\{\bot\}$ such that $T\vdash\A x(\f(x)\sii\psi(x)),$
        
            \item For any pair of types $p,q\in S_x(T),$ if $p\cap\D'\subseteq q$ and $\f(x)\in p,$ then $\f(x)\in q.$
        \end{enumerate}
    \item Let $\D$ be a set of $\LL$-formulas. The following statements are equivalent.
        \begin{enumerate}
            \item Any formula from $\D$ is equivalent, modulo $T,$ to a formula from $\D'\cup\{\bot\}.$

            \item For any finite tuple of variables $x$ and any pair of types $p,q\in S_x(T),$ if $p\cap\D'\subseteq q,$ then $p\cap\D\subseteq q.$
            
            \item For any pair of models $\MM,\NN$ of $T$ and any pair of tuples $a\in M,b\in N$ of the same length, if $(\MM,a)\Rightarrow_{\D'}(\NN,b),$ then $(\MM,a)\Rightarrow_{\D}(\NN,b).$
        \end{enumerate}
\end{enumerate}
\end{cor}
\begin{proof}
Statement 1 follows from Fact \ref{delredeqf}. The equivalence between (a) and (b) of statement 2 follows from statement 1. It is clear that (a) implies (c), so we just have to prove that (c) implies (b). To this end, let $x$ be a finite tuple of variables, and let $p,q\in S_x(T)$ be two types. Then there are two models $\MM,\NN$ of $T$ and two $|x|$-tuples $a,b$ of $M$ and $N$ respectively such that $p=\tp^{\MM}(a)$ and $q=\tp^{\NN}(b).$ If $p\cap\D'\subseteq q,$ then $(\MM,a)\Rightarrow_{\D'}(\NN,b),$ which by hypothesis implies that $(\MM,a)\Rightarrow_{\D}(\NN,b),$ meaning that $p\cap\D\subseteq q.$ 
\end{proof}

\begin{defin}
    Let $\D$ be a set of $\LL$-formulas. We say that an $\LL$-theory is \emph{$\D$-NIP} if all formulas from $\D$ are NIP. If $\AA$ is some structure, we say that $\AA$ is $\D$-NIP if $\Th(\AA)$ is $\D$-NIP.
\end{defin}

Recall that an $\LL$-theory is NIP if and only if all $\LL$-formulas $\f(x,y)$ with $|x|=1$ are NIP, cf.~\cite[Proposition 2.11]{simon}. We were unable to prove the analogous statement for arbitrary sets of formulas $\D,$ and this difference stands as one of the main difficulties when dealing with $\D$-NIP theories. For example, in Section \ref{enniptr}, given a $\mathrm{SAMK}$ valued field, we prove an NIP transfer of certain families of formulas from the residue field and the value group to the valued field, conditional on a statement that holds for indiscernible sequences \emph{of several variables} instead of indiscernible sequences of singletons, as a consequence of the inability to bound to one the number of variables involved in formulas without the independence property. 

We record the following statement, whose proof is used in the case where \emph{all} formulas with in one variable (and several parameters) are NIP.

\begin{fact}[cf.~{\cite[Proposition 2.11]{simon}}]
    \label{indisctoeindisc}
    Let $T$ be a NIP $\LL$-theory. Then for every indiscernible sequence $(a_i:i<|T|^{+})$ and every finite tuple of parameters $\mathbf{b}$ of a model $\MM$ of $T,$ there is some $\a<|T|^{+}$ such that $(a_i:\a<i<|T|^{+})$ is indiscernible over $\mathbf{b}.$
\end{fact}

Note also that if a theory is $\D$-NIP, then it is $\pm\D$-NIP too, because the set of NIP formulas is closed under boolean combinations, cf.~\cite[Lemma 2.9]{simon}. One standard way of transferring NIP from a set of formulas $\D$ in one structure to another set of formulas $\D'$ in a second structure is done via \emph{$\D$-to-$\D'$ interpretations}, as the following Lemma explains.  

\begin{lema}
\label{dtod'nip}
    Let $\LL$ and $\LL'$ be two languages, let $\MM$ be an $\LL$-structure and $\NN$ be an $\LL'$-structure. Suppose that $\MM$ is $\D$-to-$\D'$ interpretable in $\NN.$
    % through $(n,\Phi), and that $(n,\Phi)$ is $\D$-to-$\D'.$ 
    Then $\MM$ is $\D$-NIP whenever $\NN$ is $\D'$-NIP. 
\end{lema}
\begin{proof}
    If a formula $\f$ from $\D$ has IP in some elementary extension of $\MM,$ then it has IP in some ultrapower $\MM^{U}$ of $\MM.$ A calculation shows that $\MM^{U}$ is $\D$-to-$\D'$ interpretable in $\NN^{U},$ so the interpretation of $\f$ is a $\D'$-formula witnessing IP in $\NN^{U}.$ This implies that $\NN$ is not $\D'$-NIP.  
\end{proof}

One can also transfer $\D$-NIP from one theory to another via \emph{$\D$-closed} embeddings. Namely,

\begin{lema}
\label{dclosedip}
    Let $\D$ be a set of formulas of a language $\LL$ and let $\AA\subseteq\BB$ be two $\LL$-structures such that $\AA\preceq_{\D}\BB.$ If $\BB$ is $\D$-NIP, then so is $\AA.$ 
\end{lema}
\begin{proof}
    Suppose $\AA$ is not $\D$-NIP. As in the proof of Lemma \ref{dtod'nip}, there is a formula $\f$ from $\D$ and some sequences $(a_i:i<\omega)$ and $(b_S:S\subseteq \omega)$ that witness the independence property in some ultrapower $\AA^{U}$ of $\AA.$ Since $\AA\preceq_\D\BB,$ one can see by \L oś' Theorem that $\AA^{U}\preceq_{\D}\BB^{U}$,
    which yields the equivalence 
    $i\in S\Sii\AA^{U}\models\f(a_i,b_S)\Sii\BB^{U}\models\f(a_i,b_S),$
    for all $i<\omega$,
    from which it follows that the theory of $\BB$ is not $\D$-NIP.
\end{proof}

Now we turn our attention to type-definable sets.

\begin{defin}
    Let $\MM$ be an $\LL$-structure and let $A\subseteq M.$ We say that a subset $G$ of a finite power of $M$ is \emph{$A$-$\D$-type-definable} if it is the trace in $\MM$ of a partial type given by $\D$-formulas over $A.$ We say that the group $(G,\cdot)$ \emph{$A$-$\D$-type-definable} if $G$ is $A$-$\D$-type-definable and the multiplication $\cdot:G\times G\to G$ is $A$-$\D$-definable. 
\end{defin}

We will fill in the details of \cite[Theorem 8.4]{simon} in order to highlight the role that NIP formulas play in the existence of $G^{00}$. We work within a monster model $\UU$ of cardinality $\k$ of an $\LL$-theory $T.$ We will mostly need $\k$ to be a strong limit cardinal of cofinality greater than $\beth_{(2^{|T|})^{+}}.$

\begin{remark}
\label{ind}
    Let $G$ be a group and let $H$ be a subgroup of $G.$ Then $|\{H'\subseteq G:H\leq H'\leq G\}|\leq 2^{[G:H]}.$
\end{remark}

The following is a well known Baldwin-Saxl-like result for \emph{sequences of uniformly $\D$-type definable subgroups} in $\D$-NIP theories, due to Kaplan and Shelah. We state a proof to highlight the role that NIP formulas play in the argument.

\begin{lema}[cf.~{\cite[Lemma 3.2]{ksh}}]
\label{indgroups}
    Let $\D$ be a set of $\LL$-formulas closed under conjunctions, let $\Phi(x,\mathbf{y})$ be a set of $\D$-formulas and let $(b_i:i<|T|^{+})$ be a sequence of tuples of length $|\mathbf{y}|$ such that $(\Phi(\UU,b_i):i<|T|^{+})$ defines a sequence of pairwise different subgroups of a $\D$-type definable group $G$ of a $\D$-NIP structure. Then there is some $i<|T|^{+}$ for which $$\bigcap_{j<|T|^{+}}\Phi(\UU,b_j)=\bigcap_{\substack{j<|T|^{+}\\j\neq i}}\Phi(\UU,b_j).$$ 
\end{lema}
\begin{proof}
    Suppose otherwise, and define $H_i=\Phi(\UU,b_i)$ for each $i<|T|^{+}.$ For any such $i,$ choose some $a_i\in\bigcap_{j\neq i}H_j\setminus H_i.$ Then there is some $\f_i\in\Phi(x,\mathbf{y})$ such that $\models\neg\f_i(a_i,b_i).$ The function sending $i$ to $\f_i$ has to have an infinite fiber, for $|T|^{+}>|\Phi(x,\mathbf{y})|.$ Thus, there is a $\D$-formula $\f\in\Phi(x,\mathbf{y})$ and a subsequence $(a_i,b_i:i<\omega)$ of $(a_i,b_i:i<|T|^{+})$ such that $\models\neg\f(a_i,b_i)$ for all $i<\omega.$ Also, $a_i\in H_j$ for all $i\neq j,$ so we have that $\models\f(a_i,b_j)\Sii i\neq j.$
We can now find a conjunction $\th$ of formulas from $\Phi(x,\mathbf{y})$ such that 
\begin{equation}
\label{mult}
    \models\A x_1,x_2,x_3\left(\left(\bigwedge_{k=1}^{3}\th(x_k,b_i)\right)\to\f(x_1x_2x_3,b_i)\right)  
\end{equation}
for all $i<\omega:$ if equation \ref{mult} does not hold for any finite conjunction $\th$ of formulas from  $\Phi(x,\mathbf{y}),$ by saturation of $\UU,$ the set $H_i=\Phi(\UU,b_i)$ would not be closed under the group operation. 
Finally, given a finite subset $I=\{i_1,\.,i_n\}$ of $\omega,$ define $a_I=a_{i_1}\.a_{i_n}.$ Then $\models\th(a_I,b_j)\Sii j\not\in I$ for any such $I$ and any $j<\omega.$
Indeed, if $j\not\in I,$ then $a_{i_1},\.,a_{i_n}$ and $a_I$ are elements of $H_j=\Phi(x,b_j),$ so $\models\th(a_I,b_j).$ If $j\in I,$ then $a_I=ca_jd$ for some $c,d\in H_j,$ so $a_j=c^{-1}a_Id^{-1}.$ Therefore $\models\neg\f(c^{-1}a_Id^{-1},b_j),$ so $\models\neg\th(c^{-1},b_j)\vee\neg\th(a_I,b_j)\vee\neg\th(d^{-1},b_j).$ But $c^{-1},d^{-1}\in H_j=\Phi(x,b_j),$ so $\models\th(c^{-1},b_j)\wedge\th(d^{-1},b_j)$ and thus $\models\neg\th(a_I,b_j).$
We get that the VC-dimension of $\neg\th^{op}$ is at least $n$ for any $n<\omega,$ so it has to be infinite, contradicting NIP, and $\th$ is a $\D$-formula as it is a conjunction of formulas from $\Phi(x,\mathbf{y}).$
\end{proof}

Extending a given countable indiscernible sequence to a sequence of length $|T|^{+},$ then Ramsey and saturation yield the following corollary.

\begin{cor}[cf.~{\cite[Corollary 3.3]{ksh}}]
\label{kapsh}
    Let $\D$ be a set of $\LL$-formulas closed under conjunctions, let $\Phi(x,\mathbf{y})$ be a set of $\D$-formulas and let $(b_i:i\in\Z)$ be an indiscernible sequence of $|\mathbf{y}|$-tuples for which $(\Phi(\UU,b_i):i\in\Z)$ defines a sequence of pairwise different subgroups of a $\D$-type definable group $G$ of a $\D$-NIP structure. Then $$\bigcap_{i\neq0}\Phi(\UU,b_i)=\bigcap_{i\in\Z}\Phi(\UU,b_i).$$ 
\end{cor}

The following Lemma, which follows by compactness, corresponds to the first remark of Section 8.1 of \cite{simon}.

\begin{lema}
\label{count}
    Suppose that $\D$ is a set of $\LL$-formulas closed under conjunctions. Let $A$ be a small subset of $\UU,$ let $G$ be an $A$-$\D$-type-definable group, and let $\mathfrak{F}$ be the family of $A$-$\D$-type-definable groups $H$ such that $G\leq H$ and $H$ is defined by an intersection of countably many formulas from $\D$. Then $G=\bigcap\mathfrak{F}.$
\end{lema}

For the rest of this section, let $\D$ be a fragment and let $G$ be a fix $\emptyset$-$\D$-type-definable group. Recall that if $A$ is a small set of parameters, then an $A$-type-definable subgroup $H\leq G$ is said to have \emph{bounded index} if $|G/H|<\k.$ By \cite[Proposition 5.1]{simon}, as the equivalence relation defined by $H$ is $\Aut(\UU/A)$-invariant, $H$ has bounded index in $G$ if and only if $|G/H|\leq2^{|A|+|T|}.$ Since $\k$ is a strong limit, then this equivalence shows that an intersection of \emph{boundedly-many} (i.e.~less than $\k$) subgroups of bounded index is of bounded index itself.

\begin{defin}
For a small set of parameters $A,$ let $G_A^{00,\D}$ be the intersection of all $A$-$\D$-type-definable subgroups of $G$ of bounded index. It is itself $A$-$\D$-type-definable of bounded index. We say that $G^{00,\D}$ \emph{exists} if $G_A^{00,\D}=G_{\emptyset}^{00,\D}$ for all small sets $A.$ If $G^{00,\D}$ exists, we set $G^{00,\D}=G_{\emptyset}^{00,\D}.$ 
\end{defin}

Note that for any small set of parameters $A$, the inclusion $G_A^{00,\D}\subseteq G_{\emptyset}^{00,\D}$ always holds.

\begin{lema}
\label{aut}
Let $\sigma\in\Aut(\UU).$ Then $\sigma G_A^{00,\D}=G_{\sigma A}^{00,\D}.$
\end{lema}
\begin{proof}
    Note that $\sigma G=G$ because $G$ is $\emptyset$-type-definable.
    Since $G_A^{00,\D}$ is $A$-$\D$-type-definable of bounded index in $G$, we have that $\sigma G_A^{00}$ is $\sigma A$-$\D$-type-definable of bounded index in $\sigma G=G$ and thus $G_{\sigma A}^{00,\D}\subseteq \sigma G_A^{00,\D}.$ The same argument shows that $\sigma^{-1}G_{\sigma A}^{00,\D}$ is an $A$-$\D$-type-definable subgroup of $G$ of bounded index in $G,$ so $G_A^{00,\D}\subseteq\sigma^{-1}G_{\sigma A}^{00,\D}$ and therefore $\sigma G_A^{00,\D}\subseteq G_{\sigma A}^{00,\D}.$
\end{proof}

\begin{lema}
\label{paramtoemp}
    Let $A\supseteq B$ be small subsets of $\UU$ and let $X$ be an $A$-$\D$-type-definable subset of $\UU.$ If $X$ is $\Aut(\UU/B)$-invariant, then it is $B$-$\pm\D$-type-definable.
\end{lema}
\begin{proof}
    Denote by $S^{\pm\D}(C)$ the space of $\pm\D$-types over $C$. When endowed with its natural topology, it is compact, Hausdorff and totally disconnected, because $\pm\D$ is closed under negations. It follows that the restriction map $f:S^{\pm\D}(A)\to S^{\pm\D}(B)$ is closed, being a continuous map from a compact space to a Hausdorff space. If $X$ is defined by a closed subset $[\Theta]=\{p\in S^{\pm\D}(A):p\supseteq\Theta\}$ of $S^{\pm\D}(A),$ then $f([\Theta])$ is closed in $S^{\pm\D}(B),$ so by compactness there is some set of $\pm\D$-formulas $\Sigma$ over $B$ such that $f([\Theta])=\{p\in S^{\pm\D}(B):p\supseteq\Sigma\}.$ By $\Aut(\UU/B)$-invariance, $X=\bigcap_{\f\in\Sigma}\f(\UU).$ 
\end{proof}

\begin{propo}[Cf.~{\cite[Theorem 8.4]{simon}}]
\label{g00}
Suppose $\D$ is an $\LL$-fragment closed under negations, and let $T$ be a $\D$-NIP $\LL$-theory. Then $G^{00,\D}$ exists.
\end{propo}

\begin{proof}
We will split the proof of the Proposition into two statements. First, if $G^{00,\D}$ does not exist, we will see that there is an indiscernible sequence of uniformly $\D$-type definable of subgroups of $G$ \emph{of bounded index}. Second, we will prove that in this case, there is a $\D$-formula with IP using the Baldwin-Saxl-like result of Lemma \ref{indgroups}.
\begin{claim}
    Suppose that $G^{00,\D}$ does not exist. Then there is an indiscernible sequence of $\D$-type-definable subgroups of bounded index. 
\end{claim}
\begin{proof}[Proof of Claim]
If $G^{00,\D}$ does not exist, then there is some small set $A$ for which $G_A^{00,\D}\subsetneq G_\emptyset^{00,\D}.$ Then $\bigcap_{|B|=|A|}G_B^{00,\D}$ is a $\D$-type-definable subgroup. Lemma \ref{aut} shows that this subgroup is $\Aut(\UU)$-invariant, hence $\emptyset$-$\D$-type-definable by Lemma \ref{paramtoemp} and because $\pm\D=\D.$ It is not of bounded index, otherwise $$G_\emptyset^{00,\D}\subseteq\bigcap_{|B|=|A|}G_B^{00,\D}\subseteq G_A^{00,\D}\subsetneq G_\emptyset^{00,\D}.$$

For any set of parameters $B$ of cardinality $|A|,$ let $\mathfrak{F}_B$ be the family corresponding to $G_B^{00,\D}$ as in Lemma \ref{count}. Then $\bigcup_{|B|=|A|}\mathfrak{F}_B$ has to have at least $\k$ elements, otherwise the group $$\bigcap_{|B|=|A|}G_B^{00,\D}=\bigcap_{|B|=|A|}\bigcap\mathfrak{F}_B$$ would have bounded index in $G,$ as it would be an intersection of boundedly many subgroups of bounded index. This allows us to choose a sequence $(H_i:i<\k)$ of pairwise different elements of $\bigcup_{|B|=|A|}\mathfrak{F}_B,$ where for any $i<\k,$ $H_i=\Phi_i(\UU,d_i)$ is a $d_i$-$\D$-type-definable subgroup of $G$ of bounded index, $d_i$ is a countable tuple of elements of $\UU$ and $\Phi_i(x,\mathbf{y})$ is some countable set of $\D$-formulas. 
Since there are at most $2^{|T|}<\k$ sets of $\D$-formulas from $\LL$, by the pigeonhole principle we may assume that there is some set $\Phi(x,\mathbf{y})$ such that $H_i=\Phi(\UU,d_i)$ for all $i<\k$ and the sequence $(H_i:i<\k)$ is made of pairwise different subgroups of $G$ of bounded index.

Since $\k>\beth_{(2^{|T|})^{+}}$, by Erd\H{o}s–Rado and saturation, there is an indiscernible sequence $(c_i:i<\omega)$ in $\UU$ such that for any $n<\omega$ there are some $i_0<\.<i_n<\k$ such that $c_0\.c_n\equiv d_{i_0}\.d_{i_n},$ cf.~\cite[Proposition 1.1]{simon}.
Use Ramsey and saturation of $\UU$ to find an indiscernible sequence $(b_i:i<|T|^{+})$ realizing the EM-type of $(c_i:i<\omega).$ In particular, for a fix $j<|T|^{+},$ $b_j\equiv b_0\equiv c_0\equiv d_{i_0}$ implies that $\Phi(\UU,b_j)$ is an $\Aut(\UU)$-conjugate of $\Phi(\UU,d_{i_0}),$ so it defines itself a subgroup of $G$ of bounded index. Similarly, if $j_1<j_2<|T|^{+},$ then $b_{j_1}b_{j_2}\equiv b_0b_1\equiv c_0c_1\equiv d_{i_1}d_{i_2}$ and $\Phi(\UU,d_{i_1})\neq\Phi(\UU,d_{i_2})$ implies that $\Phi(\UU,b_{j_1})\neq\Phi(\UU,b_{j_2}).$ 
\qedhere{$_\text{Claim}$}
\end{proof}
This allows us to abuse notation by writing $H_i=\Phi(\UU,b_i)$ for each $i<|T|^{+}.$ Note that for any $i<|T|^{+},$ the index $\left[G:\bigcap_{j\neq i}H_j\right]$ is bounded, for $\bigcap_{i\neq j<|T|^{+}}H_j$ is an intersection of boundedly many subgroups of bounded index.   
   
By Lemma \ref{indgroups}, there is some $i_0<|T|^{+}$ for which $H_{i_0}\supseteq\bigcap_{i_0\neq j<|T|^{+}}H_j.$ Let $\eta=\left[G:\bigcap_{j\neq i_0}H_j\right]$ and let $\l=(2^{\eta})^{+}<\k.$ Find an indiscernible sequence $(c_i:i<(i_0-1)+\l+(|T|^{+}-i_0))$ realizing the EM-type of $(b_i:i<|T|^{+}).$ Put $H_i'=\Phi(x,c_i).$ Then $(H_i':i<(i_0-1)+\l+(|T|^{+}-i_0))$ is again a sequence of pairwise different $\D$-type-definable subgroups of $G$ of bounded index.
In particular $[G:H_j']=[G:H_j]$ for all $j\not\in\l$ and $[G:H_i']=[G:H_{i_0}]$ for all $i\in\l,$ so $\left[G:\bigcap_{j\not\in\l} H_j'\right]=\left[G:\bigcap_{j\neq i_0}H_j\right]=\eta.$
Since $\bigcap_{j\neq i_0}H_j\subseteq H_{i_0},$ we get that $\bigcap_{j\not\in\l} H_j'\subseteq H_i'$ for any $i\in \l.$ Indeed, 
\begin{itemize}[wide]
            \item For all $\f\in\Phi(x,\mathbf{y})$ there is some finite conjunction $\th_\f$ of formulas from $\Phi(x,\mathbf{y})$ and some $j_1,\.,j_n$ different than $i_0$ such that $\models\A x(\th_\f(x,b_{j_1},\.,b_{j_n})\to\f(x,b_{i_0})).$ If this does not hold for the formula $\f\in\Phi(x,\mathbf{y}),$ then the set $\bigcup_{j\neq i_0}\Phi(x,b_j)\cup\{\neg\f(x,b_{i_0})\}$ is finitely satisfiable and thus, by saturation of $\UU$, non-empty, contradicting the hypothesis.

            \item By indiscernibility of $(b_i:i<|T|^{+})$ and the choice of the sequence $(c_i:i<(i_0-1)+\l+(|T|^{+}-i_0)),$ we get that for any formula $\f\in\Phi(x,\mathbf{y})$ and any $i\in\l$ there is a conjunction  $\th_\f$ of formulas from $\Phi(x,\mathbf{y})$ and some $j_1,\.,j_n\not\in\l$ such that $\models\A x(\th_\f(x,c_{j_1},\.,c_{j_n})\to\f(x,c_i)).$

            \item If $g\in\bigcap_{j\not\in\l} H_j'$ and $\f(x,c_i), i\in\l$ are given, then $g$ satisfies the corresponding $\th_\f(x,c_{j_1},\.,c_{j_n})$ and thus it satisfies $\f(x,c_i).$ This shows the desired inclusion. 
\end{itemize}
It follows that $G$ contains at least $\l$-many subgroups containing $\bigcap_{j\not\in\l} H_j',$ so $\l\leq2^{\eta}$ by Remark \ref{ind}, contradicting the choice of $\l.$
\end{proof}

The following are some well known lemmas about type-connected ideals in NIP structures.

\begin{lema}
    Suppose that $\D$ is a fragment closed under negations. Let $R$ be a $\D$-type-definable ring in a $\D$-NIP theory. If $I$ is a $\D$-type-definable ideal of $R,$ then $I^{00,\D}$ is also an ideal of $R.$
\end{lema}
\begin{proof}
Suppose that $a\in R$ is non-zero. First, to see that $aI^{00,\D}\subseteq(aI)^{00,\D},$ note that since $(aI)^{00,\D}$ is a subgroup of $aI$ of bounded index, then $a^{-1}(aI)^{00,\D}$ is also a subgroup of $I$ of bounded index. Therefore $I^{00,\D}\subseteq a^{-1}(aI)^{00,\D},$ hence $aI^{00,\D}\subseteq (aI)^{00,\D}.$ Second, to see that $(aI)^{00,\D}\subseteq I^{00,\D},$ note that since $|I/I^{00,\D}|$ is bounded, $|(aI+I^{00,\D})/I^{00,\D}|$ and $|aI/(aI\cap I^{00,\D})|$ are also bounded. Therefore $(aI)^{00,\D}\subseteq aI\cap I^{00,\D}\subseteq I^{00,\D}.$
\end{proof}

\begin{lema}[Cf.~{\cite[Lemma 2.6]{dp1a}}]
\label{j00}
    Suppose that $\D$ is a fragment closed under negations. Let $\KK=(K,\O_1,\.,\O_n)$ be a multi-valued field, $R=\O_1\cap\.\cap\O_n$ and $I$ be a $\D$-type-definable ideal of $R.$ If all the residue fields $\O_i/\M_i$ are infinite and $I^{00,\D}$ exists, then $I=I^{00,\D}.$
\end{lema}
\begin{proof}
    For any $b\in I\setminus I^{00,\D},$ there is a natural injection from $(Rb+I^{00,\D})/I^{00,\D}$ into $I/I^{00,\D},$ given that $I$ and $I^{00,\D}$ are ideals of $R.$ The natural map $R\to Rb/(Rb\cap I^{00,\D})$ sending $x$ to $xb+Rb\cap I^{00,\D}$ has a proper kernel $I'=\{x\in R:xb\in I^{00,\D}\},$ forcing $I'$ to be contained in a maximal ideal $\mathfrak{m}_i$ of $R,$ for some $i\in\{1,\.,l\}.$ Thus there is a natural surjection $R/I'\to R/\mathfrak{m}_i.$ Finally, since $\O_i/\M_i\cong R/\mathfrak{m}_i$ given that $\mathfrak{m}_i=R\cap\M_i$ (cf.~\cite[Lemma 2.1]{dp1a}), we can conclude that $|\O_i/\M_i|\leq|I/I^{00,\D}|.$
    If $\KK^{*}$ is a $\k$-saturated, $\k$-strongly homogeneous elementary extension of $\KK$ with $\k>2^{|T|},$ then $\k\leq|\O_i(K^{*})/\M_i(K^{*})|\leq |I(K^{*})/I^{00,\D}(K^{*})|\leq2^{|T|},$ where the last equality follows from the fact that $I^{00,\D}$ has bounded index in $I$, cf.~\cite[Proposition 5.1]{simon}.
\end{proof}

\subsection{Indiscernibles, Stably Embedded Sets and Chernikov-Hils-Jahnke-Simon's \textbf{SE} and \textbf{Im}}

Let $\UU$ be a monster model of a complete theory $T.$ Let $\t:\N\to\N$ be the function defined by $\t(n)=2n$ if $n$ is even, and $\t(n)=2n+1$ otherwise. This is a strictly increasing function satisfying that, for all $n<\omega,$ $n$ is even if and only if so is $\t(n)$ and $\t(2n+1)>\t(2n)+1,$ so $\t(2n)$ and $\t(2n+1)$ are not consecutive. 

\begin{fact}[Cf.~{\cite[Lemma 2.1]{js}}]
\label{dse}
     Let $A$ be a small set of parameters, let $I$ be an $A$-indiscernible sequence and let $d$ be some finite tuple of an $\emptyset$-definable stably-embedded set $D$ with NIP induced structure. Then the truth value of no formula with parameters in $Ad$ alternates infinitely many times along $I.$ 
\end{fact}

\begin{fact}[Cf.~{\cite[Lemma 2.2]{js}}] 
\label{indisc}
    Let $(a_i:i<\omega)$ be an indiscernible sequence and let $b$ be a finite tuple of parameters. For every $i<\omega,$ let $b_{2i}$ be a tuple coming from a stably-embedded sort $D$ whose induced structure is NIP. Assume further that the sequences $(a_{2i},a_{2i+1}:i<\omega)$ and $(a_{2i},b_{2i}:i<\omega)$ are indiscernible over $b.$ Then there are some sequences $(b_{2i+1}:i<\omega),(a_i':i<\omega)$ such that, if $(a_{2i}':i<\omega)=(a_{2i}:i<\omega),$ then 
    \begin{enumerate}[label*={\arabic*.}]
        \item $(a_i':i<\omega)\equiv_b(a_{\t(i)}:i<\omega),$ and 
        \item $(a_i',b_i:i<\omega)$ is indiscernible. 
    \end{enumerate}
    If, moreover, $(a_{2i}',b_{2i},a_{2i+1}':i<\omega)$ is indiscernible over $b,$ then $(b_{2i+1}:i<\omega)$ can be chosen in such a way that $(a_{2i}',b_{2i},a_{2i+1}',b_{2i+1}:i<\omega)$ is indiscernible over $b.$
\end{fact}
\begin{proof}
    We prove only the second statement, as the first statement of the lemma is proved in \cite[Lemma 2.2]{js}. Suppose that the sequence $(a_{2i}',b_{2i};a_{2i+1}':i<\omega)$ is indiscernible over $b.$ Let $(a_{2i}'',b_{2i}';a_{2i+1}'',b_{2i+1}':i<\omega)$ be a $b$-indiscernible sequence satisfying the EM-type of $(a_{2i}',b_{2i};a_{2i+1}',b_{2i+1}:i<\omega)$ over $b.$ Since $(a_{2i}',b_{2i};a_{2i+1}':i<\omega)$ is indiscernible over $b,$ we get that $(a_{2i}'',b_{2i}';a_{2i+1}'':i<\omega)\equiv_{b}(a_{2i}',b_{2i};a_{2i+1}':i<\omega),$ so there is some sequence $(b_{2i+1}^{*}:i<\omega)$ such that $(a_{2i}'',b_{2i}';a_{2i+1}'',b_{2i+1}':i<\omega)\equiv_{b}(a_{2i}',b_{2i};a_{2i+1}',b_{2i+1}^{*}:i<\omega).$ If we define $(b_{2i}^{*}:i<\omega)=(b_{2i}:i<\omega),$ then we get that the sequence $(a_{2i}',b_{2i}^{*};a_{2i+1}',b_{2i+1}^{*}:i<\omega)$ is indiscernible over $b.$ Finally, the sequence $(a_i',b_i^{*}:i<\omega)$ is indiscernible, because for any $i_1<\.<i_k<\omega$ and for any formula $\f(x_0y_0,x_1y_1;\.;x_{2k}y_{2k},x_{2k+1}y_{2k+1}),$ there are some $j_1<\.<j_k<\omega$ for which $$\models\f(a_{2i_0}'b_{2i_0}^{*},a_{2i_0+1}'b_{2i_0+1}^{*};\.;a_{2i_k}'b_{2i_k}^{*},a_{2i_k+1}'b_{2i_k+1}^{*})\sii\f(a_{2j_0}'b_{2j_0},a_{2j_0+1}'b_{2j_0+1};\.;a_{2j_k}'b_{2j_k},a_{2j_k+1}'b_{2j_k+1}),$$ and because the sequence $(a_i'b_i:i<\omega)$ is indiscernible. 
\end{proof}

Let $T$ be a complete theory in some language $\LL$ with sorts $\{\HH\}\cup\{\so_i:i<\nu\},$ for some cardinal $\nu.$ Note that, in this case, $\nu\leq|T|.$ Until the end of this section, for an $\LL$-structure $\MM$, we will denote its universe by the corresponding non-Fraktur letter $M,$ and by $\so(M)$ the universe of its $\so$-sort for $\so\in\{\HH\}\cup\{\so_i:i<\nu\}.$ 
\begin{defin}
Given an $\LL$-structure $\MM$, we say that a \emph{finitary invariant closure operator} on $\MM,$ $\cl^{\MM}:\mathcal{P}(M)\to\mathcal{P}(M),$ is a closure operator on the subsets of $M$ satisfying, for all $A\subseteq M$ and all $a\in M,$
\begin{enumerate}[label*={\arabic*.}]
    \item $A\subseteq\cl^{\MM}(A),$

    \item $\cl^{\MM}(\cl^{\MM}(A))=\cl^{\MM}(A),$

    \item \underline{\textit{Finite Character}}: $a\in\cl^{\MM}(A)$ if and only if there is some finite subset $A_0$ of $A$ for which $a\in\cl^{\MM}(A_0),$

    \item \underline{\textit{Invariance}}: $\s(\cl^{\MM}(A))=\cl^{\MM}(\s(A))$ for all $\s\in\Aut(\MM),$

    \item $\cl^{\MM}(A)=\cl^{\NN}(A)$ for all $\NN\succeq\MM,$ and

    \item $|\cl(A)|\leq |A|+|T|.$ 
\end{enumerate}
\end{defin}
If the ambient model is clear from the context, we drop the superscript and write $\cl(A)$ instead of $\cl^{\MM}(A)$ for all $A\subseteq M.$ Examples of such closure operations are $\acl,\,\dcl$ or the closure $A\to\langle A\rangle$ that sends $A$ to the $\LL$-substructure of $\MM$ generated by $A.$ For a subset $A$ of $M$ and a sort $\so\in\{\HH\}\cup\{\so_i:i<\nu\},$ we write $\so[A]:=\cl(A)\cap\so(M).$  Let $\UU$ be a monster model of $T.$  

\begin{lema}
\label{pivot}
    Let $\cl$ be a finitary invariant closure operator on $\UU.$ For each $i<\omega,$ let $a_{i}$ be a tuple of length $\k$ from $\UU$ with $\k\leq|T|$ and let $b$ be some parameter. If $(a_{2i},a_{2i+1}:i<\omega)$ is indiscernible over $b,$ then there is a sequence of tuples $(\a_{2i}:i<\omega)$ of length $|T|$ such that, for each $i<\omega,$ the tuple $\a_{2i}$ enumerates $\so[a_{2i}b],$ and such that the sequence $(a_{2i},\a_{2i};a_{2i+1}:i<\omega)$ is indiscernible over $b.$ 
\end{lema}
\begin{proof}
    Let $a_0'$ be an enumeration of $\so[a_0b],$ and for all $i>0$ let $a_{2i}'$ be such that $a_0,a_0'\equiv_b a_{2i},a_{2i}'.$ Let $(c_{2i},c_{2i}';c_{2i+1}:i<\omega)$ be a $b$-indiscernible sequence satisfying the EM-type of $(a_{2i},a_{2i}';a_{2i+1}:i<\omega)$ over $b.$ Since $(a_{2i},a_{2i+1}:i<\omega)$ is indiscernible over $b,$ we get that $(a_{2i},a_{2i+1}:i<\omega)\equiv_{b} (c_{2i},c_{2i+1}:i<\omega).$ We can define $(\a_i:i<\omega)$ as any sequence satisfying $(a_{2i},\a_{2i};a_{2i+1}:i<\omega)\equiv_b(c_{2i},c_{2i}';c_{2i+1}:i<\omega).$ 
\end{proof}

\begin{lema}
\label{liminm}
    Let $\cl$ be a finitary invariant closure operator on $\UU.$ Let $(\FF_n)_{n<\omega}$ be a sequence of substructures of $\UU$ and let $b\in U$ be a tuple satisfying $\FF_n\preceq\UU,$ $\FF_n\subseteq\FF_{n+1}$ and $\so_i[F_n]\subseteq\so_i[F_nb]\subseteq\so_i[F_{n+1}]$ for each $n<\omega$ and each $i<\nu.$ Let $\FF=\bigcup_{n<\omega}\FF_n.$ Then $\FF\preceq\UU$ and $\so_i[Fb]=\so_i[F]$ for all $i<\nu.$  
\end{lema}
\begin{proof}
    By construction of $F$, $\cl(Fb)=\bigcup_{n<\omega}\cl(F_nb),$ so $\so_i[Fb]=\bigcup_{n<\omega}\so_i[F_nb]=\bigcup_{n<\omega}\so_i[F_{n+1}]=\so_i[F]$ for all $i<\nu.$
    Finally, given that $\FF_n\subseteq\FF_{n+1}$ are both elementary substructures of $\UU,$ we get that $\FF_n\preceq\FF_{n+1}$ for each $n<\omega,$ so $\FF\preceq\UU.$  
\end{proof}

\begin{lema}
\label{2i2i+1}
    Let $(a_i:i<\omega)$ be an indiscernible sequence in $\UU$ and $b\in\UU$ be a singleton such that $\models\f(a_i,b)$ if and only if $i$ is even, for some $\LL$-formula $\f(x,y).$ Then there is some automorphism $\sigma$ of $\UU$ such that 
    $(\sigma(a_i):i<\omega)$ is indiscernible, $(\sigma(a_{2i}),\sigma(a_{2i+1}):i<\omega)$ is indiscernible over $b,$ and $\models\f(\sigma(a_i),b)$ if and only if $i$ is even. 
\end{lema}
\begin{proof}
    Let $(c_i:i<\omega)$ be a sequence in $\UU$ such that $(c_{2i},c_{2i+1}:i<\omega)$ is indiscernible over $b$ and realizes the EM-type of $(a_{2i},a_{2i+1}:i<\omega)$ over $b.$ This implies that $(c_i:i<\omega)\equiv(a_i:i<\omega),$ because $(a_i:i<\omega)$ is indiscernible. Let $\sigma$ be an automorphism of $\UU$ sending $(a_i:i<\omega)$ to $(c_i:i<\omega).$ Since $\models\f(a_{2i},b)\wedge\neg\f(a_{2i+1},b)$ for all $i<\omega,$ then $\f(x_1,b)\wedge\neg\f(x_2,b)$ is a formula in the EM-type of $(a_{2i},a_{2i+1}:i<\omega)$ over $b,$ and is thus realized by $(c_{2i},c_{2i+1}:i<\omega),$ i.e.~$\models\f(c_{i},b)$ if and only if $i$ is even. 
\end{proof}

\begin{lema}
\label{tupleexttomod}
    For each $i<\omega,$ let $a_i$ be a tuple of length $\k$ from $\UU$ with $\k\leq|T|.$ 
    Suppose that the sequence $(a_i:i<\omega)$ is indiscernible and that the sequence $(a_{2i},a_{2i+1}:i<\omega)$ is indiscernible over some singleton $b\in U.$
    Then, for each $i<\omega,$ there is some tuple $A_i$ of length $|T|$ and an elementary substructure $\MM_i\preceq\UU$ such that 
     \begin{enumerate}[label*={\arabic*.}]
        \item $A_i$ enumerates the universe of $\MM_i,$
    
        \item $a_i$ is an initial segment of $A_i,$
    
        \item the sequence $(A_i:i<\omega)$ is indiscernible, and

        \item the sequence $(A_{2i},A_{2i+1}:i<\omega)$ is indiscernible over $b.$
    \end{enumerate}
\end{lema}
\begin{proof} 

    For each $i<\omega,$ let $\NN_i\preceq\UU$ be a model of size $|T|$ containing $a_i,$ and let $B_i=a_i^{\frown}a_i'$ be a tuple of length $|T|$ enumerating the universe of $\NN_i$ in such a way that $B_i\equiv B_0$ for all $i<\omega.$ Let $(C_i:i<\omega)$ be a sequence of tuples of $\UU$ of length $|T|$ such that 
    \begin{enumerate}[label*={\arabic*.}]
        \item $(C_i:i<\omega)$ is indiscernible and satisfies the EM-type of the sequence $(B_i:i<\omega),$ 

        \item $(C_{2i},C_{2i+1}:i<\omega)$ is indiscernible over $b$ and satisfies the EM-type of the sequence $(B_{2i},B_{2i+1}:i<\omega)$ over $b.$ 
    \end{enumerate}
    Such a sequence exists by Ramsey and saturation of $\UU.$ If $C_i=c_i^{\frown}c_i',$ then $(c_i:i<\omega)\equiv_{b}(a_i:i<\omega),$ 
    because $(a_{2i},a_{2i+1}:i<\omega)$ is indiscernible over $b.$ Then there is some partial elementary map $F$ that fixes $b,$ that sends $(c_i:i<\omega)$ to $(a_i:i<\omega),$ and whose domain contains each $C_i.$ Since each $B_i$ enumerates an elementary substructure of $\UU,$ the same holds for each $C_i.$ We may then define $A_i=F(C_i),$ so that $A_i$ is indeed an enumeration of the universe of an elementary substructure $\MM_i\preceq\UU.$  
\end{proof}

In what follows, Definition \ref{sideconds} and Proposition \ref{niptransfer} were suggested by Itay Kaplan and Martin Bays as a model-theoretic generalization of the original phenomenon of NIP transfer presented by Jahnke and Simon in \cite[Theorem 2.3]{js}, which is stated as a phenomenon of henselian valued fields. The author is grateful for their encouragement to include a proof in this document. 

\begin{defin}[Side Conditions]
\label{sideconds}
Let $T$ be a complete theory in some language $\LL$ with sorts $\{\HH\}\cup\{\so_i:i<\nu\}$ for some cardinal $\nu.$ Let $\cl$ be a finitary invariant closure operator on $\UU.$
    \begin{itemize}
        \item We say that \emph{$T$ satisfies \textbf{SE}} if for all models $\MM\models T$ and all $i<\nu,$ the set $\so_i(M)$ is stably-embedded in $\MM.$ 
        
        \item We say that \emph{$T$ satisfies $\cl$-\textbf{Im}} if whenever $\NN\preceq\MM$ are models of $T$ and $b\in M$ is such that $\so_i[Nb]=\so_i[N]$ for all $i<\nu,$ then there is some subset $D\subseteq \tp(b/N)$ consisting of NIP formulas such that $D\vdash \tp(b/N).$ 
    \end{itemize}
\end{defin}

\begin{propo}[Cf.~{\cite[Theorem 2.3]{js}}]
\label{niptransfer}
    Let $T$ be a theory in some language $\LL$ with sorts $\{\HH\}\cup\{\so_i:i<\nu\}$ satisfying \textbf{SE} and $\cl$-\textbf{Im} for some finitary invariant closure operator $\cl$ on $\UU.$ If the induced structure of $\so_j(M)$ is NIP for all models $\MM$ of $T$ and all $j<\nu,$ then $T$ is NIP.
\end{propo}

\begin{proof}
    If $T$ is not NIP, there is some $\LL$-formula $\f(x,y)$ with the independence property. 
    By \cite[Proposition 2.11]{simon}, we may assume that $|y|=1.$ 
    Therefore, there is some indiscernible sequence $(a_i:i<\omega)$ of $|x|$-tuples of $\UU$ and some singleton $b\in U$ such that $\models\f(a_i,b)$ if and only if $i$ is even. Since $T$ satisfies \textbf{SE}, by Fact \ref{dse}, such an element $b$ has to be an element of the sort $\HH(U).$

    By induction on $n<\omega,$ we are going to construct a sequence $(S_n:n<\omega),$ with $S_n=(\aa_i^{n}:i<\omega),$ of sequences of $|T|$-tuples of elements of $\UU,$ satisfying the following properties:
    \begin{enumerate}[label*={\arabic*.}]
        \item each $\aa_i^{n}=(a_\mu^{i,n}:\mu<|T|)$ enumerates the universe of a substructure $\MM_i^{n}\preceq\UU,$

        \item $(\aa_i^{n}:i<\omega)$ is indiscernible,

        \item $(\aa_{2i}^{n},\aa_{2i+1}^{n}:i<\omega)$ is indiscernible over $b,$

        \item $\models\f(a_{0}^{i,n},b)$ if and only if $i$ is even,
        
        \item $
            \MM_{2i}^{n+1}\supseteq\MM_{2i}^{n}$ and $\so_j[M_{2i}^{n+1}]\supseteq\so_j[M_{2i}^{n}b]\supseteq\so_j[M_{2i}^{n}]$ for all $j<\nu.$
            
    \end{enumerate}
    \begin{figure}[ht]
    \centering
    \begin{tabular}{cccc}
    
       \begin{tikzcd}
\MM_{2i}^{n+1}                              &  & {\so_j[M_{2i}^{n+1}]}                                       &                                          \\
                                            &  &                                                             & {\so_j[M_{2i}^{n}b]} \arrow[lu, no head] \\
\MM_{2i}^{n} \arrow[uu, "\preceq", no head] &  & {\so_j[M_{2i}^{n}]} \arrow[uu, no head] \arrow[ru, no head] &                                   
\end{tikzcd} \\ 
    \end{tabular}
    \caption{The construction of the $2i^{th}$ structure $\MM_{2i}^{n+1}$ at stage $n+1.$}
\end{figure}
    Let $\MM=\bigcup_{n<\omega}\MM_{0}^{n}.$ By Lemma \ref{liminm}, $\so_j[Mb]=\so_j[M]$ for all $j<\nu$ and $\MM\preceq\UU.$ Also, putting $\aa_0=\bigcup_{n<\omega}\aa_0^{n},$
    yields that $\MM$ is enumerated by $\aa_{0}.$
    Since $a_0^{0,0}\in M$ and $\models\f(a_0^{0,0},b),$ by $\cl$-\textbf{Im}, there is some NIP formula $\psi(z,y),$ some $n<\omega$ and some $|z|$-tuple $c_0=(a_{\mu_1}^{0,n},\.,a_{\mu_{|z|}}^{0,n})$ of $M_{0}^{n},$ where $\mu_1<\.<\mu_{|z|}<|T|,$ such that $\models\psi(c_0,b)$ and $\models\A y(\psi(c_0,y)\to\f(a_0^{0,0},y)).$
    It follows that $c_0$ is a subtuple of $\aa_0^{n}.$ Since $\aa_0^{n}$ extends $\aa_{0}^{0},$ we get that $a_0^{0,0}=a_0^{0,n}.$
    Let $c_i=(a_{\mu_1}^{i,n},\.,a_{\mu_{|z|}}^{i,n})$ be the corresponding $|z|$-tuple of $\aa_i^{n}$ for all $i<\omega.$ Since $(\aa_i^{n}:i<\omega)$ is indiscernible, the sequence of subtuples $(a_0^{i,n},c_i:i<\omega)$ is also indiscernible.
    Since $\models\f(a_0^{i,n},b)$ if and only if $i$ is even, we conclude that $\models\neg\psi(c_i,b)$ whenever $i$ is odd.
    Additionally, since $\psi(c_0,y)\in\tp(b/M),$ then $\models\psi(c_0,b)\wedge\neg\psi(c_1,b).$ 
    Since $(\aa_{2i}^{n}\aa_{2i+1}^{n}:i<\omega)$ is indiscernible over $b,$ we get that the sequence of subtuples $(c_{2i},c_{2i+1}:i<\omega)$ is also indiscernible over $b,$ thus $\models\psi(c_{2i},b)\wedge\neg\psi(c_{2i+1},b)$ for all $i<\omega,$ i.e.~the truth value of $\psi(c_i,b)$ alternates infinitely many times along $i<\omega$. This contradicts the hypothesis of $\psi(z,y)$ being NIP, because the sequence $(c_i:i<\omega)$ is indiscernible.
    Now we proceed with the required construction.

\underline{Base Case:} By Lemma \ref{2i2i+1}, there is some automorphism $\sigma$ of $\UU$ such that $(\sigma(a_i):i<\omega)$ is indiscernible, $(\sigma(a_{2i}),\sigma(a_{2i+1}):i<\omega)$ is indiscernible over $b,$ and $\models\f(\sigma(a_i),b)$ if and only if $i$ is even. We may thus extend each $\sigma(a_i)$ to a $|T|$-tuple $\aa_i^{0},$ with initial segment $\sigma(a_i),$ as in Lemma \ref{tupleexttomod}. We get in particular that $\aa_i^{0}$ enumerates the universe of some $\MM_i^{0}\preceq\UU$ of size $|T|,$ that $(\aa_i^{0}:i<\omega)$ is indiscernible, $(\aa_{2i}^{0},\aa_{2i+1}^{0}:i<\omega)$ is indiscernible over $b,$ and that $\models\f(a_0^{i,0},b)$ if and only if $i$ is even.
Hence, $S_0=(\aa_i^{0}:i<\omega)$ satisfies 1-4 above.

\underline{Inductive Step:}
    Suppose $S_n$ has been constructed satisfying 1-4 above. We shall construct $S_{n+1}$ satisfying 1-5 above. By induction on $\l<\nu,$ we will show below that for every $\l<\nu$ and every $i<\omega$ there are some sequences $\a_i^{\l}=(\a_{i,\mu}^{\l}:\mu<|T|)$ and $\bb_{i,\l}=(\b_{i,j}:j\leq\l)$ satisfying that   
    \begin{enumerate}[label*=\textit{\roman*.}]
        \item $\a_{2i}^{\l}=\aa_{2i}^{n}$ and $\b_{2i,j}$ enumerates $\so_j[M_{2i}^{n}b]$ for all $j\leq\l,$ 

        \item $\models\f(\a_{i,0}^{\l},b)$ if and only if $i$ is even, 
        
        \item $(\a_{i}^{\l},\bb_{i,\l}:i<\omega)$ is indiscernible, and

        \item $(\a_{2i}^{\l},\bb_{2i,\l};\a_{2i+1}^{\l},\bb_{2i+1,\l}:i<\omega)$ is indiscernible over $b.$
    \end{enumerate}
    Once these sequences are constructed, for each $i<\omega,$ define $\bb_i=(\b_{i,j}:j<\nu).$ By compactness, there is some sequence $(\a_i:i<\omega)$ with $\a_{i}=(\a_{i,\mu}:\mu<|T|)$ for all $i<\omega,$ satisfying that 
    \begin{enumerate}[label*=\textit{\Roman*.}]
        \item $\a_{2i}=\aa_{2i}^{n}$ for all $i<\omega,$

        \item $\models\f(\a_{i,0},b)$ if and only if $i$ is even, 
        
        \item $(\a_{i},\bb_{i}:i<\omega)$ is indiscernible, and

        \item $(\a_{2i},\bb_{2i};\a_{2i+1},\bb_{2i+1}:i<\omega)$ is indiscernible over $b.$
    \end{enumerate}
    By Lemma \ref{tupleexttomod}, each tuple $(\a_{i},\bb_{i})$ is an initial segment of some tuple $\aa_i^{n+1}$ of length $|T|$ that enumerates some elementary substructure $\MM_i^{n+1}\preceq\UU,$ whose sequence $(\aa_i^{n+1}:i<\omega)$ is indiscernible and whose sequence $(\aa_{2i}^{n+1},\aa_{2i+1}^{n+1}:i<\omega)$ is indiscernible over $b.$ 
    Since $\a_{i}=(\a_{i,\mu}:\mu<|T|)$ is an initial segment of $\aa_i^{n+1}=(a_\mu^{i,n+1}:\mu<|T|),$ we get that $\aa_0^{i,n+1}=\a_{i,0}$ and that $a_0^{i,n+1}=\a_{i,0}$, implying that $\models\f(a_0^{i,n+1},b)$ if and only if $i$ is even.
    Hence $S_{n+1}=(\aa_i^{n+1}:i<\omega)$ satisfies 1-4 above. 
    The tuple $\aa_{2i}^{n+1}$ enumerates an extension of $\a_{2i}=\aa_{2i}^{n},$ implying that $\MM_{2i}^{n+1}\supseteq\MM_{2i}^{n}.$ 
    Finally, the tuple $\aa_{2i}^{n+1}$ enumerates an extension of $\b_{2i,j}$ for all $j<\nu,$ meaning that $\so_j[M_{2i}^{n+1}]\supseteq\so_j[M_{2i}^{n}b]\supseteq\so_j[M_{2i}^{n}]$ for all $j<\nu.$ We conclude that $S_{n+1}$ satisfies 1-5 above. We now proceed with the aforementioned induction on $\l<\nu.$ 
    \begin{itemize}
        \item \underline{$\l=0$:} 
        Use Lemma \ref{pivot} to find, for each $i<\omega,$ a $|T|$-tuple $\b_{2i,0}$ from $\UU$ enumerating $\so_0[M_{2i}^{n}b]$ and such that $(\aa_{2i}^{n},\b_{2i,0},\aa_{2i+1}^{n}:i<\omega)$ is indiscernible over $b.$ 
        By Fact \ref{indisc}, there are tuples $(\b_{2i+1,0}:i<\omega)$ and $(\a_{2i+1}^{0}:i<\omega)$ such that, if $\a_{2i}^{0}=\aa_{2i}^{n}$ for all $i<\omega,$ then $(\a_i^{0}:i<\omega)\equiv_{b}(\aa_{\t(i)}^{n}:i<\omega),$
        the sequence $(\a_i^{0},\b_{i,0}:i<\omega)$ is indiscernible and the sequence $(\a_{2i}^{0},\b_{2i,0};\a_{2i+1}^{0},\b_{2i+1,0}:i<\omega)$ is indiscernible over $b.$ In particular, if $\a_i^{0}=(\a_{i,\mu}^{0}:\mu<|T|)$ and given that $\aa_{\t(i)}^{n}=(a_{\mu}^{\t(i),n}:\mu<|T|),$ we get that
        $\models\f(\a_{i,0}^{0},b)\Sii\models\f(a_{0}^{\t(i),n},b)\Sii i\text{ is even.}$
        It follows that, for all $i<\omega,$ the sequences $\a_{i}^{0}$ and $\bb_{i,0}=\b_{i,0}$ satisfy $i$-$iv$ above.

        \item \underline{$\l=\eta+1$ for some $\eta<\nu$:} Suppose that, for all $i<\omega,$ the sequences $\a_{i}^{\eta}=(\a_{i,\mu}^{\eta}:\mu<|T|)$ and $\bb_{i,\eta}$ satisfy $i$-$iv$ above. 
        By Lemma \ref{pivot}, there is a sequence $(\b_{2i,\l}:i<\omega)$ of $|T|$-tuples from $\UU$ such that, for each $i<\omega,$ the tuple $\b_{2i,\l}$ enumerates $\so_\l[M_{2i}^{n}b]$ and the sequence $(\a_{2i}^{\eta},\bb_{2i,\eta},\b_{2i,\l};\a_{2i+1}^{\eta},\bb_{2i+1,\eta}:i<\omega)$ is indiscernible over $b.$ 
        By Fact \ref{indisc}, there are tuples $(\b_{2i+1,\l}:i<\omega)$ and $(\a_{2i+1}^{\l},\bb_{2i+1,\eta}':i<\omega)$ such that, if $(\a_{2i}^{\l},\bb_{2i,\eta}')=(\a_{2i}^{\eta},\bb_{2i,\eta})$ for all $i<\omega,$ then $(\a_i^{\l},\bb_{i,\eta}':i<\omega)\equiv_b(\a_{\t(i)}^{\eta},\bb_{\t(i),\eta}:i<\omega),$ the sequence $(\a_i^{\l},\bb_{i,\eta}',\b_{i,\l}:i<\omega)$ is indiscernible and the sequence $(\a_{2i}^{\l},\bb_{2i,\eta}',\b_{2i,\l};\a_{2i+1}^{\l},\bb_{2i+1,\eta}',\b_{2i+1,\l}:i<\omega)$ is indiscernible over $b$. 
        In particular, if $\a_i^{\l}=(\a_{i,\mu}^{\l}:\mu<|T|),$ then
        $\models\f(\a_{i,0}^{0},b)\Sii\models\f(a_{0}^{\t(i),n},b)\Sii i\text{ is even.}$
        It follows that, for all $i<\omega,$ the sequences $\a_{i}^{\l}$ and $\bb_{i,\l}=(\bb_{i,\eta}',\b_{i,\l})$ satisfy $i$-$iv$ above.
        
        \item \underline{$\l<\nu$ is a non-zero limit ordinal:}
        For each $i<\omega,$ define $\bb_i'=(\b_{i,j}:j<\l).$ 
        By compactness, there is some sequence $(\a_i':i<\omega)$ with $\a_{i}'=(\a_{i,\mu}':\mu<|T|)$ for all $i<\omega,$ satisfying that 
        \begin{itemize}[label=$\checkmark$]
            \item $\a_{2i}'=\aa_{2i}^{n}$ for all $i<\omega,$

            \item $\models\f(\a_{i,0}^{\l},b)$ if and only if $i$ is even, 
        
            \item $(\a_{i}',\bb_{i}':i<\omega)$ is indiscernible, and

            \item $(\a_{2i}',\bb_{2i}';\a_{2i+1}',\bb_{2i+1}':i<\omega)$ is indiscernible over $b.$
        \end{itemize}
        By Lemma \ref{pivot}, there is a sequence $(\b_{2i,\l}:i<\omega)$ of $|T|$-tuples from $\UU$ such that, for each $i<\omega,$ the tuple $\b_{2i,\l}$ enumerates $\so_\l[M_{2i}^{n}b]$ and the sequence $(\a_{2i}',\bb_{2i}',\b_{2i,\l};\a_{2i+1}',\bb_{2i+1}':i<\omega)$ is indiscernible over $b.$ 
        By Fact \ref{indisc}, there are tuples $(\b_{2i+1,\l}:i<\omega)$ and $(\a_{2i+1}^{\l},\bb_{2i+1}^{*}:i<\omega)$ such that, if $(\a_{2i}^{\l},\bb_{2i}^{*})=(\a_{2i}',\bb_{2i}')$ for all $i<\omega,$ then $(\a_i^{\l},\bb_{i}^{*}:i<\omega)\equiv_b(\a_{\t(i)}',\bb_{\t(i)}':i<\omega),$ the sequence $(\a_i^{\l},\bb_{i}^{*},\b_{i,\l}:i<\omega)$ is indiscernible and the sequence $(\a_{2i}^{\l},\bb_{2i}^{*},\b_{2i,\l};\a_{2i+1}^{\l},\bb_{2i+1}^{*},\b_{2i+1,\l}:i<\omega)$ is indiscernible over $b$. 
        In particular, if $\a_i^{\l}=(\a_{i,\mu}^{\l}:\mu<|T|),$ then
        $\models\f(\a_{i,0}^{0},b)\Sii\models\f(a_{0}^{\t(i),n},b)\Sii i\text{ is even.}$
        It follows that, for all $i<\omega,$ the sequences $\a_{i}^{\l}$ and $\bb_{i,\l}=(\bb_{i}^{*},\b_{i,\l})$ satisfy $i$-$iv$ above.
        \qedhere
    \end{itemize}

    \qedhere
\end{proof}

\section{Quantifier-free \texorpdfstring{$\LL_{div,\l}$}{}-formulas and \texorpdfstring{$\texttt{SCVF}_e$}{}}
\label{sectlres}

In what follows, a \emph{$\l$-field} is a field expanded by its parametrized $\l$-functions. Recall that $\LL_{div}$ is the language $\LL_{ring}$ expanded by a binary relation $|,$ and any valued field $(K,v)$ is naturally an $\LL_{div}$-structure by interpreting $x|y$ as $v(x)\leq v(y)$ for all $x,y\in K.$ We denote by $\LL_{div,\l}$ the expansion of $\LL_{div}$ by the symbols of the parametrized $\l$-functions.     

In this section we prove that quantifier-free $\LL_{div,\l}$-formulas are NIP in any theory of valued $\l$-fields. As a corollary, we get that $\texttt{SCVF}_e$ is NIP for \emph{any} Ershov degree $e\in\N\cup\{\infty\}.$ We will use the following fact.

\begin{fact}[Cf.~{\cite[Theorem 4.12]{h}}]
\label{scvfqe}
    The theory $\texttt{SCVF}_e$ admits quantifier elimination in the language $\LL_{div,\l}$ for any $e\in\N\cup\{\infty\}.$ 
\end{fact}

Let us start with a remark. Let $M$ be a field of characteristic exponent $p$, let $m\geq n$ be positive integers and let $A$ be an $m\times n$ matrix of rank $n$ over $M.$ This implies that there are some indices $\mu_1<\.<\mu_n\leq m$ for which the corresponding rows of $A$ form an $n\times n$ invertible sub-matrix $A^{\mu}.$ Let $A^{\mu}_{L}$ be the $n\times m$ matrix whose $\mu_i^{th}$ column corresponds to the $i^{th}$ column of $(A^{\mu})^{-1}$ for all $i\in\{1,\.,n\},$ and whose $j^{th}$ column is zero for all $j\in\{1,\.,m\}\setminus\{\mu_1,\.,\mu_n\}.$ It follows that $A^{\mu}_{L}$ is a left inverse of $A.$ If $\adj(A^{\mu})$ is the adjugate matrix of $A^{\mu},$ then $(A^{\mu})^{-1}_{i,j}=(\adj(A^{\mu}))_{i,j}/\det(A^{\mu}).$ If $x=(X_{k,l}:(k,l)\in n\times n),$ then there is some polynomial $J_{i,j}\in\F_p[x]$ such that $(\adj(C))_{i,j}=J_{i,j}(C)$ for all $n\times n$ matrices $C,$ and thus $(A^{\mu})^{-1}_{i,j}=J_{i,j}(A^{\mu})/\det(A^{\mu}).$ Therefore $$(A^{\mu}_{L})_{i,j}=\begin{cases}
        J_{i,k}(A^{\mu})/\det(A^{\mu})&\text{ if }j=\mu_k\text{ for some }k\in\{1,\.,n\},\\
        0&\text{ otherwise. }
    \end{cases}$$ 
We have proved the following Lemma.

\begin{lema}
\label{la}
If $A$ is a full rank $m\times n$ matrix over $M,$ with $m\geq n,$ and $A^{\mu}$ is an invertible $n\times n$ sub-matrix of $A,$ then $A^{\mu}_{L}$ is a left inverse of $A$ and there are some rational functions $(h_{i,j})_{(i,j)\in n\times m}$ over $\F_p$ such that $(A^{\mu}_{L})_{i,j}=h_{i,j}(A).$
\end{lema}

Now we define some notations. If $N\in\N,$ we defined $[N]$ to be the set $\{1,\.,N\}.$ For each $n<\omega,$ let $a_n=(a_n^{k})_{k\in[N]}$ be an $N$-tuple of elements of a $\l$-field $\MM$ of characteristic exponent $p.$ If $S\subseteq[N],$ we will write $a_n^{S}=(a_n^{k})_{k\in S}.$ If $l<\omega$ and $(S_n)_{n<\omega}$ is a sequence of subsets of $[N],$ we will write $a_{<l}^{S_{<l}}=(a_n^{S_n})_{n<l},$ $a_{\leq l}^{S_{\leq l}}=(a_n^{S_n})_{n\leq l}$ and $a_{\geq l}^{S_{\geq l}}=(a_n^{S_n})_{n\geq l}.$ 
Also, let $\bb$ be a finite tuple of parameters from $\MM.$ Throughout this section, we will assume that the sequence $a_{<\omega}=(a_n)_{n<\omega}$ is indiscernible.
If $x=(x_1,\.,x_n)$ and $y$ are finite tuples of elements of $M$ and $m\in\Mon(|y|),$ we will write $\l_{m}(x_1;y)=\l_{|y|,m}(x_1;y)$ and $\l_{m}(x;y):=(\l_{m}(x_1;y),\.,\l_{m}(x_n;y)).$ Finally, we will write $\l(x;y):=(\l_{m}(x;y))_{m\in\Mon(|y|)}.$

\begin{lema}
    There is a sequence $(S_{n})_{n<\omega}$ of subsets of $[N]$ such that $a_{<\omega}^{S_{<\omega}}$ is $p$-independent in $M,$ $S_{n+1}\subseteq S_{n}$ and $a_{<n}\subseteq M^{p}(a_{<n}^{S_{<n}})$ for all $n<\omega.$
\end{lema}
\begin{proof}
    Let $S_0\subseteq[N]$ be such that $a_0^{S_0}$ is a maximally $p$-independent subset of $a_0.$ In this way, $a_0\subseteq M^{p}(a_0^{S_0}).$ If $S_{\leq n}$ have been defined for some $n<\omega$ in such a way that $a_{\leq n}^{S_{\leq n}}$ is $p$-independent and $a_{\leq n}\subseteq M^{p}(a_{\leq n}^{S_{\leq n}}),$ then let $S_{n+1}\subseteq S_n$ be such that $a_{\leq n+1}^{S_{\leq n+1}}$ is a maximally $p$-independent subset of $a_{\leq n}^{S_{\leq n}}a_{n+1}^{S_{n}}.$ It follows that $a_{\leq n}^{S_{\leq n}}a_{n+1}^{S_{n}}\subseteq M^{p}(a_{\leq n+1}^{S_{\leq n+1}}).$ We claim that, moreover, $a_{\leq n+1}\subseteq M^{p}(a_{\leq n+1}^{S_{\leq n+1}}).$ To achieve this, we only have to check that $a_{n+1}^{i}\in M^{p}(a_{\leq n+1}^{S_{\leq n+1}})$ for any $i\not\in S_{n+1}.$ Indeed, by induction hypothesis, $a_n^{i}\in M^{p}(a_{\leq n}^{S_{\leq n}}),$ and by indiscernibility, $a_{n+1}^{i}\subseteq M^{p}(a_{<n}^{S_{<n}}a_{n+1}^{S_{n}})\subseteq M^{p}(a_{\leq n+1}^{S_{\leq n+1}}).$    
\end{proof}

Let $\bb'\subseteq\bb$ be such that $a_{<\omega}^{S_{<\omega}}\bb'$ is a maximally $p$-independent subset of $a_{<\omega}^{S_{<\omega}}\bb.$ 
In particular, since $\bb$ is finite, there is a minimum $l_1<\omega$ such that $\bb\subseteq M^{p}(a_{<l_1}^{S_{<l_1}}\bb').$
Also, since $(S_n)_{n<\omega}$ is a $\subseteq$-decreasing sequence of subsets of the finite set $[N],$ there is some minimum $l_2<\omega$ such that $S_n=S_{l_2}$ for all $n\geq l_2.$ This implies, by indiscernibility, that $a_n\subseteq M^{p}(a_n^{S_{l_2}}a_{<l_2}^{S_{<l_2}}).$ Therefore, if $l=\max\{l_1,l_2\},$ then $\F_p(a_n,\bb)\subseteq M^{p}(a_n^{S_l}a_{<l}^{S_{<l}}\bb')$ for all $n\geq l.$
As a consequence, if $n\geq l,$ then $a_{n}\subseteq M^{p}(a_n^{S_l}a_{<l}^{S_{<l}})$ and $\bb\subseteq M^{p}(a_{<l}^{S_{<l}}\bb')$ imply that $\l(a_n;a_n^{S_l}a_{<l}^{S_{<l}}\bb')\subseteq\l(a_n;a_n^{S_l}a_{<l}^{S_{<l}})\cup\{0\}$ and $\l(\bb;a_n^{S_l}a_{<l}^{S_{<l}}\bb')\subseteq\l(\bb;a_{<l}^{S_{<l}}\bb')\cup\{0\}$ respectively.

\begin{cor}
\label{ltorat}
    For any monomial $m\in\Mon(|S_0|+\.+|S_l|+|\bb'|)$ and any rational function $f$ over $\F_p,$ there is some rational function $g$ over $\F_p$ such that $$\l_m(f(a_n,\bb);a_n^{S_l}a_{<l}^{S_{<l}}\bb')=g(a_{n}^{S_l},\l(a_n;a_n^{S_l}a_{<l}^{S_{<l}}),\l(\bb;a_{<l}^{S_{<l}}\bb'),a_{<l}^{S_{<l}}\bb')$$ for all $n\geq l.$
\end{cor}
\begin{proof}
    By \cite[Lemma 2.16]{at}, there is a canonical rational function $h$ over $\F_p,$ depending only on $m$ and $f,$\footnote{Uniformity of $h$ follows rather from an iteration of \cite[Lemma 2.15]{at} than from \cite[Lemma 2.16]{at}.} satisfying that $$\l_m(f(a_n,\bb);a_n^{S_l}a_{<l}^{S_{<l}}\bb')=h(a_n^{S_l},\l(a_n;a_n^{S_l}a_{<l}^{S_{<l}}\bb'),\l(\bb;a_n^{S_l}a_{<l}^{S_{<l}}\bb'),a_{<l}^{S_{<l}}\bb')$$ for all $n\geq l.$ By construction, the inclusions $\l(a_n;a_n^{S_l}a_{<l}^{S_{<l}}\bb')\subseteq\l(a_n;a_n^{S_l}a_{<l}^{S_{<l}})\cup\{0\}$ and $\l(\bb;a_n^{S_l}a_{<l}^{S_{<l}}\bb')\subseteq\l(\bb;a_{<l}^{S_{<l}}\bb')\cup\{0\}$ imply that there is some rational function $g$ over $\F_p,$ depending only on $h,$ satisfying that $$h(a_n^{S_l},\l(a_n;a_n^{S_l}a_{<l}^{S_{<l}}\bb'),\l(\bb;a_n^{S_l}a_{<l}^{S_{<l}}\bb'),a_{<l}^{S_{<l}}\bb')=g(a_n^{S_l},\l(a_n;a_n^{S_l}a_{<l}^{S_{<l}}),\l(\bb;a_{<l}^{S_{<l}}\bb'),a_{<l}^{S_{<l}}\bb'),$$ as wanted.   
\end{proof}

Note that the sequence $(a_n^{S_l}\l(a_n;a_n^{S_l}a_{<l}^{S_{<l}}))_{n\geq l}$ remains indiscernible. 
The following Lemma is inspired by Statement 1 of \cite[Proposition 2.27]{at} and by \cite[Theorem 3.2]{h}. We chose to phrase it in a way that is convenient for our purposes, though the proof is the same as those of the cited statements.  

\begin{lema}
\label{prep}
    For each $k\geq1,$ let $Q_k$ be the set of $p$-independent $k$-tuples of $M.$ 
    \begin{enumerate}[label*={\arabic*.}]
        \item If $k\geq 1$ and $f_1,\.,f_k$ are rational functions over $\F_p,$ then there is a quantifier-free $\LL_{ring}$-formula $\phi$ such that $(f_1(a_n,\bb),\.,f_k(a_n,\bb))\in Q_k$ if and only if $$\MM\models\phi(a_n^{S_l},\l(a_n;a_n^{S_l}a_{<l}^{S_{<l}}),\l(\bb;a_{<l}^{S_{<l}}\bb'),a_{<l}^{S_{<l}}\bb')$$ for all $n\geq l.$

        \item If $k\geq 1,m\in\Mon(k)$ and $f_0,\.,f_k$ are rational functions over $\F_p,$ then there is some rational function $f$ over $\F_p$ such that $$\l_m(f_0(a_n,\bb);f_1(a_n;\bb),\.,f_k(a_n,\bb))=f(a_n^{S_l},\l(a_n;a_n^{S_l}a_{<l}^{S_{<l}}),\l(\bb;a_{<l}^{S_{<l}}\bb'),a_{<l}^{S_{<l}}\bb')$$ for all sufficiently large $n\geq l.$
        
    \end{enumerate}
\end{lema}
\begin{proof}
Before we start, given some $\mu\geq1,$ let $(m_i)_{i\in[p^{\mu}]}$ be a fix enumeration of $\Mon(\mu).$ If $\nu\leq\mu,$ we may assume that $(m_i)_{i\in[p^{\nu}]}$ is an enumeration of $\Mon(\nu)$. 
    \begin{enumerate}[wide, label*={\arabic*.}]
        \item Let $\nu=|S_0|+\.+|S_l|+|\bb'|,$ let $n\geq l$ and let $\mathbf{f}(a_n,\bb)$ be the tuple $(f_1(a_n,\bb),\.,f_k(a_n,\bb)).$ For an index $j\in[p^{\nu}],$ the element $\mathbf{f}(a_n,\bb)^{m_j}$ belongs to $M^{p}(a_n^{S_l}a_{<l}^{S_{<l}}\bb'),$ 
         which allows us to write $$\mathbf{f}(a_n,\bb)^{m_j}=\sum_{i=1}^{p^{\nu}}\left(\l_{m_i}(\mathbf{f}(a_n,\bb)^{m_j};a_n^{S_l}a_{<l}^{S_{<l}}\bb')\right)^{p}\cdot \left(a_n^{S_l}a_{<l}^{S_{<l}}\bb'\right)^{m_i}.$$
         Note that if $k>\nu,$ then there are no $p$-independent $k$-tuples of elements of $M^{p}(a_n^{S_l}a_{<l}^{S_{<l}}\bb'),$ in which case we can define $\phi$ to be $\bot.$ 
         If $k\leq\nu,$ and if $A_n$ and $B_n$ are the $p^{\nu}\times p^{k}$ matrices over $M$ given by $(A_n)_{i,j}=\l_{m_i}(\mathbf{f}(a_n,\bb)^{m_j};a_n^{S_l}a_{<l}^{S_{<l}}\bb')$ and $(B_n)_{i,j}=(A_n)_{i,j}^{p},$ then
        \begin{align*}
            \mathbf{f}(a_n,\bb)\in Q_k&\Sii \mathbf{f}(a_n,\bb)\text{ is }p\text{-independent in }M\\
            &\Sii\{\mathbf{f}(a_n,\bb)^{m_1},\.,\mathbf{f}(a_n,\bb)^{m_{p^{k}}}\}\text{ is linearly independent over }M^{p}\\
            &\Sii B_n\text{ has rank }p^{k}\\
            &\Sii B_n\text{ admits some non-zero minor of order }p^{k}\\
            &\Sii A_n\text{ admits some non-zero minor of order }p^{k}.
        \end{align*}
        Let $x=(x_{i,j})_{(i,j)\in p^{\nu}\times p^{k}}$ and let $\psi(x)$ be the $\LL_{ring}$-formula stating that the matrix $x$ has a non-zero minor of order $p^{k}.$ By Corollary \ref{ltorat}, for each $(i,j)\in p^{\nu}\times p^{k}$ there is some rational function $g_{i,j}$ over $\F_p,$ not depending on $n,$ such that 
        $$\l_{m_i}(\mathbf{f}(a_n,\bb)^{m_j};a_n^{S_l}a_{<l}^{S_{<l}}\bb')=g_{i,j}(a_{n}^{S_l},\l(a_n;a_n^{S_l}a_{<l}^{S_{<l}}),\l(\bb;a_{<l}^{S_{<l}}\bb'),a_{<l}^{S_{<l}}\bb'),$$ making $(A_n)_{i,j}=g_{i,j}(a_{n}^{S_l},\l(a_n;a_n^{S_l}a_{<l}^{S_{<l}}),\l(\bb;a_{<l}^{S_{<l}}\bb'),a_{<l}^{S_{<l}}\bb')$ for all $(i,j)\in p^{\nu}\times p^{k}.$ Therefore, 
        $$\mathbf{f}(a_n,\bb)\in Q_k\Sii\MM\models\psi(A_n).$$ If $w$ is a suitable tuple of variables and if $g$ is the matrix $(g_{i,j})_{(i,j)\in p^{\nu}\times p^{k}},$ then we can put $\phi(w)=\psi(g(w)),$ which is (equivalent to, up to clearing denominators) a quantifier-free $\LL_{ring}$-formula.

        \item Since the sequence $(a_n^{S_l}\l(a_n;a_n^{S_l}a_{<l}^{S_{<l}}))_{n\geq l}$ is indiscernible and $\phi$ is a quantifier-free $\LL_{ring}$-formula, the truth value of $(f_1(a_n,\bb),\.,f_k(a_n,\bb))\in Q_k\wedge(f_0(a_n,\bb),\.,f_k(a_n,\bb))\not\in Q_{k+1}$ as a function of $n\geq l$ is eventually constant.\footnote{This holds because $\texttt{ACVF}$ is NIP.} If such a truth value is false, then $$\l_m(f_0(a_n,\bb);f_1(a_n;\bb),\.,f_k(a_n,\bb))=0$$ for all sufficiently large $n\geq l,$ in which case we can put $f(a_{n}^{S_l},\l(a_n;a_n^{S_l}a_{<l}^{S_{<l}}),\l(\bb;a_{<l}^{S_{<l}}\bb'),a_{<l}^{S_{<l}}\bb')=0.$ If such a truth value is true, then $f_0(a_n,\bb)\in M^{p}(f_1(a_n;\bb),\.,f_k(a_n,\bb))$ for all sufficiently large $n\geq l,$ by the Exchange Principle. If $A_n$ is the $p^{\nu}\times p^{k}$ matrix of the last statement, then $A_n$ has full rank and, comparing the coordinates of $f_0(a_n,\bb)$ with respect to $f_1(a_n;\bb),\.,f_k(a_n,\bb)$ and with respect to $a_n^{S_l}a_{<l}^{S_{<l}}\bb',$ we get that  $$A_n\cdot\Big[\l_{m_{i}}(f_0(a_n,\bb);f_1(a_n;\bb),\.,f_k(a_n,\bb))\Big]_{i\in[p^{k}]}^{t}=\Big[\l_{m_{i}}(f_0(a_n,\bb);a_n^{S_l}a_{<l}^{S_{<l}}\bb')\Big]_{i\in[p^{\nu}]}^{t}.$$   
        Let $\{A_n^{\mu}:\mu\in[\binom{p^{\nu}}{p^{k}}]\}$ be an enumeration of the set of $p^{k}\times p^{k}$ sub-matrices of $A_n$. We claim that there is some $\mu\in[\binom{p^{\nu}}{p^{k}}]$ for which the determinant of $A_n^{\mu}$ is non-zero for all sufficiently large $n\geq l.$ Indeed, it is enough to check by induction on $\mu\in[\binom{p^{\nu}}{p^{k}}]$ if the truth value of the quantifier-free $\LL_{ring}$-formula $\det(A_n^{\mu})\neq 0$ is true for all sufficiently large $n\geq l.$ This procedure must find some suitable $\mu,$ because $A_n$ admits some non-zero minor of order $p^{k}$ for all sufficiently large $n\geq l.$   
        By Lemma \ref{la}, we get that $(A_n)^{\mu}_{L}$ is a left inverse of $A_n$ and that there are rational functions $(h_{i,j})_{(i,j)\in p^{k}\times p^{\nu}}$ over $\F_p$ such that $((A_n)^{\mu}_{L})_{i,j}=h_{i,j}(A_n)$ for all $n\geq l.$ Thus, 
        $$\Big[\l_{m_{i}}(f_0(a_n,\bb);f_1(a_n;\bb),\.,f_k(a_n,\bb))\Big]_{i\in[p^{k}]}^{t}=(A_n)^{\mu}_{L}\cdot\Big[\l_{m_{i}}(f_0(a_n,\bb);a_n^{S_l}a_{<l}^{S_{<l}}\bb')\Big]_{i\in[p^{\nu}]}^{t}$$ 
        and, in particular, if $m$ corresponds to $m_i$ for some unique $i\in[p^{k}],$ if $j\in[p^{\nu}]$ and if $g_j$ is a rational function over $\F_p$ such that $\l_{m_j}(f_0(a_n,\bb);a_n^{S_l}a_{<l}^{S_{<l}}\bb')=g_j(a_{n}^{S_l},\l(a_n;a_n^{S_l}a_{<l}^{S_{<l}}),\l(\bb;a_{<l}^{S_{<l}}\bb'),a_{<l}^{S_{<l}}\bb')$ for all $n\geq l,$ which exists by Corollary \ref{ltorat}, we get that 
        \begin{align*}
            \l_{m}(f_0(a_n,\bb);f_1(a_n,\bb),\.,f_{k}(a_n,\bb))&=\sum_{j=1}^{p^{\nu}}((A_n)^{\mu}_{L})_{i,j}\cdot\l_{m_j}(f_0(a_n,\bb);a_n^{S_l}a_{<l}^{S_{<l}}\bb')\\
            &=\sum_{j=1}^{p^{\nu}}h_{i,j}(A_n)\cdot g_j(a_{n}^{S_l},\l(a_n;a_n^{S_l}a_{<l}^{S_{<l}}),\l(\bb;a_{<l}^{S_{<l}}\bb'),a_{<l}^{S_{<l}}\bb')\\
            &=:f(a_{n}^{S_l},\l(a_n;a_n^{S_l}a_{<l}^{S_{<l}}),\l(\bb;a_{<l}^{S_{<l}}\bb'),a_{<l}^{S_{<l}}\bb')
        \end{align*}
        for all sufficiently large $n\geq l.$
        \qedhere
    \end{enumerate}
\end{proof}

The following Corollary is a summary of what we have proved thus far.

\begin{cor}
\label{indprep}
    Let $a_{<\omega}$ be an indiscernible sequence of $N$-tuples of some $\l$-field $\MM$ of characteristic exponent $p,$ and let $\bb\subseteq M$ be a finite tuple of parameters. Then there is some $\subseteq$-decreasing sequence $(S_n)_{n<\omega}$ of subsets of $[N],$ some subset $\bb'$ of $\bb$ and some $l<\omega$ such that 
    \begin{enumerate}[label*={\arabic*.}]
        \item $S_n=S_l$ for all $n\geq l,$ 

        \item $a_n^{S_l}a_{<l}^{S_{<l}}\bb'$ is $p$-independent in $M,$ $a_n\subseteq M^{p}(a_n^{S_l}a_{<l}^{S_{<l}})$ and $\bb\subseteq M^{p}(a_{<l}^{S_{<l}}\bb')$ for all $n\geq l,$ and

        \item If $k\geq 1,m\in\Mon(k)$ and $f_0,\.,f_k$ are rational functions over $\F_p,$ then there is some rational function $f$ over $\F_p$ such that $$\l_m(f_0(a_n,\bb);f_1(a_n;\bb),\.,f_k(a_n,\bb))=f(a_n^{S_l},\l(a_n;a_n^{S_l}a_{<l}^{S_{<l}}),\l(\bb;a_{<l}^{S_{<l}}\bb'),a_{<l}^{S_{<l}}\bb')$$ for all sufficiently large $n\geq l.$
    \end{enumerate}
\end{cor}

\begin{lema}[Indiscernible $\l$-Resolution]
\label{indtra}
    Let $N\geq 1,$ let $a_{<\omega}$ and $\bb$ be as above. For any $\LL_{\l}$-term $t(x,y),$ there is some rational function $f$ over $\F_p,$ some $l<\omega,$ some indiscernible sequence $\a_{\geq l}$ and some tuple $d\subseteq M$ of parameters such that
    \begin{enumerate}[label*={\arabic*.}]
        \item For all $k\in[|\a_l|],$ there is some 0-definable function $g_k$ such that $\a_n^{k}=g_k(a_n,a_{<l})$ for all $n\geq l,$ and

        \item $t(a_n,\bb)=f(\a_n,d)$ for all sufficiently large $n\geq l.$  
    \end{enumerate}
    
\end{lema}
\begin{proof}
    We proceed by induction on $t(x,y).$ If a sequence $\a_{<\omega}$ satisfies the statement of Item 1, we write $\a_{n}\subseteq^{u}\dcl(a_n,a_{<l})$ to stress the uniformity over $n\geq l$ of the corresponding definable functions.
    \begin{itemize}[wide]
        \item \emph{$t(x,y)=P(x,y)$ where $P$ is a polynomial over the base field $\F_p.$} In this case we let $f=P,$ $l=0,$ $\a_{\geq l}=a_{<\omega}$ and $d=\bb.$

        \item \emph{$t(x,y)=\l_{m}(x_0;x_1,\.,x_n)$ for some $n\geq 1$, some $m\in\Mon(n)$ and some splitting $(x,y)$ of $(x_0,\.,x_n).$} The result follows from Corollary \ref{indprep} applied to $a_{<\omega},$ $\bb$ and $f_i(x,y)=x_i$ for $i\leq n,$ where $\a_n=(a_n^{S_l}\l(a_n;a_n^{S_l}a_{<l}^{S_{<l}}))$ and $d=(\l(\bb;a_{<l}^{S_{<l}}\bb'),a_{<l}^{S_{<l}}\bb').$

        \item \emph{$t(x,y)=t_1(x,y)\diamond t_2(x,y)$ where $\diamond\in\{+,\cdot,-\}$ and $t_1,t_2$ are two $\LL_{\l}$-terms.} By induction hypothesis, there are two numbers $l_1,l_2<\omega,$ two tuples of parameters $d_1,d_2$ of $M,$ two rational functions $f_1,f_2$ over $\F_p$ and two indiscernible sequences $\b_{<\omega}^{(1)},\b_{<\omega}^{(2)}$ such that $\b_n^{(1)}\subseteq^{u}\dcl(a_na_{<l_1}),$ $\b_n^{(2)}\subseteq^{u}\dcl(a_na_{<l_2}),$ $t_1(a_n,\bb)=f_1(\b_n^{(1)},d_1)$ and $t_2(a_n,c)=f_2(\b_n^{(2)},d_2)$ for all sufficiently large $n\geq\max\{l_1,l_2\}.$\footnote{The superscripts $(1),(2)$ are just indices.} If we put $l=\max\{l_1,l_2\},$ $\a_n=(\b_n^{(1)},\b_n^{(2)}),$ $d=(d_1,d_2)$ and $f=f_1\diamond f_2,$ we get that $t(a_i,c)=f(\a_n,d)$ for all sufficiently large $n\geq l,$ that $\a_n\subseteq^{u}\dcl(a_na_{<l})$ and that $\a_{\geq l}$ is indiscernible. 

        \item \emph{$t(x,y)=\l_{m}(t_0(x,y);t_1(x,y),\.,t_n(x,y))$ for some $n\geq 1,$ some $m\in\Mon(n)$ and some $\LL_\l$-terms $t_0(x,y),\.,t_n(x,y).$} If $j\leq n,$ let $(\b_{<\omega}^{(j)},l_j,d_j,f_j)$ be the data obtained by the induction hypothesis applied to $t_j,$ i.e.~the data for which $t_j(a_i,c)=f_j(\b_i^{(j)},d_j)$ and $\b_i^{(j)}\subseteq^{u}\dcl(a_ia_{<l_{j}})$ for all sufficiently large $i\geq l_j.$ 
        Define $d=(d_j)_{j\leq n},$ $l'=\max\{l_j:j\leq n\}$ and $\b_{i}'=(\b_i^{(j)})_{j\leq n},$ so that the sequence $\b_{\geq l}'$ remains indiscernible. The result follows from Corollary \ref{indprep} applied to the sequence $\b_{\geq l}',$ the tuple of parameters $d$ and the rational functions $f_0,\.,f_n.$ 
        \qedhere
    \end{itemize}
\end{proof}

The following result was already proved by Hong for finite Ershov degree $\l$-fields, cf.~{\cite[Corollary 5.2.13]{ht}}. The case of infinite Ershov degree $\l$-fields follows from Delon's treatment of 1-types in $\texttt{SCVF}_{\infty}$, who set the ground for a Poizat-coheir-counting argument. As mentioned in the introduction, such a result, to the best of the author's knowledge, remains unpublished.

\begin{teorema}
\label{lqfnip}
    Let $t_1(x,y),t_2(x,y)$ be two $\LL_\l$-terms. Then the atomic $\LL_{div,\l}$-formulas $t_1(x,y)=0$ and $v(t_1(x,y))\square v(t_2(x,y)),$ with $\square\in\{>,=\},$ are NIP with respect to any theory of valued $\l$-fields. In particular, quantifier-free $\LL_{div,\l}$-formulas are NIP in any theory of valued $\l$-fields, and the $\LL_{div,\l}$-theory $\texttt{SCVF}_e$ is NIP for any Ershov degree $e\in\N\cup\{\infty\}.$
\end{teorema}
\begin{proof}
    Let $\MM$ be a valued $\l$-field of characteristic exponent $p.$ Let $N\geq 1,$ let $a_{<\omega}$ be an indiscernible sequence of $|x|$-tuples of $M$ and let $\bb$ be a $|y|$-tuple of $M.$ By Lemma \ref{indtra}, there is some number $l<\omega,$ and for each $j\in\{1,2\},$ a rational function $f_j$ over $\F_p,$ some finite tuple $d_j$ of elements of $\MM$ and some indiscernible sequence $\a_{\geq l}^{(1)},\a_{\geq l}^{(2)}$ such that $t_j(a_n,\bb)=f_j(\a_n^{(j)},d_j)$ for all $n\geq l$ and $\a_n^{(j)}\subseteq^{u}\dcl(a_na_{<l})$ for all sufficiently large $n\geq l.$ 
    If $f_j=P_j/Q_j$ for some polynomials $P_j,Q_j$ over $\F_p,$ then the truth value of $t_1(x,\bb)=0$ or $v(t_1(x,\bb))\square v(t_2(x,\bb))$ alternates infinitely often with respect to the indiscernible sequence 
    $a_{<\omega}$ if and only if the truth value of $P_1(x',d_1)=0$ or $v(P_1(x',d_1)\cdot Q_2(x'',d_2))\square v(P_2(x'',d_2)\cdot Q_1(x',d_1))$ alternates infinitely often with respect to the indiscernible sequence 
    $(\a_n^{(1)},\a_{n}^{(2)})_{n\geq l},$ where $x'$ and $x''$ are tuples of variables of the same length as those of $\a_{\geq l}^{(1)}$ and $\a_{\geq l}^{(2)}$ respectively. In either case, this implies that the quantifier-free $\LL_{div}$-formula $\f(x',y')$ given by $P_1(x',y')=0$ has IP, or that the formula $\f(x',x'';y'z')$ given by $v(P_1(x',y')\cdot Q_2(x'',z'))\square v(P_2(x'',z')\cdot Q_1(x',y'))$ has IP, in contradiction with the fact that $\texttt{ACVF}$ is NIP. 
    
    Since a boolean combination of NIP formulas is still a NIP formula, we get that any quantifier-free $\LL_{div,\l}$-formula is NIP, as they are boolean combinations of formulas of the form $t_1(x,y)=0$ and $v(t_1(x,y))\square v(t_2(x,y)),$ with $\square\in\{>,=\}.$ 
    Finally, if $e\in\N\cup\{\infty\},$ then the $\LL_{div,\l}$-theory $\texttt{SCVF}_e$ is NIP because, by Fact \ref{scvfqe}, any formula is equivalent to a quantifier-free $\LL_{div,\l}$-formula.
\end{proof}

\section{NIP Transfer Revisited}
\label{4}

In this section we revisit the so-called \emph{NIP transfer principle} for valued fields of equal characteristic,~\cite[Proposition 4.1]{aj}, building on Sections~2 and~3 of~\cite{js}. As mentioned in the introduction, although the statement of the transfer principle is correct, we found that the proofs given in~\cite[Lemma 3.2]{js} and~\cite[Proposition 4.1]{aj} are not accurate as stated, and that a small but consequential modification of the argument is required.
We study this transfer principle in a language expanded by the parameterized $\l$-functions, and apply the model theoretic NIP transfer of Proposition \ref{niptransfer} in this setting. We record an updated transfer principle in Corollary~\ref{fulltrans}, and give a corrected and self-contained proof of~\cite[Lemma 3.2]{js} and~\cite[Proposition 4.1]{aj} in Lemma~\ref{e+im}. 

\begin{remark}
\label{gaps}
    The issue in the proof of \cite[Lemma 3.2]{js} is the following. Let $(M,v)$ be an equi-characteristic separable-algebraically maximal Kaplansky field of characteristic exponent $p$ and \emph{finite} imperfection degree $e.$ Let $(K,v)\preceq(M,v)$ and let $b\in M$ be some element that is transcendental over $K$ and such that $(K(b)|K,v)$ is immediate. The authors proceed to find a separable extension $\widetilde{K}|K$ containing $b$ which is also separably algebraically maximal Kaplansky and which is uniquely determined up to $\LL$-isomorphism over $K(b).$ This extension does not necessarily embed elementarily \emph{over $K(b)$} into $M,$ so $\tp^{\widetilde{K}}(b/K)$ and $\tp^{M}(b/K)$ may differ: for instance, as constructed in Lemma \cite[Lemma 3.2]{js}, the maximal immediate algebraic extension $\widetilde{K}$ of $(K(b^{1/p^{\infty}}))^{h}$ contains all ${p^{n}}$-th roots of $b,$ while this may not be the case in $M$. Even if we considered $L=\widetilde{K}\cap M\cap K(b)^{sep}$ instead of $\widetilde{K}$ alone, 
    the extension $M|L$ is \emph{not} guaranteed to be separable: since $L|K(b)$ is algebraic, then $M|L$ is separable if and only if $M|K(b)$ is separable, which in turn is equivalent to $b\not\in M^{p}(B)$ for any $p$-basis $B$ of $K,$ see e.g.~\cite[Lemma 2.9]{fvkp}.
    This is the same reason why, in the proof of \cite[Proposition 4.1]{aj} for $e=\infty,$  the extension $M|L$ is not necessarily separable. In this case, $L$ is chosen to be $M\cap K(b)^{sep}\cap F,$ where $F$ is a maximal immediate algebraic extension of $K(b).$\footnote{One has to be careful with these notations. Anscombe-Jahnke's $K^{*}$ denotes what we called $M$ above, and $L$ denotes $K^{*}\cap K(b)^{sep}\cap M,$ where $M$ is a maximal immediate algebraic extension of $K(b),$ i.e.~Anscombe-Jahnke's $M$ plays the role of our $F$ above.} 
        
    In Lemma \ref{e+im} we fix this issue by studying the \emph{lambda closure $\langle K,b\rangle$ of $K(b)$ in $M$} instead of just $K(b),$ in order to make sure that $M|\langle K,b\rangle$ is separable.
\end{remark}

\subsection{Side Conditions}

In this section we will see that, for any $e\in\N\cup\{\infty\},$ the $\LL$-theory $\texttt{SAMK}_e^{\l,\ac}$ satisfies the side conditions, in the language $\LL=\LL(\LL_{\kk},\LL_{\GG})$ which consists of an expansion of $\LL_{3s}=\LL_{3s}(\LL_{\kk},\LL_{\GG})$ by symbols for the parameterized lambda functions of the home sort and a symbol for an angular component map. We will write $\LL_\l=\LL_{\l}(\LL_{\kk},\LL_{\GG})$ (resp.~$\LL_{\ac}=\LL_{\ac}(\LL_{\kk},\LL_{\GG})$) to denote the same language but excluding the symbol for the angular component map (resp.~the symbols for the parameterized lambda functions). In this theory, the side conditions of Definition \ref{sideconds} state that the residue field and the value group are stably embedded sets, and that if $\MM\succeq\NN$ are models of $\texttt{SAMK}_e^{\l,\ac},$ where $(M|N,v)$ is the corresponding extension of valued fields, and if $b\in M$ is such that the subextension $(\langle N,b\rangle|N,v)$ is immediate, then $\tp(b/N)$ is implied by a set of NIP formulas.  

\begin{lema}
\label{e+se}
    The $\LL$-theory $\texttt{SAMK}_e^{\l,\ac}$ satisfies \textbf{SE}. 
\end{lema}
\begin{proof}
    This is \cite[Corollary 5.3]{sm}. 
\end{proof}

The following lemma is well known.

\begin{lema}
\label{tamif}
    Let $(K,v)$ be a $\mathrm{SAMK}$ field, and let $w$ be an extension of $v$ to an algebraic closure $K^{alg}$ of $K.$ Then $(K^{1/p^{\infty}},w)$ is algebraically maximal and $(K^{1/p^{\infty}}|K,w)$ is immediate.
\end{lema}
\begin{proof}
    Since $(K,v)$ is $\mathrm{SAMK}$, then it is separably tame and therefore $(K^{1/p^{\infty}},w)$ is tame by \cite[Lemma 3.13]{fvktf}. By \cite[Theorem 3.2]{fvktf}, $(K^{1/p^{\infty}},w)$ is algebraically maximal. Now, $wK^{1/p^{\infty}}=\frac{1}{p^{\infty}}vK=vK$ and $K^{1/p^{\infty}}w=(Kv)^{1/p^{\infty}}=Kv$ by \cite[Lemma 2.1]{fvkp} and because $vK$ is $p$-divisible and $Kv$ is perfect, meaning that $(K^{1/p^{\infty}}|K,w)$ is immediate. 
\end{proof}

Let $(M,v)$ be a $\mathrm{SAMK}$ field, and let $K_a$ and $K_b$ be two subfields of $M$ admitting a valued field isomorphism $f:(K_a,v)\to(K_b,v).$ Moreover, suppose that $(K_a,v)$ is Kaplansky, that $M|K_a$ and $M|K_b$ are separable, and that there is an isomorphism of valued fields $f:(K_a,v)\to(K_b,v).$ It follows that $(K_b,v)$ is also Kaplansky and that both extensions $M|K_a^{rac}$ and $M|K_b^{rac}$ are regular, implying that $(K_a^{rac},v)$ and $(K_b^{rac},v)$ are $\mathrm{SAMK}$by \cite[Lemma 2.22]{sm}. If $(M^{alg},w)$ extends $(M,v),$ then by Lemma \ref{tamif}, $((K_a^{rac})^{1/p^{\infty}},w)$ and $((K_b^{rac})^{1/p^{\infty}},w)$ are algebraically maximal immediate extensions of $(K_a,v)$ and $(K_b,v)$ respectively. By Kaplansky's Uniqueness Theorem \cite[Theorem 5]{kap}, there is an isomorphism of valued fields $\Phi:((K_a^{rac})^{1/p^{\infty}},w)\to((K_b^{rac})^{1/p^{\infty}},w)$ whose restriction to $(K_a,v)$ coincides with $f.$\footnote{This is not exactly how Kaplansky's theorem is stated, because the isomorphism $\Phi$ is supposed to be defined over a \emph{common} subfield of $M$ and not over an isomorphism $f:K_a\to K_b$ of subfields of $M.$ Our version, which follows from his proof, is equivalent to his by taking $K_a=K_b$ and taking $f$ to be the identity.} 

\begin{claim} 
\label{restriso}
The image $\Phi(K_a^{rac})$ equals to $K_b^{rac}.$ As a consequence, if $\iota$ is the restriction of $\Phi$ to $K_a^{rac},$ then $\iota:(K_a^{rac},v)\to(K_b^{rac},v)$ is an isomorphism of valued fields.
\end{claim}
\begin{proof}
Let $m\in K_a^{rac}.$ Since $M|K_a$ is separable, the extension $K_a(m)|K_a$ is separable-algebraic. It follows that $K_b(\Phi(m))|K_b$ and therefore $K_b^{rac}(\Phi(m))|K_b^{rac}$ are separable-algebraic too. Since the latter extension is immediate, we get that $\Phi(m)\in K_b^{rac}$ because $K_b^{rac}$ is separable-algebraically maximal, i.e.~$\Phi(K_a^{rac})\subseteq K_b^{rac}$. Mutatis mutandis, we see that $\Phi^{-1}(K_b^{rac})\subseteq K_a^{rac},$ so the desired equality follows.
        \qedhere$_{\text{Claim}}$
\end{proof}
\begin{figure}
    
\begin{center}
    \begin{tikzcd}
                                                                              &                                                                    & M^{1/p^{\infty}}                                   &                                                 \\
                                                                              & M \arrow[ru, no head]                                              &                                                    &                                                 \\
                                                                              & (K_a^{rac})^{1/p^{\infty}} \arrow[rr, "\Phi"] \arrow[ruu, no head] &                                                    & (K_b^{rac})^{1/p^{\infty}} \arrow[luu, no head] \\
K_a^{rac} \arrow[ru, no head] \arrow[ruu, no head] \arrow[rr, "\iota", dotted] &                                                                    & K_b^{rac} \arrow[ru, no head] \arrow[luu, no head] &                                                 \\
                                                                              &                                                                    &                                                    &                                                 \\
K_a \arrow[rr, "f"] \arrow[uu, no head]                                       &                                                                    & K_b \arrow[uu, no head]                            &                                                
\end{tikzcd}
\end{center}

    \caption{The picture of Claim \ref{restriso}}
    \label{sttt}
\end{figure}
Having said this, we are ready to prove the promised correction of \textbf{Im}. Recall that the notation $\cl(A)=\langle A\rangle$ stands for the substructure generated by $A$ in some (any) model $\MM$ containing $A.$ This closure operator is finitary and invariant.

\begin{lema}
\label{e+im} 
    The $\LL$-theory $\texttt{SAMK}_{e}^{\l,\ac}$ satisfies $\langle\rangle$-\textbf{Im}.
\end{lema}

\begin{proof}
Let $\MM\models\texttt{SAMK}_e^{\l,\ac},$ let $(M,v)$ be its underlying valued field, let $K$ be the underlying field of an elementary substructure $\mathfrak{K}\preceq\MM,$ and let $a\in M$ be such that $(\langle K,a\rangle|K,v)$ is immediate. Let us fix an algebraic closure $M^{alg}$ of $M$ and an extension $w$ of $v$ to $M^{alg},$ so that all fields mentioned below are subfields of $M^{alg}.$ Without loss, we can assume that $\MM$ is $|K|^{+}$-saturated. We will show that $\tp^{\MM}(a/K)$ is implied by $\qftp^{\MM}_{\LL_{\l}}(a/K).$ 
This is sufficient to prove \textbf{Im} because quantifier-free $\LL_\l$-formulas are NIP by Theorem \ref{lqfnip}. 
The type $\qftp^{\MM}_{\LL_{\l}}(a/K)$ determines the $\LL_{\l}$-isomorphism type of $\langle K,a\rangle$ over $K,$ i.e.~if $a,b\in M$ and $a\equiv_{K}b$ with respect to $\LL_{\l},$ then $\langle K,a\rangle$ is $\LL_{\l}$-isomorphic to $\langle K,b\rangle$ over $K.$ Since $(\langle K,a\rangle|K,v)$ is immediate, the angular component map of $K$ extends uniquely to $\langle K,a\rangle,$ so $\langle K,a\rangle$ and $\langle K,b\rangle$ are even $\LL$-isomorphic over $K,$ say, by an embedding $f:\langle K,a\rangle\to\langle K,b\rangle$ that fixes $K$ and that takes $a$ to $b.$ Our goal is to find, for such a pair $(a,b),$ a pair of $\LL$-substructures $\NN_a,\NN_b\subseteq\MM$ and an $\LL$-isomorphism $\iota:\NN_a\to\NN_b$ such that 
        $N_a\supseteq\langle K,a\rangle,$ $N_b\supseteq \langle K,b\rangle,$
        $\NN_a\preceq\MM,$ $\NN_b\preceq\MM,$
        $\iota(a)=b$ and $\iota(k)=k$ for all $k\in K.$
    This is sufficient for what we want, because if $\f(x,y)$ is an $\LL$-formula and $k$ is any $|y|$-tuple from $K,$ then $\MM\models\f(a,k)\Sii\NN_a\models\f(a,k)\Sii\NN_b\models\f(b,k)\Sii\MM\models\f(b,k),$ i.e.~$\tp^{\MM}(a/K)=\tp^{\MM}(b/K).$ 
    Since $(\langle K,a\rangle|K,v)$ is immediate, $(\langle K,a\rangle,v)$ is still Kaplansky. Since $\mathfrak{K}\preceq\MM,$ the extension $M|K$ is regular: it is separable and $K$ is relatively algebraically closed in $M$. If $a$ is algebraic over $K,$ then $a$ lies in $K^{alg}\cap M=K,$ so the extension $(\langle K,a\rangle|K,v)$ must be trivial, yielding that $\tp^{\MM}(a/K)$ is implied by the NIP formula $x=a.$ 
    Now suppose that $a$ is transcendental over $K.$ 
    Let $N_a=\langle K,a\rangle^{rac}$ and $N_b=\langle K,b\rangle^{rac}$ be the relative algebraic closures of $\langle K,a\rangle$ and $\langle K,b\rangle$ in $M$ respectively.       
    
    The extension $M|N_a$ is separable if and only if $M|\langle K,a\rangle$ is separable, as $N_a|\langle K,a\rangle$ is algebraic. In fact, the extension $M|\langle K,a\rangle$ \emph{is} separable by Fact \ref{lcl}, from which it follows that $M|N_a$ is regular and that $(N_a,v)$ is $\mathrm{SAMK}$, by \cite[Lemma 2.22]{sm}.
    Since $M|N_a$ is separable, then the $\LL$-structure $\NN_a$ associated to $N_a$ is indeed an $\LL$-substructure of $\MM,$ as $N_a$ is lambda-closed in $M.$ This also implies that $\imdeg(M)\geq\imdeg(N_a).$ 
    In addition, since $M|K$ is separable, then $N_a|K$ is separable, so $\imdeg(N_a)\geq\imdeg(K).$
    Since $K\preceq M,$ we get that $\imdeg(K)=\imdeg(N_a)=\imdeg(M),$ so actually $\NN_a\models\texttt{SAMK}_e^{\l,\ac}.$ Note that $vN_a=vK$ and $N_av=Kv$ because the extension $(\langle K,a\rangle|K,v)$ is immediate, $Kv$ is relatively-algebraically closed in $Mv$ and $vN_a$ is relatively divisible in $vM.$ If $\kk$ and $\GG$ denote the residue field sort and the value group sort respectively, then $\kk(N_a)\preceq\kk(M)$ and $\GG(N_a)\preceq\GG(M).$ 
    By \cite[Corollary 5.1]{sm}, we get that $\NN_a\preceq\MM.$ Mutatis mutandis, we also obtain that $\NN_b\preceq\MM,$ where $\NN_b$ is the $\LL$-structure associated to $N_b.$ Following Claim \ref{restriso} and its preceding discussion, if we put $K_a=\langle K,a\rangle$ and $K_b=\langle K,b\rangle,$ we see that there is an $\LL_{3s}$-isomorphism $\iota:(N_a,v)\to(N_b,v)$ extending $f.$ Since angular component maps extend uniquely in immediate extensions, the embedding $\iota$ induces moreover an $\LL_{ac}$-isomorphism. Finally, since $M|\iota(N_a)$ is separable, then $\iota$ commutes with the parameterized $\l$-functions of $M$ as seen in \cite[Lemma 2.7]{sm}, meaning that $\iota$ determines a well defined $\LL$-isomorphism of $\NN_a$ and $\NN_b.$ The result follows.
    \qedhere    
\end{proof}

We can finally state the expected transfer principle.

\begin{cor}[Cf.~{\cite[Proposition 4.1]{aj}}]
\label{fulltrans}
    Suppose that $(M,v)$ is a henselian valued field of equal characteristic, seen as a structure in $\LL_{3s}=\LL_{3s}(\LL_{\kk},\LL_{\GG}).$ If $(M,v)$ is separably defectless Kaplansky, $Mv$ is NIP with respect to $\LL_{\kk}$ and $vM$ is NIP with respect to $\LL_{\GG},$ then $(M,v)$ is NIP with respect to $\LL_{3s}.$  
\end{cor}
\begin{proof}
    By the fundamental equality, if a valued field is henselian and separably defectless, then it is separably algebraically maximal. If $\MM$ is the $\LL_{\l}$-structure generated by $(M,v)$ and if $\NN\succeq\MM$ is $\aleph_1$-saturated, then $\NN$ admits an angular component map and thus the expansion $\NN^{\ac}$ satisfies $\NN^{\ac}\models\texttt{SAMK}_{e}^{\l,\ac},$ where $e\in\N\cup\{\infty\}$ equals the Ershov degree of $M.$ If $(N,w)$ is the underlying valued field of $\NN^{\ac},$ then the induced structure of the sets $wN$ and $Nw$ are NIP because they are interpretable in the $\LL_{\GG}$-structure $\GG(N)$ and the $\LL_{\kk}$-structure $\kk(N)$ respectively, which are NIP by hypothesis. 
    By Lemmas \ref{e+se} and \ref{e+im}, the hypotheses of Proposition \ref{niptransfer} are satisfied by the $\LL$-theory of $\NN^{\ac},$ implying that it is NIP, and so is its $\LL_{3s}$-theory. This last theory coincides with the $\LL_{3s}$-theory of $(M,v).$      
\end{proof}

\section{NIP Fields Revisited}
\label{5}

In this section we highlight some existential formulas from the literature that, under the assumption of being NIP, account for crucial algebraic properties of fields and valued fields. 

\subsection{Interpretations of Fields and Valued Fields}

We start with a well known remark.

\begin{remark}
\label{inte}
Let $L|K$ be a finite simple extension of fields. We will describe an interpretation of the field structure $(L,+^{L},\times^{L},0^{L},1^{L})$ in the structure $(K,+,\times,0,1),$ following Lemma \ref{intcrit}. 
If $L=K(\b)$ for some $\b\in L,$ let $\a=(\a_0,\.,\a_{n-1})\in K^{n}$ be the coefficients of the minimal polynomial of $\b$ over $K,$ so that $n=[L:K].$ If we define $\Phi:K^{n}\to L$ as $\Phi(a_0,\.,a_{n-1})=a_0+a_1\b+\.+a_{n-1}\b^{n-1},$ then the fact that $B=\{1,\b,\.,\b^{n-1}\}$ is a $K$-basis of $L$ implies that $\Phi$ is surjective. Also, $+^{L}$ is quantifier-free-definable over 0 and $\times^{L}$ is quantifier-free-definable over $\a.$ Indeed, $(x_1,\.,x_n)+^{L}(y_1,\.,y_n)=(z_1,\.,z_n)$ if and only if $\bigwedge_{i=1}^{n}x_i+y_i=z_i,$ and in $\Z[X_1,\.,X_n,Y_1,\.,Y_n,Z_1,\.,Z_n],$ there are $n$ polynomials $P_1,\.,P_{n}$ such that $(x_1,\.,x_n)\times^{L}(y_1,\.,y_n)=(z_1,\.,z_n)$ if and only if $\bigwedge_{i=1}^{n}z_i=P_i(x_1,\.,x_n,y_1,\.,y_n,\a_0,\.,\a_{n-1}).$ Since the latter interpretation was made with quantifier-free formulas, Statement 2 of Corollary \ref{eint} implies that this interpretation is $\mirrorexists_{n}$-to-$\mirrorexists_{n}$ and $A_n$-to-$A_n$ for all $n\geq 1.$  
\end{remark}

\begin{lema}
\label{finextint}
    Let $n\geq 1,$ let $K$ be an $\mirrorexists_{n}$-NIP field and let $L|K$ be a simple algebraic extension. Then the field $L$ is $\mirrorexists_{n}$-NIP too.
\end{lema}
\begin{proof}
    By Remark \ref{inte}, the field $L$ is $\mirrorexists_{n}$-to-$\mirrorexists_{n}$ interpretable in $K.$ Since $K$ is $\mirrorexists_{n}$-NIP, the result follows from Lemma \ref{dtod'nip}.    
\end{proof}

The celebrated theorem of Kaplan, Scanlon and Wagner for NIP fields \cite[Theorem 4.3]{ksw} shows that NIP fields are Artin-Schreier closed. In fact,

\begin{fact}[Cf.~{\cite[Theorem 4.3]{ksw}}]
    If $K$ is an $\E$-NIP field, then $K$ is Artin-Schreier closed.
\end{fact}
\begin{proof}
    Suppose $K$ is $\aleph_0$-saturated, and let $k=K^{p^{\infty}}=\bigcap_{n<\omega}K^{p^{n}}.$ Then $k$ is a type-definable infinite perfect subfield of $K.$ Consider the existential formula $\f(x,y)=\E z\,(x+yz=yz^{p})$ saying that $\f(K,a)=a\cdot\wp(K),$ where $\wp:K\to K$ is the Artin-Schreier map $\wp(y)=y^{p}-y.$ By Baldwin-Saxl's Theorem \cite[Theorem 2.13]{simon}, since $\f(x,y)$ is NIP, there is some $n<\omega$ such that for any $(n+1)$-tuple $\ol{a}$ of $K$ there is some $n$-subtuple $\ol{a}'$ of $\ol{a}$ satisfying $\bigcap_{a\in\ol{a}}\f(K,a)=\bigcap_{a\in\ol{a}'}\f(K,a).$
    The proof continues exactly as in the proof of \cite[Theorem 4.3]{ksw}. 
\end{proof}

Also, Duret's \cite[Theorem 6.4]{duret} says that every relatively algebraically closed, pseudo-algebraically closed subfield of a NIP field is separably closed. Recall that a field $K$ is called \emph{pseudo-algebraically closed} if every plane curve defined over $K$ has a $K$-rational point, or equivalently, if $K$ is existentially closed in any of its regular extensions, cf.~\cite[Proposition 11.3.5]{fj}. In fact,

\begin{fact}[Cf.~{\cite[Theorem 6.4]{duret}}]
\label{dur}
    Every relatively algebraically closed, pseudo-algebraically closed subfield of an $\E$-NIP field $K$ is separably closed. 
\end{fact}
\begin{proof}
     If $K$ admits a relatively algebraically closed, pseudo-algebraically closed subfield, then \cite[Theorem 6.4]{duret} shows that there is some prime $p$ and some finite simple extension $L$ of $K$ in which one of the formulas $\E z(z^{p}=x+y)$ or $\E z(z^{p}=xy+z)$ interpret the Random Graph, hence they have IP in $L$. If the formulas $\phi_1$ and $\psi_1$ are the respective interpretations in $K$ of such formulas, then following the proof of \cite[Theorem 6.4]{duret} ---or Lemma \ref{finextint}--- we see that $\phi_1$ and $\psi_1$ are also existential, transferring IP from $L$ to $K.$     
\end{proof}

This allows us to restate \cite[Corollary 4.5]{ksw} as follows.

\begin{cor}[Cf.~{\cite[Corollary 4.5]{ksw}}]
\label{enipfpalg}
    Let $K$ be an infinite $\E$-NIP field of characteristic $p>0.$ Then $K$ contains $\F_p^{alg}.$  
\end{cor}
\begin{proof}
    The algebraic part of $K$ is Artin-Schreier closed, hence infinite, hence pseudo-algebraically closed. By Fact \ref{dur}, it is separably closed.
\end{proof}

We can also make Corollary 4.4 of \cite{ksw} more explicit as follows.

\begin{cor}[Cf.~{\cite[Corollary 4.4]{ksw}}]
\label{sepextp}
    Let $K$ be an $\E$-NIP field of characteristic $p>0$ and let $L|K$ be a finite separable extension. Then $p$ does not divide $[L:K].$
\end{cor}
\begin{proof}
    Suppose otherwise. We may assume that $L|K$ is Galois. Let $K^{G}$ be the fix field of some subgroup $G\leq\Gal(L|K)$ of order $p.$ Then $K^{G}|K$ is a finite separable extension, implying that it is simple. It follows that $K^{G}$ is $\E$-NIP by Lemma \ref{finextint}, hence Artin-Schreier closed, contradicting that $L|K^{G}$ is an Artin-Schreier extension.
\end{proof}

Duret's Theorem (Fact \ref{dur}) implies that every pseudo-algebraically closed non-separably closed field has the independence property, witnessed by an existential formula. Examples of such fields are \emph{pseudo-finite} fields, cf.~\cite[Lemma 2]{ax}. Pseudo-algebraically closed fields $k$ are \emph{large} in the sense that $k\preceq_\E\hs{k}{}$, cf.~\cite[Proposition 11.3.5 and Remark 16.12.3]{fj}. Equivalently, a field $k$ is large if and only if for every smooth $k$-curve $C,$ if $C(k)$ is non-empty, then it is infinite, cf.~\cite[Proposition 1.1]{pop}.
For such fields $k$ we have the following existentially-closed embeddings.

\begin{fact}
\label{largeecl}
    Let $k$ be a large field. Then $k\preceq_{\E} K$ for every extension $K|k$ that admits a \emph{(not necessarily henselian)} valuation $v$ with $Kv=k$ in case 
    $k$ is perfect or $k$ is pseudo-algebraically closed.
\end{fact}
See \cite[Theorem 17]{fvk17} for the case in which $k$ is perfect. On the other hand, if $k$ is pseudo-algebraically closed, then $k$ is large, cf.~\cite[Remark 16.12.3]{fj}. The result follows by \cite[Lemma 9]{feh} and \cite[Proposition 2.3]{afd}.
The following corollary has the advantage of including valued fields that are not necessarily henselian. 

\begin{cor}
    Let $(K,v)$ be a valued field that extends $Kv$ and such that $Kv$ is pseudo-algebraically closed non-separably closed. Then $K$ has the independence property with respect to an existential $\LL_{ring}$-formula.
\end{cor}

\begin{proof}
    By Fact \ref{largeecl} we have that $Kv\preceq_\E K$ with respect to $\LL_{ring}.$ Duret's Fact \ref{dur} implies that $Kv$ has IP with respect to an existential $\LL_{ring}$-formula, and so does $K$ by Lemma \ref{dclosedip}.
\end{proof}

Now we turn our attention to valued fields.

\begin{lema}
\label{resfieldint}
    Let $n\geq 1$ and let $\LL_{3s}=\LL_{3s}(\LL_\kk,\LL_\GG).$ If the $\LL_{3s}$-structure associated to $(K,v)$ is $\mirrorexists_{n}$-NIP, then the $\LL_{\kk}$-structure generated by $Kv$ and the $\LL_{\GG}$-structure generated by $vK$ are $\mirrorexists_{n}$-NIP.
\end{lema}
\begin{proof}
    Recall Corollary \ref{eint}. The highlighted sets from Lemma \ref{intcrit}, with respect to the natural interpretations of the $\LL_{\kk}$-structure associated to $Kv$ and of the $\LL_{\GG}$-structure associated to $vK$ in the $\LL$-structure generated by $(K,v),$ are quantifier-free definable. By Corollary \ref{eint}, both interpretations are $\mirrorexists_{n}$-to-$\mirrorexists_{n}.$ If the $\LL$-structure associated to $(K,v)$ is $\mirrorexists_{n}$-NIP, then the result follows from Lemma \ref{dtod'nip}.
\end{proof}

We highlight the following corollaries, which are minor modifications of Propositions 5.3 and 5.4 of \cite{ksw}.

\begin{cor}[Cf.~{\cite[Proposition 5.3 and 5.4]{ksw}}] 
\label{vgpdiv}
    Let $(K,v)$ is a valued field of characteristic $p>0,$ and suppose its associated $\LL$-structure is $\E$-NIP. Then $Kv$ contains $\F_p^{alg}$ and
    $vK$ is $p$-divisible.
\end{cor}
\begin{proof}
    First, if the $\LL_{3s}$-structure associated to $(K,v)$ is $\E$-NIP, then $K$ is an $\E$-NIP field and therefore Artin-Schreier closed. Then $vK$ is $p$-divisible and $Kv$ is Artin-Schreier closed too, hence infinite. Since the $\LL_{\kk}$-structure generated by $Kv$ is $\E$-NIP by Lemma \ref{resfieldint}, we get that $Kv$ is an $\E$-NIP field too. The result follows by Corollary \ref{enipfpalg}. 
\end{proof}

We are also able to restate Proposition 4.1 of \cite{js}.
\begin{lema}[Cf.~{\cite[Proposition 4.1]{js}}]
\label{kapl}
    If $(K,v)$ is an $\E$-NIP non-trivially valued field of characteristic $p>0,$ then $(K,v)$ is Kaplansky.
\end{lema}
\begin{proof}
    First, $vK$ is $p$-divisible as seen in Corollary \ref{vgpdiv}. 
    Second, $Kv$ is perfect. To see this, let $\a=\res(a)\in Kv^{\times},$ $c\in\M$ and consider the polynomial $P(X)=X^{p}+cX-a\in\O[X].$ If $P(X)$ does not have a root in $K,$ then $K[X]/(P)$ would yield a proper separable algebraic extension of degree $p,$ implying by Lemma \ref{sepextp} that neither $K$ nor $(K,v)$ can be $\E$-NIP. Therefore there is some root of $P$ in $K.$ As $\O$ is integrally closed, such a root is an element of $\O.$ Taking $\res,$ we obtain a root in $Kv,$ which would be a $p^{th}$ root of $\a.$    
    Finally, $Kv$ does not admit any proper separable algebraic extension of degree divisible by $p.$ Indeed, since $Kv$ is an $\E$-NIP field by Lemma \ref{resfieldint}, the result follows from Corollary \ref{sepextp}.
\end{proof}

We can also restate Proposition 3.1 of \cite{aj}.

\begin{lema}[Cf.~{\cite[Proposition 3.1]{aj}}]
\label{sd}
    If $(K,v)$ is an equi-characteristic $\E$-NIP valued field, then $(K,v)$ is trivial or separably defectless Kaplansky.
\end{lema}
\begin{proof}
    We may assume that the characteristic of $K$ is $p>0.$ Since $(K,v)$ is Kaplansky by Lemma \ref{kapl}, we only need to show that $(K,v)$ is separably defectless. By \cite[Theorem 18.2]{oe}, $(K,v)$ is separably defectless if and only if any henselization $(K^{h},v^{h})$ is separably defectless. Let $L|K^{h}$ be a separable finite extension. Since $K^{h}|K$ is separable algebraic, $L|K$ is also separable algebraic. If $L|K^{h}$ has defect $d\geq1$, then
    $$[L:K^{h}]=p^{d}\cdot(wL:v^{h}K^{h})\cdot[Lw:K^{h}v^{h}]$$ by the fundamental equality and henselianity, where $w$ is the unique extension of $v^{h}$ to $L.$ This shows that $p$ divides $[L:K^{h}].$ Thus, there is some finite extension $F|K$ for which $p$ divides $[F:K],$ implying by Corollary \ref{sepextp} that neither $K$ nor $(K,v)$ can be $\E$-NIP.
\end{proof}

\subsection{Henselianity in Positive Characteristic}

In this section, we will study (multi-) valued fields using the language of rings expanded by predicate symbols to be interpreted as the valuation rings of the corresponding valuations. 

The following Lemma is folklore. We present a proof in this document for the sake of completeness, which uses a more basic instance of \cite[Proposition 4.8]{dp5}.

\begin{lema}[Cf.~{\cite[Proposition 4.8]{dp5}}]
\label{finextdef}
    Let $L|K$ be a finite extension and let $\O$ be a valuation ring of $K.$ 
    \begin{enumerate}[label*={\arabic*.}]
        \item Let $\widetilde{\O}$ be the integral closure of $\O$ in $L.$ If $a\in \widetilde{\O},$ then the minimal polynomial $\mu_{a,K}$ of $a$ over $K$ is an element of $\O[X].$
        
        \item The integral closure $\widetilde{\O}$ is $\emptyset$-definable in the structure $(L,K,\O).$

        \item Let $\{\p_1,\.,\p_r\}$ be the set of maximal ideals of $\widetilde{\O}.$ Then there are some $\{a_S:S\subseteq\{2,\.,r\}\}\subseteq \p_1$ such that $a_S\in\p_i$ if and only if $i\in S$ for all $i\in\{2,\.,r\}.$

        \item The maximal ideal $\p_1$ and the localization $\widetilde{\O}_{\p_1}$ are definable in $(L,K,\O)$ over $\{a_{S}:S\subseteq\{2,\.,r\}\}.$ 

        \item If $\{\O_1,\.,\O_r\}$ is the set of all extensions of $\O$ to $L,$ then $\O_1$ is definable in $(L,K,\O)$ over $\{a_S:S\subseteq\{2,\.,r\}\}.$
    \end{enumerate} 
\end{lema}
\begin{proof}
\begin{enumerate}[label*={\arabic*.}]
    \item Let $a\in\widetilde{\O},$ so that there is a monic polynomial $p\in\O[X]$ with $p(a)=0.$ It follows that $\mu_{a,K}|\,p$ in $K[X],$ so all roots of $\mu_{a,K}$ are also roots of $p$ and are, consequently, integral over $\O.$ This implies that the coefficients of $\mu_{a,K},$ which are symmetric polynomial expressions of such roots, are moreover elements of $K$ that are integral over $\O,$ so they are elements of $\O$ themselves.  

    \item If $a\in L$ is an element in $\widetilde{\O},$ then it is algebraic over $K$ and its degree over $K$ divides $[L:K].$ By Statement 1, we get that $a\in\widetilde{\O}$ if and only if
    \begin{equation}
        \label{eqintcl}
        (L,K,\O)\models\bigvee_{n|[L:K]}\E x_0,\.,x_{n-1}\left(a^{n}=\sum_{i=0}^{n-1}x_ia^{i}\wedge\bigwedge_{i=0}^{n-1} x_i\in\O\right).
    \end{equation}

    \item Given a subset $S\subseteq\{2,\.,r\},$ use the Chinese Reminder Theorem to find a solution $a_S\in\widetilde{\O}$ of the system of congruences given by $$\begin{cases}
        x\equiv 0 \text{ mod }\p_1 \\ 
        x\equiv 0 \text{ mod }\p_i \text{ if }i\in S\\
        x\equiv 1 \text{ mod }\p_i \text{ if }i\not\in S.
    \end{cases}$$
    The result follows.

    \item Let $a\in L.$ Then $a\in\p_1$ if and only if $a\in\widetilde{\O}$ and $a+a_S$ is not invertible in $\widetilde{\O}$ for any $S\subseteq\{2,\.,r\},$ i.e.~if and only if 
    \begin{equation}
    \label{eqmaxid}
        (L,K,\O)\models a\in\widetilde{\O}\wedge\bigwedge_{S\subseteq\{2,\.,r\}}\A x\left(x\in\widetilde{\O}\to 1\neq x(a+a_S)\right).
    \end{equation}
    Indeed, if $a\in\p_1$ and $S\subseteq\{2,\.,r\},$ then $a_S\in\p_1$ implies that $a+a_S\in\p_1,$ so $a+a_S$ is not invertible in $\widetilde{\O}.$ Conversely, if $a\in\widetilde{\O}\setminus\p_1,$ then define $S=\{i\in\{2,\.,r\}:a\not\in\p_i\}.$ It follows that $a\in\p_i$ if and only if $i\not\in S,$ if and only if $a_S\not\in\p_i,$ for $i\in\{2,\.,r\}.$ It follows that $a+a_S\not\in\p_i$ for any $i\in\{1,\.,r\}.$ Indeed, if $a+a_S\in\p_1$ then $a=(a+a_S)-a_S\in\p_1,$ which is absurd. Also, if $i\in S,$ and $x+a_S\in\p_i,$ then $a=(a+a_S)-a_S\in\p_i,$ contradicting the choice of $S.$ Finally, if $i\not\in S$ and $x+a_S\in\p_i,$ then $a\in\p_i$ and $a_S=(a+a_S)-a\in\p_i,$ contradicting the construction of $a_S.$ It follows that $a+a_S\in\widetilde{\O}\setminus\bigcup_{i=1}^{r}\p_i,$ so $a+a_S$ is invertible in $\widetilde{\O}.$
    Consequently, if $a\in L,$ then $a\in\widetilde{\O}_{\p_1}$ if and only if 
    \begin{equation}
        \label{eqloc}
        (L,K,\O)\models\E x,y\left(x\in\widetilde{\O}\wedge y\in\widetilde{\O}\setminus\p_1\wedge ya=x\right).
    \end{equation}

    \item First, the integral closure $\widetilde{\O}$ of $\O$ in $L$ coincides with $\bigcap_{i=1}^{r}\O_i$ and the set $\{\p_1,\.,\p_r\}$ of maximal ideals of $\widetilde{\O}$ is given by $\p_i=\M_i\cap\widetilde{\O}=\M_i\cap\bigcap_{j\neq i}\O_j$ for each $i\in\{1,\.,r\},$ see e.g.~\cite[Corollary 3.1.4, Lemma 3.2.6 and Theorem 3.2.7]{ep}. Let   
    $\{a_S:S\subseteq\{2,\.,r\}\}\subseteq\p_1$ be the parameters given by Statement 3, so that $a_S\in\p_i$ if and only if $i\in S$ for all $i\in\{2,\.,r\}$ and $S\subseteq\{2,\.,r\}.$ The aforementioned references also show that each $\O_i$ coincides with the localization $\widetilde{\O}_{\p_i},$ so we conclude that $\O_1$ is definable in $(L,K,\O)$ over $\{a_S:S\subseteq\{2,\.,r\}\}.$
    \qedhere
\end{enumerate}
\end{proof} 

\begin{lema}
\label{defin}
    Let $(K,\O)$ be a separably defectless valued field and let $(L,\O^{*})\supseteq(K,\O)$ be a finite Galois extension. Then $\O^{*}$ is existentially and universally definable in $(L,K,\O).$ 
\end{lema}
\begin{proof}
    Let $\{\O_1,\O_2,\.,\O_r\},$ with $\O_1=\O^{*},$ be the set of all extensions of $\O$ to $L.$ With the notations of Lemma \ref{finextdef}, let $a=(a_S:S\subseteq\{2,\.,r\})\subseteq\p_1$ and let $\f(x,y)$ be a formula such that $\O^{*}$ is defined by $\f(L,a)$ in $(L,K,\O).$   
    Since $L|K$ is normal, by the Conjugation Theorem \cite[Theorem 3.2.15]{ep}, there are some $\sigma_2,\.,\sigma_r\in\Gal(L|K)$ such that $\O_i=\f(L,\sigma_i a),$ and thus $a_{\{i\}}$ is not invertible in $\f(L,a)$ but invertible in $\f(L,\sigma_i a)$ for $i\in\{2,\.,r\}.$ Note also that if $n=[L:K],$ $e=e(\O^{*}|\O)$ and $f=f(\O^{*}|\O),$ we get that the fundamental inequality corresponds in this case to $n=ref$ because $L|K$ is Galois and $(K,\O)$ is separably defectless. 
    It follows that if $(L',K',\O',a_1',a_2',\.,a_r')\equiv(L,K,\O,a,\sigma_2a,\.,\sigma_ra),$ then the number $r'$ of extensions of $\O'$ to $L'$ is still $r.$ This holds because $n,e$ and $f$ remain unchanged when computed in $(L',K',\O'),$ so there is some $d\in\N$ and some prime number $p$ for which the fundamental inequality computed for the extension $L'|K'$ corresponds to $n=r'efp^{d}.$ It follows that $r'\leq r,$ and since $\f(x,a_1'),\.,\f(x,a_r')$ are \emph{pairwise different} extensions of $\O'$ to $L',$ we get that $r\leq r'$ and $r=r'.$ 
    
    Now we are ready to prove that $\O^{*}$ is existentially and universally definable in $(L,K,\O).$ 
    Indeed, we claim that $\O^{*}$ is existentially and universally $\emptyset$-definable in the structure $\AA$ given by $(L,K,\O,a,\sigma_2 a,\.,\sigma_r a).$ To this end, let $\MM,\NN\models\Th(\AA)$ be such that $\MM\subseteq\NN.$ If we let $\MM=(L_1,K_1,\O_1,b_1,b_2,\.,b_r)$ and $\NN=(L_2,K_2,\O_2,b_1,b_2,\.,b_r)$ with $b_i=(b_{i,S}:S\subseteq\{2,\.,r\})\subseteq L_1,$ then we want to prove that $\f(L_1,b_1)\subseteq\f(L_2,b_1)$ and that $\f(L_2,b_1)\cap L_1\subseteq\f(L_1,b_1)$ for existential and universal definability respectively. We will see that $\f(L_2,b_1)\cap L_1=\f(L_1,b_1),$ which will take care of both statements at once. 
    First, note that $\f(L_2,b_1)\cap L_1$ is a valuation ring of $L_1$ extending $\O_1.$ Indeed, $\MM\subseteq\NN$ implies that $L_1\subseteq L_2,$ that $K_2\cap L_1= K_1$ and that $\O_2\cap L_1=\O_1.$ Then $\f(L_2,b_1)\cap L_1$ is a valuation ring of $L_1$ and 
    $(\f(L_2,b_1)\cap L_1)\cap K_1=\f(L_2,b_1)\cap K_1=\f(L_2,b_1)\cap(K_2\cap L_1)=\O_2\cap L_1=\O_1.$
    It follows that either $\f(L_2,b_1)\cap L_1=\f(L_1,b_1)$ or that $\f(L_2,b_1)\cap L_1=\f(L_1,b_i)$ for some $i\in\{2,\.,r\}.$ Indeed, since $\f(x,a),\f(x,a_2),\.,\f(x,a_r)$ define pairwise different extensions of $\O$ to $L$ in $\AA,$ then the same holds for $\MM$ and $\NN,$ and since $r$ remains unchanged when computed in $\AA$ and in $\MM$ or $\NN,$ it follows that the traces of $\f(x,b_1),\f(x,b_2),\.,\f(x,b_r)$ in $L_1$ and $L_2$ define \emph{all} extensions of $\O_1$ to $L_1$ and \emph{all} the extensions of $\O_2$ to $L_2$ respectively. We also have that $b_{1,\{i\}}$ is not invertible in $\f(L_1,b_1)$ nor in $\f(L_2,b_1),$ but it is in $\f(L_1,b_i)$ and in $\f(L_2,b_i),$ because, in $\AA,$ the parameter $a_{\{i\}}$ is not invertible in $\f(L,a)$ but it is in $\f(L,a_i)$ for all $i\in\{2,\.,r\}.$ 
    Hence, if $\f(L_2,b_1)\cap L_1=\f(L_1,b_i)$ for some $i\in\{2,\.,r\},$ then $b_{1,\{i\}}$ is not invertible in $\f(L_2,b_1)\cap L_1=\f(L_1,b_i),$ which is absurd, and thus $\f(L_2,b_1)\cap L_1=\f(L_1,b_1).$       
\end{proof}

\begin{cor}
\label{enipexp}
    Let $(K,\O)$ be a non-trivial $\E$-NIP valued field, and let $L|K$ be a finite Galois extension. Then, for any pair of extensions $\O_1,\O_2$ of $\O$ to $L,$ the structure $(L,\O_1,\O_2)$ is $\E$-NIP.
\end{cor}
\begin{proof}
    By Corollary \ref{eint}, the structure $(L,K,\O)$ is $\E$-to-$\E$ interpretable in $(K,\O)$ as the highlighted sets of Lemma \ref{intcrit} can be chosen to be quantifier-free definable. Hence $(L,K,\O)$ is itself $\E$-NIP by Lemma \ref{dtod'nip}. Also, $(K,\O)$ is separably defectless by Lemma \ref{sd}. By Lemma \ref{defin}, $\O_1$ and $\O_2$ are existentially and universally definable in $(L,K,\O),$ so it follows from Corollary \ref{eint} that $(L,\O_1,\O_2)$ is $\E$-to-$\E$-interpretable in $(L,K,\O).$ We conclude by Lemma \ref{dtod'nip}.        
\end{proof}

\begin{lema}[Cf.~{\cite[Lemma 2.6]{dp1a}}]
\label{compar}
    Let $\KK=(K,\O_1,\O_2)$ be a 2-valued field of characteristic $p>0.$ If $\KK$ is $\E$-NIP, $\O_1$ and $\O_2$ are comparable.
\end{lema}
\begin{proof}
    If $\O_1,\O_2$ are incomparable, then $R=\O_1\cap\O_2$ is a multi-valuation ring of $K$ which is not a valuation ring. Since any valuation of a finite field is trivial, $K$ has to be infinite. Since $\KK$ is $\E$-NIP, $K$ is Artin-Schreier closed, forcing $\O_1/\M_1$ and $\O_2/\M_2$ to be also Artin-Schreier closed and hence infinite.
    
    The Jacobson radical $J$ of $R$ is equal to $\M_1\cap\M_2$ (cf.~\cite[Lemma 2.1]{dp1a}). Since $\M_1$ and $\M_2$ are existentially definable by Lemma \ref{defin}, $J^{00,\pm\E}$ exists by Proposition \ref{g00}. Additionally, if $x\in J,$ then any Artin-Schreier root $y$ of $x$ would be an element of $R$ whose residue modulo $\M_i$  is fixed by the Frobenius. Therefore there are two unique elements $a,b\in\Z/p\Z$ such that $y\in(a+\M_1)\cap(b+\M_2).$
    If $z$ is another Artin-Schreier root of $x,$ then $(y-z)^p=y-z$ implies that $z=y+i$ for some $i\in\Z/p\Z,$ hence $z\in(a-i+\M_1)\cap(b-i+\M_2)$. Hence the map $f:(J,+)\to(\Z/p\Z,+)$ sending $x$ to $a-b$ is a well defined definable morphism, as it does not depend on the choice of an Artin-Schreier root of $x.$ 
    In particular, $x\in\ker(f)$ is definable by the existential formula 
    \begin{equation}
        \label{eqkerf}
        \E z\left(z^{p}-z=x\wedge\bigvee_{a\in\F_p}(z-a\in\M_1\wedge z-a\in\M_2)\right).
    \end{equation}
    The morphism $f$ is also surjective, since any element $y\in(1+\M_1)\cap\M_2$ will satisfy that $v_1(y^{p}-y)=v_1(y-1)>0$ and that $v_2(y^{p}-y)=v_2(y)>0,$
    yielding $f(y^{p}-y)=1.$ Any such $y$ can be found by Weak Approximation of incomparable valuations on $K,$ cf.~\cite[Theorem 3.2.7]{ep}.
    Finally, given that $\ker(f)$ is an existentially definable subgroup of $J$ of index $p,$ it follows that $J^{00,\pm\E}\subseteq\ker(f)\subsetneq J,$ in contradiction with Lemma \ref{j00}. 
\end{proof}

\begin{teorema}[Johnson, cf.~{\cite[Theorem 2.8]{dp1a}}]
\label{eniphens}
    Let $(K,\O)$ be an $\E$-NIP valued field of characteristic $p>0.$ Then $\O$ is henselian. 
\end{teorema}
\begin{proof}
    Let $L$ be a finite Galois extension of $K$ admitting two extensions $\O_1,\O_2$ of $\O.$ By Corollary \ref{enipexp}, the structure $(L,\O_1,\O_2)$ is also $\E$-NIP. Thus, by Lemma \ref{compar}, $\O_1$ and $\O_2$ are comparable, hence equal, cf.~\cite[Lemma 3.2.8]{ep}.
\end{proof}

We can improve this result by considering type- and $\E_{+}$-type-definable valued fields of positive characteristic in NIP and $\E_{+}$-NIP structures respectively. Recall that, by compactness, the class of type-definable subsets is closed under finite disjunctions and quantifications.

\begin{defin}
   Let $\D$ be a set of formulas of some language $\LL$ and let $\MM$ be an $\LL$-structure with universe $M.$ We say that the valued field $(K,\O)$ is \emph{$\D$-type-definable in $\MM$} if $\O$ and $K$ are $\D$-type-definable subsets of some power of $M,$ and if the field operations and constants are $\D$-definable in $\MM.$ 
\end{defin}

We will need the following result of Johnson.

\begin{fact}[cf.~{\cite[Main Theorem]{tdf}}]
\label{tdfasc}
    Let $K$ be an infinite type-definable field of positive characteristic in some sufficiently saturated NIP structure. Then $K$ is Artin-Schreier closed.
\end{fact}

In fact, the proof of Fact \ref{tdfasc} relies on Corollary \ref{kapsh} in the case in which $\D$ is the full set of formulas. For $\D=\E_{+}^{n},$ one can follow the proof of Fact \ref{tdfasc} and, together with Corollary \ref{kapsh}, conclude the following lemma.

\begin{lema}
\label{tdfasccor}
    Let $\MM$ be some sufficiently saturated $\LL$-structure with universe $M$ and let $K\subseteq M^{d}$ be an infinite $\E^{n}_{+}$-type-definable field of characteristic $p>0$. If $\MM$ is $\E_{+}^{n+d(p+1)}$-NIP, then $K$ is Artin-Schreier closed.
\end{lema}
\begin{proof}
    Denote by $\wp(x)=x^{p}-x$ the Artin-Shreier polynomial. In \cite[Lemma 2.1]{tdf}, Johnson proves that if $L|K$ is a non-trivial Artin-Schreier extension and $(c_i:i\in\Z)$ is an indiscernible sequence, then there is some $N<\omega$ such that 
    \begin{equation}
        \label{emptyint}
        (c_0\cdot(\wp(L\setminus K)\cap K))\cap\bigcap_{i=1}^{N}(c_i\cdot\wp(K))\cap\bigcap_{i=-N}^{-1}(c_i\cdot\wp(K))=\emptyset,
    \end{equation}
    using \cite[Corollary 3.3]{ksh}. 
    Identifying $L$ with the set $\{(x_1,\.,x_p):x_1,\.,x_p\in K\}$ and $K\subseteq L$ with $\{(x, 0,\.,0):x\in K\}$, we see that $L$ can be defined as a $\E^{n}_{+}$-type-definable subset of $M^{dp},$ making equation \ref{emptyint} a statement in $M^{dp}.$
    We can apply Corollary \ref{kapsh} instead of \cite[Corollary 3.3]{ksh} in order to recover equation \ref{emptyint}, by taking $\D=\E_{+}^{n+d(p+1)}.$ 
    This is because if $K\subseteq M^{d}$ is $\E_{+}^{n}$-type-definable, then each $c_i\cdot\wp(K)$ is $\E_{+}^{n+d}$-type-definable (and thus $\E_{+}^{n+d(p+1)}$-type-definable) by counting the extra $d$ existential quantifiers that define the image of the Artin-Shreier polynomial, and analogously, $c_0\cdot(\wp(L\setminus K)\cap K)$ is $\E_{+}^{n+d(p+1)}$-type-definable: an element $x\in M^{dp}$ is in $c_0\cdot(\wp(L\setminus K)\cap K)$ if and only if $x\in K$ and there is some $y\in M^{dp}$ such that $y\in L,$ $y\not\in K$ and $x+c_0y=c_0y^{p},$ and if $y=(y_1,y_2,\.,y_p),$ then $y\not\in K$ holds if and only if, for some $i\in\{2,\.,p\},$ there is some $z\in M^{d}$ such that $z\in K$ and $zy_i=1.$    
    Johnson then proves in \cite[Main Theorem]{tdf} that there is some indiscernible sequence for which the intersection in equation \ref{emptyint} is non-empty, contradicting the existence of a non-trivial Artin-Schreier extension. 
\end{proof}

\begin{lema}
\label{valringexttd}
Let $n,d$ and $l$ be some positive integers, let $\MM$ be some $\LL$-structure with universe $M$ and let $(K,\O)$ be a $\E_{+}^{n}$-type-definable valued field with $K\subseteq M^{d}$. Let $L|K$ be some finite separable extension of $K$ of degree at most $l.$ Then any extension of $\O$ to $L$ is a $\E_{+}^{n+4dl}$-type-definable subset of $M^{dl}.$ 
\end{lema}
\begin{proof}
    Identifying $L$ with the set $\{(x_1,\.,x_l):x_1,\.,x_l\in K\}$ and $K\subseteq L$ with $\{(x,0,\.,0):x\in K\}$, we see that $L$ can be defined as a $\E^{n}_{+}$-type-definable subset of $M^{dl}.$
    Reading the proof of Lemma \ref{finextdef}, counting quantifiers in equation \ref{eqintcl} and noting that the variables $x_0,\.,x_{n-1}$ are meant to be elements of $\O\subseteq M^{d}$, we get that the additional number of existential quantifiers in equation \ref{eqintcl} is $dl.$ 
    This implies that integral closure $\widetilde{\O}$ of $\O$ in $L$ is a $\E_{+}^{n+dl}$-type-definable subset of $M^{dl}$. If $\p$ is a maximal ideal of $\widetilde{\O},$ using the negation of the formula in equation \ref{eqmaxid}, we see that $\widetilde{\O}\setminus\p$ is a $\E_{+}^{n+2dl}$-type-definable subset of $M^{dl}$ over the parameters $\{a_{S}:S\subseteq\{2,\.,r\}\}.$ Finally, one recovers any extension of $\O$ to $L$ as some localization $\widetilde{\O}_{\p},$ which can be seen to be a $\E_{+}^{n+4dl}$-type-definable subset of $M^{dl}$ by equation \ref{eqloc}.  
\end{proof}

\begin{remark}
    \label{maxidtd}
Let $\MM$ be some $\LL$-structure with universe $M,$ and let $(K,\O)$ be a $\E^{n}_{+}$-type-definable valued field for which $\O\subseteq K\subseteq M^{d}$. If $\M$ is the maximal ideal, then it is a $\E^{n+d}_{+}$-type-definable subset of $M^{d}.$ 
\end{remark}
\begin{proof}
    We first see that $\M$ is $\A_{-}^{n+d}$-$\vee$-definable. Indeed, if $a\in\O,$ then $a\in\M$ if and only if $a^{-1}\not\in\O,$ which can be written as $ab\neq 1$ \emph{for all} $b\in\O.$ Since the arity of $\O$ is $d,$ the highlighted universal quantification counts as $d$ additional quantifiers in the definition scheme of $\M.$ Then $\O\setminus\M=\bigcup_{a\not\in\M}a+\M$ is $\A^{n+d}_{-}$-$\vee$-definable, implying that $\M=\O\setminus\left(\bigcup_{x\in\O\setminus\M}x+\M\right)$ is $\E_{+}^{n+d}$-type-definable.
\end{proof}

The following definition is similar to \cite[Definition 4.5]{ajlring}, which is reminiscent of the definition of $t_n$-henselianity from \cite{pz}. 

\begin{defin}
    Let $l$ be a positive integer. We say that a valued field $(K,\O)$ is \emph{$l$-henselian} if $\O$ extends uniquely to any separable extension $L|K$ of degree at most $l.$
\end{defin}

Now we are ready to state several improvements of Theorem \ref{eniphens}.

\begin{teorema}
\label{finereniphens}
    Let $d$ and $l$ be some positive integers and let $n\in\N.$ Let $\MM$ be some sufficiently saturated $\LL$-structure with universe $M$ and let $(K,\O)$ be a type-definable valued field of characteristic $p>0$, with $K\subseteq M^{d}.$
    \begin{enumerate}[label*={\arabic*.}]
        \item If $\MM$ is NIP, then $\O$ is henselian.

        \item If $\MM$ is $\E_{+}$-NIP and $(K,\O)$ is $\E_{+}$-type-definable, then $\O$ is henselian.

        \item Let $m=\max\{6l,p+1\}.$ If $\MM$ is $\E_{+}^{n+dm}$-NIP and $(K,\O)$ is $\E_{+}^{n}$-type-definable, then $\O$ is $l$-henselian.
    \end{enumerate}
\end{teorema}
\begin{proof}
    We only prove Statement 3, as the first two statements follow analogously. Let $\D=\E_{+}^{n+dm},$ let $L|K$ be a finite separable extension of degree at most $l$ and let $\O_1,\O_2$ be two valuation rings of $L$ extending $\O.$ By Lemma \ref{valringexttd}, $\O_1$ and $\O_2$ are $\E^{n+4dl}_{+}$-type-definable subsets of $M^{dl}.$
    If $\M_1$ and $\M_2$ are the maximal ideals of $\O_1$ and $\O_2$ respectively, then they are $\E_{+}^{n+5dl}$-type-definable by Remark \ref{maxidtd}. 
    As in Lemma \ref{compar}, if $\O_1$ is not equal to $\O_2,$ then $\O_1$ and $\O_2$ are incomparable and thus the Jacobson radical $J=\M_1\cap\M_2$ of $\O_1\cap\O_2$ is a $\E_{+}^{n+5dl}$-type-definable group.  
    Finally, if $f:(J,+)\to(\Z/p\Z,+)$ is the morphism of Lemma \ref{compar}, then $\ker(f)$ is a $\E_{+}^{n+6dl}$-type-definable subgroup of $J$ of index $p,$ following equation \ref{eqkerf}. Hence $\ker(f)$ is $\D$-type-definable. Since $\MM$ is $\D$-NIP, then $J^{00,\pm\D}$ exists by Proposition \ref{g00} and, in this case, $J^{00,\pm\D}\subseteq\ker(f)\subsetneq J.$ This contradicts Lemma \ref{j00}, because if $\MM$ is $\D$-NIP, then it is $\E_{+}^{n+d(p+1)}$-NIP; if $L$ is $\E_{+}^{n}$-type-definable then $L$ is Artin-Schreier closed by Lemma \ref{tdfasccor}, and thus the residue fields of $\O_1$ and $\O_2$ are also Artin-Schreier closed and infinite. 
\end{proof}

\begin{remark}
    Theorem \ref{eniphens} can be obtained as a corollary of Theorem \ref{finereniphens}. Indeed, if a valued field of characteristic $p>0$ is $\E$-NIP, then it is $\E_{+}$-NIP and $\E_{+}^{\max\{6l,p+1\}}$-NIP for all $l\geq 1.$ Statement 3 of Theorem \ref{finereniphens}, where $n=0$ and $d=1,$ tells us that the valuation is $l$-henselian for all $l.$ 
\end{remark}

\section{A Finer NIP Transfer}
\label{6}

\subsection{\texorpdfstring{$\mirrorexists_{n}$}{}-Elimination}

The following fact, whose information is summarized in Figure \ref{nevenodd}, can be obtained as a corollary of Fact \ref{exembg}.  

\begin{fact}[Cf.~Theorem 5.3.8 and Exercise 5.2.7 of \cite{ck}]
\label{enequiv}
    Let $n\geq 1$ and let $\AA,\BB$ be two $\LL$-structures.
    \begin{enumerate}[label*={\arabic*.}]
        \item Suppose that $n$ is even. The following statements are equivalent.
        \begin{enumerate}
            \item $\AA\Rightarrow_{\mirrorexists_{n}}\BB.$

            \item There are some $\AA=\AA_0\preceq\AA_2\preceq\AA_4\preceq\.\preceq\AA_n$, some $\BB=\BB_0\preceq\BB_1\preceq\BB_3\preceq\.\preceq\BB_{n-1}$ and some $\LL$-embeddings $f_i:\AA_{i-1}\to\BB_{i}$ if $i$ is odd, $f_i:\BB_{i-1}\to\AA_i$ if $i$ is even, $i\in\{1,\.,n\},$ such that $f_{i+1}\circ f_i=\id_{A_{i-1}}$ if $i$ is odd and $f_{i+1}\circ f_i=\id_{B_{i-1}}$ if $i$ is even.
        \end{enumerate}
    
        \item Suppose that $n$ is odd. The following statements are equivalent.
        \begin{enumerate}
            \item $\AA\Rightarrow_{\mirrorexists_{n}}\BB.$

            \item There are some $\AA=\AA_0\preceq\AA_2\preceq\AA_4\preceq\.\preceq\AA_{n-1}$, some $\BB=\BB_0\preceq\BB_1\preceq\BB_3\preceq\.\preceq\BB_{n}$ and some $\LL$-embeddings $f_i:\AA_{i-1}\to\BB_{i}$ if $i$ is odd, $f_i:\BB_{i-1}\to\AA_i$ if $i$ is even, $i\in\{1,\.,n\},$ such that $f_{i+1}\circ f_i=\id_{A_{i-1}}$ if $i$ is odd and $f_{i+1}\circ f_i=\id_{B_{i-1}}$ if $i$ is even.
        \end{enumerate}
    \end{enumerate}
\end{fact} 
\begin{figure}[ht]
    \centering
    \begin{tabular}{cccc}
    % \hline
       \begin{tikzcd}
\AA_{n} \arrow[d, "\preceq"', no head]                 &                                                              \\
\vdots                                                 & \BB_{n-1} \arrow[lu, "f_{n}"'] \arrow[d, "\preceq", no head] \\
\AA_2 \arrow[u, "\preceq", no head]                    & \vdots                                                       \\
                                                       & \BB_1 \arrow[lu, "f_2"'] \arrow[u, "\preceq"', no head]      \\
\AA_0 \arrow[ru, "f_1"] \arrow[uu, "\preceq", no head] & \BB_0 \arrow[u, "\preceq"', no head]                        
\end{tikzcd} & & &\begin{tikzcd}
                                                           & \BB_{n}                                                                                        \\
\AA_{n-1} \arrow[d, "\preceq"', no head] \arrow[ru, "f_n"] &                                                                                                \\
\vdots                                                     & \BB_{n-2} \arrow[lu, "f_{n-1}"'] \arrow[d, "\preceq", no head] \arrow[uu, "\preceq"', no head] \\
\AA_2 \arrow[u, "\preceq", no head]                        & \vdots                                                                                         \\
                                                           & \BB_1 \arrow[lu, "f_2"'] \arrow[u, "\preceq"', no head]                                        \\
\AA_0 \arrow[ru, "f_1"] \arrow[uu, "\preceq", no head]     & \BB_0 \arrow[u, "\preceq"', no head]                                                          
\end{tikzcd} \\ 

       The case $\AA\Rightarrow_{\mirrorexists_{n}}\BB,$ $n$ even.  & & &  The case $\AA\Rightarrow_{\mirrorexists_{n}}\BB,$ $n$ odd.\\
    % \hline
    \end{tabular}
    \caption{Embeddings of Fact \ref{enequiv}}
    \label{nevenodd}
\end{figure}

For $n=1,$ the following proposition recovers \cite[Corollary 4.6]{sm}.

\begin{propo}
\label{delredn}
Let $e\in\N\cup\{\infty\},$ let $\LL=\LL(\LL_\kk,\LL_\GG),$ let $n\geq 1$ and let $x,y,z$ be tuples of variables of sort $\KK,\kk$ and $\GG,$ respectively. Any $\mirrorexists_{n}$-formula $\Phi(x,y,z)$ from $\LL$ is equivalent, modulo $\texttt{SAMK}_{e}^{\l,\ac},$ to a $\{\vee,\wedge\}$-combination of the following formulas:
    \begin{enumerate}[label*={\arabic*.}]
        \item $t(x)=0$ and their negations,
            
        \item $\th(\underline{ac}(t_1(x)),\.,\underline{ac}(t_m(x)),y),$ where $\th$ is an $\mirrorexists_{n}$-formula from $\LL_\kk$, 

        \item $\psi(\underline{v}(t_1(x)),\.,\underline{v}(t_m(x)),z),$ where $\psi$ is an $\mirrorexists_{n}$-formula from $\LL_\GG,$
    \end{enumerate}
    and where $t(x),t_1(x),\.,t_m(x)$ are $\LL$-terms of sort $\KK^{|x|}\to\KK$.
\end{propo}
\begin{proof} 
Let $\D'$ be the collection of all $\{\vee,\wedge\}$-combinations of formulas of the form (1), (2) and (3) at level $n.$ Let $\MM,\NN$ be models of $\texttt{SAMK}_{e}^{\l,\ac},$ and let let $a,\a,\g,b,\b,\d$ be finite tuples from $M,$ $\kk(M)$, $\GG(M)$, $N,$ $\kk(N)$ and $\GG(N)$ respectively. Suppose that $|a|=|b|,$ $|\a|=|\b|,$ and $|\g|=|\d|.$ Suppose as well that $\MM$ is $\aleph_0$-saturated. By Item 2 of Lemma \ref{delredeq}, it is enough to prove that $(\MM,a,\a,\g)\Rightarrow_{\D'}(\NN,b,\b,\d)$ implies $(\MM,a,\a,\g)\Rightarrow_{\mirrorexists_{n}}(\NN,b,\b,\d).$
Let $\AA$ be the $\LL$-substructure of $\MM$ generated by $a,\a$ and $\g$, and let $A$ be the universe of the $\KK$-sort of $\AA.$ Since $\D'$ contains all quantifier-free $\LL$-formulas, then $(\MM,a,\a,\g)\Rightarrow_{\D'}(\NN,b,\b,\d)$ implies that there is an $\LL$-embedding $\Sigma_{0}=(\iota_{0},\sigma_{0},\rho_{0}):\AA\to\NN$ mapping $(a,\a,\g)$ to $(b,\b,\d).$ Also, since $\D'$ contains all formulas of the form (2) and (3), we get that $(\kk(M),\xi)_{\xi\in\kk(A)}\Rightarrow_{\mirrorexists_{n}}(\kk(N),\sigma_{0}(\xi))_{\xi\in\kk(A)}$
with respect to the corresponding expansion by constants of $\LL_\kk,$ and $(\GG(M),\e)_{\e\in\GG(A)}\Rightarrow_{\mirrorexists_{n}}(\GG(N),\rho_{0}(\e))_{\e\in\GG(A)}$ with respect to the corresponding expansion by constants of $\LL_\GG.$ 
If $\NN_1$ is an $|M|^{+}$-saturated elementary extension of $\NN,$ then Fact \ref{exembg} yields the existence of two functions $\sigma_1:\kk(M)\to\kk(N_1)$ and $\rho_1:\GG(M)\to\GG(N_1)$ preserving $\mirrorexists_{n}$-formulas from $\LL_\kk$ and $\LL_\GG$ and such that $\sigma_1|_{\kk(A)}=\sigma_{0}$ and $\rho_1|_{\GG(A)}=\rho_{0}$ respectively. In particular, $\sigma_1|_{Av}=\sigma_{0}|_{Av}$ and $\rho_1|_{vA}=\rho_{0}|_{vA}$ respectively. By the Embedding Lemma \cite[Theorem 1.1]{sm} applied to $\langle A\rangle$, there is some ring embedding $\iota_1:M\to N_1$ inducing an $\LL$-embedding $\Sigma_1=(\iota_1,\sigma_1,\rho_1):\MM\to\NN_1$ that extends $\Sigma_0.$   
Now suppose that we have constructed $\MM_{l-1},\NN_{l}$ together with an $\LL$-embedding $\Sigma_l=(\iota_l,\sigma_l,\rho_l):\MM_{l-1}\to\NN_{l}$ satisfying that $\sigma_l:\kk(M_{l-1})\to_{E_{n-l}}\kk(N_{l})$ with respect to $\LL_\kk$ and that $\rho_l:\GG(M_{l-1})\to_{E_{n-l}}\GG(N_{l})$ with respect to $\LL\GG.$ In particular, $(\kk(N_{l}),\sigma_l(\xi))_{\xi\in\kk(M_{l-1})}\Rightarrow_{\mirrorexists_{n-l-1}}(\kk(M_{l-1}),\xi)_{\xi\in\kk(M_{l-1})}$
with respect to the corresponding expansion by constants of $\LL_\kk,$ and $(\GG(N_{l}),\rho_l(\e))_{\e\in\GG(M_{l-1})}\Rightarrow_{\mirrorexists_{n-l-1}}(\GG(M_{l-1}),\e)_{\e\in\GG(M_{l-1})}$ with respect to the corresponding expansion by constants of $\LL_\GG.$ Let $\MM_{l+1}$ be an $|N_{l}|^{+}$-saturated elementary extension of $\MM_{l-1}.$ By Fact \ref{exembg}, there are two functions $\sigma_{l+1}:\kk(N_{l})\to\kk(M_{l+1})$ and $\rho_{l+1}:\GG(N_{l})\to\GG(M_{l+1})$ preserving $\mirrorexists_{n-l-1}$-formulas from $\LL_\kk$ and $\LL_\GG$ respectively, and satisfying that $\sigma_{l+1}(\sigma_l(\xi))=\xi$ for all $\xi\in\kk(M_{l-1})$ and $\rho_{l+1}(\rho_{l}(\e))=\e$ for all $\e\in\GG(M_{l-1}).$ By the Embedding Lemma \cite[Theorem 1.1]{sm}, there is a ring embedding $\iota_{l+1}:N_{l}\to M_{l+1}$ inducing an $\LL$-embedding $\Sigma_{l+1}=(\iota_{l+1},\sigma_{l+1},\rho_{l+1}):\NN_{l}\to\MM_{l+1}$ that extends $\Sigma_l^{-1}.$ Analogously, one may repeat the same argument and find, for an $|M_{l+1}|^{+}$-saturated elementary extension $\NN_{l+2}$ of $\NN_l,$ an embedding $\Sigma_{l+2}=(\iota_{l+2},\sigma_{l+2},\rho_{l+2}):\MM_{l+1}\to\NN_{l+2}$ extending $\Sigma_{l+1}^{-1}$ and such that $\sigma_{l+2}:\kk(M_{l+1})\to_{\E_{n-l-2}}\kk(N_{l+2})$ with respect to $\LL_\kk,$ and $\rho_{l+2}:\GG(M_{l+1})\to_{\E_{n-l-2}}\GG(N_{l+2})$ with respect to $\LL_\GG.$ Iterating this construction until the $n^{th}$ stage, we may conclude that $(\MM,a,\a,\g)\Rightarrow_{\mirrorexists_{n}}(\NN,b,\b,\d)$ by Fact \ref{enequiv}.
\end{proof}

\begin{cor}
\label{sent}
    Any $\mirrorexists_{n}$-sentence from $\LL$ is equivalent modulo $\texttt{SAMK}_e^{\l,\ac}$ to a $\{\vee,\wedge\}$-combination of $\mirrorexists_{n}$-sentences of $\LL(\kk)$ and $\LL(\GG).$
\end{cor}
\begin{proof}
    Any $\mirrorexists_{n}$-sentence from $\LL$ is equivalent to a $\{\vee,\wedge\}$-combination of $\mirrorexists_{n}$-sentences from $\LL(\kk),$ from $\LL(\GG)$ and of sentences of the form $t=0$ and $t\neq 0$ where $t$ is a term in $\LL(\KK)$ without free variables. These terms are either of the form $\sum_{i=1}^{n}1_{\KK}$ for $n\in\N,$ or of the form $\l_{n,m}\left(l,\sum_{i=1}^{k}1_{\KK}\right)$ where $n\geq 1,$ $m\in\Mon(n),$ $l=\left(\sum_{i=1}^{l_1}1_{\KK},\.,\sum_{i=1}^{l_n}1_{\KK}\right)$ and $l_1,\.,l_n,k\in\N.$ In the first case, since $\texttt{SAMK}_e^{\l,\ac}$ is a theory of \emph{equi-characteristic} valued fields, the sentences $\sum_{i=1}^{n}1_{\KK}=0$ and $\sum_{i=1}^{n}1_{\KK}\neq 0$ are equivalent modulo $\texttt{SAMK}_e^{\l,\ac}$ to $\sum_{i=1}^{n}1_{\kk}=0$ and $\sum_{i=1}^{n}1_{\kk}\neq 0$ respectively, which are $\mirrorexists_{n}$-sentences from $\LL(\kk).$ In the second case, since the prime subfield $\F$ of any field is perfect, there are no $p$-independent tuples of elements of $\F$ for any $p,$ implying that $\texttt{SAMK}_{e}^{\l,\ac}\vdash\l_{n,m}\left(l,\sum_{i=1}^{k}1_{\KK}\right)=0,$ so the corresponding sentences $t=0$ and $t\neq0$ are equivalent to $0=0$ or $0\neq0,$ which can be trivially dismissed from the $\{\vee,\wedge\}$-combination.
\end{proof}

\begin{cor}[AKE$^{\Rightarrow_{\mirrorexists_{n}}},$ AKE$^{\preceq_{\mirrorexists_{n}}}$]
\label{akesen}
    Let $\MM,\NN\models \texttt{SAMK}_{e}^{\l,\ac}$ and let $n\geq1.$ Then: 
    \begin{enumerate}[label*={\arabic*.}]
        \item $\MM\Rightarrow_{\mirrorexists_{n}}\NN$ if and only if $\GG(M)\Rightarrow_{\mirrorexists_{n}}\GG(N)$ as $\LL_\GG$-structures and $\kk(M)\Rightarrow_{\mirrorexists_{n}}\kk(N)$ as $\LL_\kk$-structures. 

        \item If $\MM\subseteq\NN,$ then $\MM\preceq_{\mirrorexists_{n}}\NN$ if and only if $\GG(M)\preceq_{\mirrorexists_{n}}\GG(N)$ as $\LL_\GG$-structures and $\kk(M)\preceq_{\mirrorexists_{n}}\kk(N)$ as $\LL_\kk$-structures.
        
    \end{enumerate}
\end{cor}
\begin{proof}
    All direct implications are clear. We prove their converse implications. 
    \begin{enumerate}[label*={\arabic*.}]
        \item Suppose that $\kk(M)\Rightarrow_{\mirrorexists_{n}}\kk(N)$ and $\GG(M)\Rightarrow_{\mirrorexists_{n}}\GG(M).$ By Corollary \ref{sent}, any $\LL$-sentence is equivalent, modulo $\texttt{SAMK}_{e}^{\l,\ac},$ to a $\{\vee,\wedge\}$-combination of $\mirrorexists_{n}$-sentences from $\LL(\kk)$ and $\LL(\GG),$ so the hypotheses yield that $\MM\Rightarrow_{\mirrorexists_{n}}\NN.$

        \item Suppose that $\MM\subseteq\NN,$ that $\kk(M)\preceq_{\mirrorexists_{n}}\kk(N)$ and that $\GG(M)\preceq_{\mirrorexists_{n}}\GG(N)$. If $t(x)$ is an $\LL(\KK)$-term, then $\MM\subseteq\NN$ implies that $\MM\models t(m)=0$ if and only if $\NN\models t(m)=0$ for any $|x|$-tuple $m$ of $M.$ Also, since $\kk(M)\preceq_{\mirrorexists_{n}}\kk(N)$ and $\GG(M)\preceq_{\mirrorexists_{n}}\GG(N),$ then $\MM\models\chi(m)$ if and only if $\NN\models\chi(m)$ for any $\{\vee,\wedge\}$-combination $\chi(x)$ of $\mirrorexists_{n}$-formulas from $\LL(\kk)$ and $\LL(\GG).$ We conclude that $\MM\preceq_{\mirrorexists_{n}}\NN$ by Theorem \ref{delredn}.  
        \qedhere
    \end{enumerate}
\end{proof}

\subsection{Conditional \texorpdfstring{$\mirrorexists_{n}$}{}-NIP Transfer}

In Section \ref{sectlres} we saw that performing $\l$-resolution to an indiscernible sequence of tuples of length $N,$ with respect to any $\LL_{\l}$-term, produces an indiscernible sequence of tuples of length \emph{at least} $N$. The goal of this section is to prove a refined version of Corollary \ref{fulltrans}, replacing NIP by $\mirrorexists_{n}$-NIP for any given $n\geq 1,$ conditional to Condition \ref{acv}, which is an adaptation of \cite[Lemma A.18]{simon} to our multi-variable context. 
For readability purposes, we now change the notation for indiscernible sequences of tuples, putting the coordinate index as a subscript rather than as a superscript, i.e.~the $i^{th}$ element $a_i$ of a sequence $a_{<\omega}$ of $N$-tuples is now noted as $(a_{i,1},\.,a_{i,N})$ rather than $(a_{i}^{1},\.,a_{i}^{N}).$ We do this so that exponentiation keeps its original notation: for example, $a_{i}^{2}$ denotes the \emph{square} of $a_i$ rather than the second coordinate of $a_i.$   

\begin{condition}[Cf.~{\cite[Lemma A.18]{simon}} for $N=1$]
\label{acv}
    Let $(M,v,\ac)$ be an ac-valued field and, for some $N\geq 1,$ let $a_{<\omega}$ be an indiscernible sequence of $N$-tuples of $M.$ Let $\mathbf{d}$ be any finite tuple from $M.$ Then there is an indiscernible sequence $(v_{<\omega,1},\.,v_{<\omega,N})$ of $N$-tuples of $vM$ and an indiscernible sequence $(c_{<\omega,1},\.,c_{<\omega,N})$ of $N$-tuples of $Mv$ such that,
    for any $P(X_1,\.,X_N)\in \F_p[\mathbf{d}][X_1,\.,X_N],$
    \begin{enumerate}[label*={\arabic*.}]
        \item there are $r_1,\.,r_{N}\in\N$ and $\g\in vM$ such that $v(P(a_i))=\g+\sum_{j=1}^{N}r_j\cdot v_{i,j}$, and
        \item there is some $q(X_1,\.,X_N)\in Mv[X_1,\.,X_N]$ such that $\ac(P(a_i))=q(c_{i,1},\.,c_{i,N})$  
    \end{enumerate}
    for all sufficiently large $i<\omega$.
\end{condition}

\begin{propo}
\label{niptrd}
    Let $(M,v,\ac)$ be a non-trivial ac-valued field of equal characteristic.
    Let $\LL_{3s}=\LL_{3s}(\LL_{\kk},\LL_{\GG}),$ $\LL=\LL(\LL_{\kk},\LL_{\GG})$ and let $\MM$ and $\MM^{\l,\ac}$ denote the $\LL_{3s}$- and $\LL$-structures generated by $(M,v,\ac)$ respectively. Finally, let $n\geq 1.$
    \begin{enumerate}[label*={\arabic*.}]
        \item If $\MM$ is $\mirrorexists_{n}$-NIP, then $(M,v)$ is separably-defectless Kaplansky, $Mv$ is $\mirrorexists_{n}$-NIP with respect to $\LL_\kk$ and $vM$ is $\mirrorexists_{n}$-NIP with respect to $\LL_\GG.$ 

        \item Suppose Condition \ref{acv} holds. If $(M,v)$ is separable-algebraically maximal Kaplansky, then $\MM^{\l,\ac}$ is $\mirrorexists_{n}$-NIP whenever $Mv$ is $\mirrorexists_{n}$-NIP with respect to $\LL_\kk$ and $vM$ is $\mirrorexists_{n}$-NIP with respect to $\LL_\GG.$
    \end{enumerate}
\end{propo}
\begin{proof}
    \begin{enumerate}[label*={\arabic*.}]
        \item Follows from Lemmas \ref{sd} and \ref{resfieldint}.

        \item Since $\MM^{\l,\ac}\models\texttt{SAMK}^{\l,\ac}_e,$ where $e$ is the Ershov degree of $M,$ we only need to check that the highlighted formulas of Proposition \ref{delredn} are NIP. By Corollary \ref{lqfnip}, the formulas of the form of item 1 are NIP. Now, for the formula $\Phi(x;yz)=\th(\ac(t_1(x,y)),\.,\ac(t_{\nu}(x,y)),z)$ of the form of item 2, if $\NN$ is an $\aleph_1$-saturated elementary extension of $\MM^{\l,\ac},$ $a_{<\omega}$ is an indiscernible sequence of $|x|$-tuples of $N,$ $\mathbf{b}$ is a $|y|$-tuple of parameters of $N$ and $\ol{b}$ is a $|z|$-tuple of parameters of $Nv$ such that $\NN\models\Phi(a_i;\mathbf{b},\ol{b})$ if and only if $i$ is even, then by Lemma \ref{indtra}, there is some $l<\omega,$  some polynomials $P_1,\.,P_{\nu},Q_1,\.,Q_{\nu}$ over $\F_p,$ some indiscernible sequences $\a^{(1)}_{<\omega},\.,\a^{(\nu)}_{<\omega}$ and some parameters $d_1,\.,d_{\nu}$ such that $\a_i^{(k)}\subseteq^{u}\dcl(a_ia_{<l})$ for all $i\geq l$ and all $k\in[\nu],$ and $t_k(a_i,\mathbf{b})=P_k(\a^{(k)}_i,d_i)/Q_k(\a^{(k)}_i,d_i)$ for all $k\in[\nu]$ and all sufficiently large $i\geq l.$
        It follows that there is some $\mirrorexists_{n}$-formula $\th'$ from $\LL_\kk$ such that $$\Phi(a_i;\mathbf{b},\ol{b})=\th'(\ac(P_1(\a_i^{(1)},d_1)),\ac(Q_1(\a_i^{(1)},d_1)),\.,\ac(P_{\nu}(\a_i^{(\nu)},d_{\nu})),\ac(Q_{\nu}(\a_i^{(\nu)},d_{\nu})),\ol{b})$$ for all sufficiently large $i<\omega.$ Note that the sequence $\a_{<\omega}=(\a_i^{(1)},\.,\a_i^{(\nu)})_{i<\omega}$ remains indiscernible by definition of $\subseteq^{u}$.  
        Finally, applying Condition \ref{acv} to the polynomials $R_k,S_k,$ the sequence $\a_{<\omega}$ and the parameters $\mathbf{d},$ if $N=|\a_0|,$ we can find an indiscernible sequence $(c_{<\omega,1},\.,c_{<\omega,N})$ of $N$-tuples of $Mv$, some $\mirrorexists_{n}$-formula $\th''$ from $\LL_\kk$ and some parameters $\e$ of $Nv$ such that, if $i<\omega$ is sufficiently large, 
        $\NN\models \th''(c_{i,1},\.,c_{i,N},\e,\ol{b})$ if and only if $i$ is even, which contradicts the hypothesis of $Nv\succeq Mv$ being $\mirrorexists_{n}$-NIP with respect to $\LL_{\kk}$.
        The remaining splittings of the free variables of $\Phi$ are clear. The result for formulas of the form of item 3 is analogous.  
        \qedhere
    \end{enumerate}
\end{proof}

\begin{cor}[Conditional $\mirrorexists_{n}$-NIP Transfer]
\label{enniptr}
    Let $\LL_{3s}=\LL_{3s}(\LL_\kk,\LL_\GG),$ let $n\geq 1$ and let $\MM$ be the $\LL_{3s}$-structure generated by a henselian equi-characteristic valued field $(M,v)$. If $\MM$ is $\mirrorexists_{n}$-NIP, then $(M,v)$ is separably-defectless Kaplansky, $Mv$ is $\mirrorexists_{n}$-NIP with respect to $\LL_{\kk}$ and $vM$ is $\mirrorexists_{n}$-NIP with respect to $\LL_{\GG}$. The converse implication holds if Condition \ref{acv} is satisfied.
\end{cor}
\begin{proof}
    The direct implication is Statement 1 of Proposition \ref{niptrd}. For the converse implication, if $\MM$ is separably-defectless, Kaplansky and henselian, then it is a separable-algebraically maximal Kaplansky field. If $\NN\succeq\MM$ is $\aleph_1$-saturated, then $\NN$ admits an angular component map $\ac.$ If $\NN^{\l,\ac}$ is the associated $\LL$-structure, then $\NN^{\l,\ac}\models\texttt{SAMK}_e^{\l,\ac}$ where $e$ equals the Ershov degree of $M,$ and $Nv$ and $vN$ are still $\mirrorexists_{n}$-NIP with respect to $\LL_\kk$ and $\LL_\GG$ respectively. It follows by Statement 2 of Proposition \ref{niptrd} that $\NN^{\l,\ac}$ and, in particular, $\MM,$ are $\mirrorexists_{n}$-NIP as wanted. 
\end{proof}

If we set $\LL_{\GG}$ to be the usual language of ordered abelian groups with a symbol $\infty$ for a point at infinity, then \cite[Theorem 3.1]{gs} shows that any theory of ordered abelian groups is NIP in this language. This means that, in this language, formulas from the value group sort play no role in the transfer principle. Thus, if we set $\LL_{\kk}$ to be the language of rings, we get our third main theorem as a corollary of Corollary \ref{enniptr}.

\begin{teorema}
\label{nonrespltr}
    Let $\LL_{3s}$ be the usual three-sorted language of valued fields, let $n\geq 1$ and let $\MM$ be the $\LL_{3s}$-structure generated by a henselian equi-characteristic valued field $(M,v)$. If $\MM$ is $\mirrorexists_{n}$-NIP, then $(M,v)$ is separably-defectless Kaplansky and $Mv$ is $\mirrorexists_{n}$-NIP with respect to $\LL_{ring}$. The converse implication holds if Condition \ref{acv} is satisfied.
\end{teorema}

The study of Condition \ref{acv} is still in progress. For instance, we were able to note the following generalization of \cite[Lemma A.17]{simon}. 

\begin{lema}[Cf.~{\cite[Lemma A.17]{simon}}]
\label{mono}
    For each $i<\omega,$ let $a_i=(a_{i,1},\.,a_{i,N})$ be an $N$-tuple of non-zero elements of an ac-valued field $(M,v,\ac).$ If the sequence $a_{<\omega}$ is indiscernible and the set $\left\{v\left(\frac{a_{1,j}}{a_{0,j}}\right):j\in[N]\right\}$ is rationally independent in $vM,$ then for any polynomial $P\in M[X_1,\.,X_N],$ there are some $\g\in vM,$ some $\a\in Mv$ and some non-negative integers $r_1,\.,r_{N}$ such that $$v(P(a_{i,1},\.,a_{i,N}))=\g+\sum_{j=1}^{N}r_j\cdot v(a_{i,j})\text{ and }\ac(P(a_{i,1},\.,a_{i,N}))=\a\cdot\prod_{j=1}^{N}\ac(a_{i,j})^{r_j}$$ for all sufficiently large $i<\omega.$  
\end{lema}
If $N=1,$ this lemma does generalize \cite[Lemma A.17]{simon}. Indeed, the sequence $(v(a_{i,1}))_{i<\omega}$ is monotone if and only if $v(\frac{a_{1,1}}{a_{0,1}})$ is nonzero, which in turn is equivalent to $\left\{v(\frac{a_{1,1}}{a_{0,1}})\right\}$ being a rationally independent subset of $vM.$

\begin{proof}[Proof of Lemma \ref{mono}]
    Note first that if the set $\left\{v\left(\frac{a_{1,j}}{a_{0,j}}\right):j\in[N]\right\}$ is rationally independent in $vM,$ then for any non-trivial $N$-tuple of integers $(r_1,\.,r_N),$ we have that $\sum_{j=1}^{N}r_j\cdot v(a_{0,j})\neq \sum_{j=1}^{N}r_j\cdot v(a_{1,j}),$ which, by indiscernibility, implies that the sequence $\left(\sum_{j=1}^{N}r_j\cdot v(a_{i,j})\right)_{i<\omega}$ is monotone. Second, write $P(X_1,\.,X_N)=\sum_{\nu=1}^{k}b_{\nu}X_1^{m_{\nu,1}}\.X_N^{m_{\nu,N}}$ where $\{(m_{\nu,1},\.,m_{\nu,N}):\nu\in[k]\}$ is a set of pairwise different $N$-tuples of $\N.$ It follows from monotonicity of the sequence $\left(\sum_{j=1}^{N}r_j\cdot v(a_{i,j})\right)_{i<\omega}$ that there is some $i^{*}<\omega$ such that $$\sum_{j=1}^{N}(m_{l,j}-m_{l',j})\cdot v(a_{i,j})\not\in\Bigg\{v(b_\nu)-v(b_{\nu'}):\nu,\nu'\in[k],\nu\neq \nu'\Bigg\}$$ for any $i\geq i^{*}$ and any pair $l,l'\in[k]$ with $l\neq l'.$  
    Let $S_i$ be the set of values of the monomials of $P(a_{i,1},\.,a_{i,N}),$ namely $$S_i=\left\{v(b_{\nu})+\sum_{j=1}^{N}m_{\nu,j}\cdot v(a_{i,j}):\nu\in[k]\right\}.$$ In particular, for all $i\geq i^{*},$ $S_i$ consists of pairwise different elements of $vM,$ so there is some unique index $\nu_i\in[k]$ such that $v(b_{\nu_i})+\sum_{j=1}^{N}m_{\nu_i,j}\cdot v(a_{i,j})$ is the minimum element of $S_i.$ 
    Since the sequence $(v(a_{i,1}),\.,v(a_{i,N}))_{i<\omega}$ is indiscernible and the ordered abelian group $vM$ is NIP, cf.~\cite[Theorem 3.1]{gs}, there is some common $\nu\in[k]$ such that $v(b_{\nu})+\sum_{j=1}^{N}m_{\nu,j}\cdot v(a_{i,j})$ is the unique minimum element of $S_i$ for all sufficiently large $i<\omega.$ 
    It follows that $v(P(a_{i,1},\.,a_{i,N}))=v(b_{\nu})+\sum_{j=1}^{N}m_{\nu,j}\cdot v(a_{i,j})$ and $\ac(P(a_{i,1},\.,a_{i,N}))=\ac(b_{\nu})\cdot\prod_{j=1}^{N}\ac(a_{i,j})^{m_{\nu,j}}$ for all sufficiently large $i<\omega.$
\end{proof}

If we witness IP with longer indiscernible sequences, Fact \ref{indisctoeindisc} suggests a characterisation of henselian NIP valued fields of equal characteristic.

\begin{condition}
\label{dindis}
    Let $\LL=\LL(\LL_{\kk},\LL_{\GG})$ and let $\MM$ be the $\LL$-structure generated by some ac-valued field $(M,v,\ac)$ of equal characteristic. Let $\mathbf{d}$ be any finite tuple from $M,$ and for some $N\geq 1,$ let $a_{<|T|^{+}}$ be an indiscernible sequence of $N$-tuples of $M.$ Then there is some $\LL_\GG$-indiscernible sequence $(v_{<|T|^{+},1},\.,v_{<|T|^{+},N})$ of $N$-tuples of $vM$ and some $\LL_\kk$-indiscernible sequence $(c_{<|T|^{+},1},\.,c_{<|T|^{+},N})$ of $N$-tuples of $Mv$ such that,
    for any $P(X_1,\.,X_N)\in \F_p[\mathbf{d}][X_1,\.,X_N],$ there is some $\a<|T|^{+}$, some $r_1,\.,r_{N}\in\N$ and $\g\in vM,$ and some $q(X_1,\.,X_N)\in Mv[X_1,\.,X_N]$ for which
    $$\begin{cases}
        v(P(a_{i}))=\g+\sum_{j=1}^{N}r_j\cdot v_{i,j}\text{ and}\\
        \ac(P(a_{i}))=q(c_{i,1},\.,c_{i,N})  
    \end{cases}$$
    for all $i<|T|^{+}$ with $\a\leq i.$
\end{condition}

\begin{defin}
    We say that a valued field seen as an $\LL_{3s}$-structure satisfies \emph{\textbf{Ind}} if for any $\LL_{3s}$-elementary extension that admits an angular component map, its associated $\LL$-structure satisfies Condition \ref{dindis}.
\end{defin}

\begin{propo}
\label{niptrdd}
    Let $(M,v)$ be a non-trivial henselian valued field of equal characteristic.
    Let $\LL_{3s}=\LL_{3s}(\LL_{\kk},\LL_{\GG}),$ $\LL_{\l}=\LL_{\l}(\LL_{\kk},\LL_{\GG})$ and let $\MM$ and $\MM^{\l}$ denote the $\LL_{3s}$- and $\LL_{\l}$-structures generated by $(M,v)$ respectively. The following statements are equivalent.
    \begin{enumerate}
        \item $\MM^{\l}$ is NIP.

        \item $\MM$ is NIP.

        \item $(M,v)$ is separably defectless Kaplansky, $vM$ is NIP with respect to $\LL_{\GG}$, $Mv$ is NIP with respect to $\LL_{\kk}$ and $\MM$ satisfies \textbf{Ind}.
    \end{enumerate}
\end{propo}
\begin{proof}
    \begin{enumerate}[wide]
        \item[$(1)\implies (2):$] Follows from $\LL_{3s}\subseteq\LL_{\l}.$ 
        
        \item[$(2)\implies(3):$] The first three statements follow from Lemmas \ref{sd} and \ref{dtod'nip}.
        We have to check that $\MM$ satisfies \textbf{Ind}. Let $\NN$ be an $\LL_{3s}$-elementary extension of $\MM$ with underlying valued field $(N,v)$ admitting an angular component map $\ac.$ If $\mathbf{d}$ is a finite tuple from $N$ and $a_{<|T|^{+}}$ is an indiscernible sequence of finite tuples from $N,$ then by Fact \ref{indisctoeindisc} there is some $\a<|T|^{+}$ for which the sequence $(a_i)_{\a<i<|T|^{+}}$ is indiscernible over $\mathbf{d}.$ 
        Therefore, if $P$ is any polynomial over $\F_p[\mathbf{d}],$ the sequences $(v(P(a_i)))_{\a<i<|T|^{+}}$ and $(\ac(P(a_i)))_{\a<i<|T|^{+}}$ are indiscernible with respect to $\LL_{\GG}$ and $\LL_{\kk}$ respectively.
        
        \item[$(3)\implies (1):$] Suppose that $\MM^{\l}$ has IP, so that there is some elementary extension $\NN^{\l}$ of $\MM^{\l}$ and some $\LL_\l$-formula that witnesses IP in $\NN^{\l}.$ Without loss of generality, we may assume that $\NN^{\l}$ is $\aleph_1$-saturated, so that its underlying valued field $(N,v)$ admits an angular component map $\ac.$ Let $e$ be the Ershov degree of $M.$ Since henselian valued fields that are separably defectless are separable-algebraically maximal, we get that $(M,v)$ is $\mathrm{SAMK}$ and thus the extension $\NN^{\l,\ac}$ of $\NN^{\l}$ by $\ac$ is a model of $\texttt{SAMK}_e^{\l,\ac}.$        
        By Proposition \ref{delredn}, one of the highlighted formulas has IP in $\NN^{\l,\ac}$. By Corollary \ref{lqfnip}, the formulas of the form of item 1 are NIP. Suppose that the formula $\Phi(x;yz)=\th(\ac(t_1(x,y)),\.,\ac(t_{\nu}(x,y)),z)$ of the form of item 2 has IP in $\NN^{\l,\ac}$. Then there is an indiscernible sequence $a_{<|T|^{+}}$ of $|x|$-tuples of $N,$ a $|y|$-tuple $\mathbf{b}$ of parameters of $N$ and a $|z|$-tuple $\ol{b}$ of parameters of $Nv$ such that $\NN^{\l,\ac}\models\Phi(a_i;\mathbf{b},\ol{b})$ if and only if $i$ is even.\footnote{The independence property for endless indiscernible sequences of arbitrary length $\k$ should be written as follows: there is no $\a<\k$ for which the truth value of $\Phi(a_i;\mathbf{b},\ol{b})$ in $\NN^{\l,\ac}$ is constant for all $i>\a.$ Since $|T|^{+}$ is regular, the set of indices $i<|T|^{+}$ for which $\NN^{\l,\ac}\models\neg(\Phi(a_i;\mathbf{b},\ol{b})\sii\Phi(a_{i+1};\mathbf{b},\ol{b}))$ is cofinal in $|T|^{+}$, so we may assume that such a set is the set of even ordinals.} By Lemma \ref{indtra}, there is some $l<\omega,$ some polynomials $P_1,\.,P_{\nu},Q_1,\.,Q_{\nu}$ over $\F_p,$ some indiscernible sequences $\a^{(1)}_{<|T|^{+}},\.,\a^{(\nu)}_{<|T|^{+}}$ and some parameters $d_1,\.,d_{\nu}$ such that $\a_i^{(k)}\subseteq^{u}\dcl(a_ia_{<l})$ for all $i\geq l$ and all $k\in[\nu],$ and $t_k(a_i,\mathbf{b})=P_k(\a^{(k)}_i,d_i)/Q_k(\a^{(k)}_i,d_i)$ for all $k\in[\nu]$ and all sufficiently large $i\geq l.$
        It follows that there is some $\D$-formula $\th'$ from $\LL_\kk$ such that $$\models\Phi(a_i;\mathbf{b},\ol{b})\sii\th'(\ac(P_1(\a_i^{(1)},d_1)),\ac(Q_1(\a_i^{(1)},d_1)),\.,\ac(P_{\nu}(\a_i^{(\nu)},d_{\nu})),\ac(Q_{\nu}(\a_i^{(\nu)},d_{\nu})),\ol{b})$$ for all sufficiently large $i<|T|^{+}.$ Note that the sequence $(\a_i^{(1)},\.,\a_i^{(\nu)})_{i<|T|^{+}}$ remains indiscernible by definition of $\subseteq^{u}$. 
        By \textbf{Ind}, $\NN^{\l,\ac}$ satisfies Condition \ref{dindis} and thus, if $\mu=|(\a_0^{(1)},\.,\a_0^{(\nu)})|,$ there is some $\LL_{\kk}$-indiscernible sequence $(c_{<|T|^{+},1},\.,c_{<|T|^{+},\mu})$ of $\mu$-tuples of $Nv$, some $\LL_{\kk}$-formula $\th''$ and some parameters $\e$ of $Nv$ such that, if $i<|T|^{+}$ is sufficiently large, then
        \begin{equation}
            \label{eqfin}
            \NN^{\l,\ac}\models\th''(c_{i,1},\.,c_{i,\mu},\e,\ol{b})\text{ if and only if }i\text{ is even.}
        \end{equation}
        However, by Fact \ref{indisctoeindisc} there is some $\a<|T|^{+}$ for which the sequence $(c_{i,1},\.,c_{i,\mu})_{\a<i<|T|^{+}}$ is $\LL_{\kk}$-indiscernible over $(\e,\ol{b}),$ contradicting Statement \ref{eqfin}.    
        The result for formulas of the form of item 3 is analogous. The remaining splittings of the free variables of $\Phi$ are clear. 
        \qedhere
    \end{enumerate}
\end{proof}

If we set $\LL_{\kk}=\LL_{ring}$ and $\LL_{\GG}=\LL_{og}\cup\{\infty\},$ we obtain the following Corollary. 

\begin{cor}
    Let $(M,v)$ be a non-trivial henselian valued field of equal characteristic.
    Let $\LL_{3s}=\LL_{3s}(\LL_{ring},\LL_{og}\cup\{\infty\})$ and let $\MM$ denote the $\LL_{3s}$-structure generated by $(M,v).$ If $(M,v)$ is separably defectless Kaplansky $Mv$ is NIP with respect to $\LL_{\kk},$ then $\MM$ satisfies \textbf{Ind}. In particular, Condition \ref{dindis} holds for any henselian ac-valued field $(M,v,\ac)$ of equal characteristic for which $(M,v)$ is separably defectless Kaplansky and $Mv$ is NIP.
\end{cor}
\begin{proof}
    By Gurevich-Schmitt's \cite[Theorem 3.1]{gs}, $vM$ is NIP with respect to $\LL_{og}\cup\{\infty\}.$ By Corollary \ref{fulltrans}, $\MM$ is NIP. The result follows by Proposition \ref{niptrdd}.
\end{proof}

\section*{Acknowledgments}
The author would like to thank his supervisor Sylvy Anscombe for encouraging him to pursue this project and for all her valuable advice, Franziska Jahnke and Itay Kaplan for suggesting the improvement of Theorem \ref{eniphens} to positive existential formulas and type-definable fields in NIP structures as in Theorem \ref{finereniphens}, and Itay Kaplan and Martin Bays for suggesting Proposition \ref{niptransfer} as a model-theoretic NIP transfer principle that broadens the one known for henselian valued fields. The author has received funding from the European Union’s Horizon 2020 research and innovation program under the Marie Sk\lw odowska-Curie grant agreement No 945332. 
\includegraphics*[scale = 0.028]{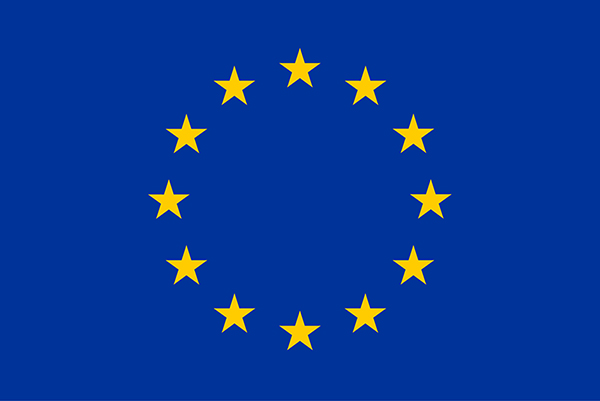}

\bibliographystyle{alpha}
\bibliography{bib}

@misc{ht,
      title={Immediate expansions by valuation of fields}, 
      author={Jizhan Hong},
      year={2013},
      note={PhD Thesis. Available at \url{https://macsphere.mcmaster.ca/handle/11375/13278}},
}

@misc{sm,
      title={Relative {Q}uantifier {E}limination for {S}eparable-{A}lgebraically {M}aximal {K}aplansky {F}ields}, 
      author={Paulo Andrés Soto Moreno},
      year={2025},
      eprint={2505.07418},
      archivePrefix={arXiv},
      primaryClass={math.LO},
      url={https://arxiv.org/abs/2505.07418}, 
note={\href{https://arxiv.org/abs/2505.07418}{arXiv:2505.07418 [math.LO]}}
}

@misc{at,
      title={On Lambda Functions in Henselian and Separably Tame Valued Fields}, 
      author={Sylvy Anscombe},
      year={2025},
    eprint={2505.07518},
    archivePrefix={arXiv},
    primaryClass={math.LO},
      url={https://arxiv.org/abs/2505.07518}, 
note={\href{https://arxiv.org/abs/2505.07518}{arXiv:2505.07518 [math.LO]}}
}

@article{dp1a,
title = {Dp-finite fields {I}(A): The infinitesimals},
journal = {Annals of Pure and Applied Logic},
volume = {172},
number = {6},
pages = {102947},
year = {2021},
issn = {0168-0072},
doi = {https://doi.org/10.1016/j.apal.2021.102947},
url = {https://www.sciencedirect.com/science/article/pii/S0168007221000063},
author = {Will Johnson},
}

@article {ax,
    AUTHOR = {Ax, James},
     TITLE = {The elementary theory of finite fields},
   JOURNAL = {Ann. of Math. (2)},
  FJOURNAL = {Annals of Mathematics. Second Series},
    VOLUME = {88},
      YEAR = {1968},
     PAGES = {239--271},
      ISSN = {0003-486X},
   MRCLASS = {10.80},
  MRNUMBER = {229613},
MRREVIEWER = {D.\ J.\ Lewis},
       DOI = {10.2307/1970573},
       URL = {https://doi.org/10.2307/1970573},
}

@article{fvk17,
title = {Places of algebraic function fields in arbitrary characteristic},
journal = {Advances in Mathematics},
volume = {188},
number = {2},
pages = {399-424},
year = {2004},
issn = {0001-8708},
doi = {https://doi.org/10.1016/j.aim.2003.07.021},
url = {https://www.sciencedirect.com/science/article/pii/S0001870803003451},
author = {Franz-Viktor Kuhlmann}
}

@article{feh,
author = {Fehm, Arno},
year = {2011},
month = {03},
pages = {533-544},
title = {Embeddings of function fields into ample fields},
volume = {134},
journal = {Manuscripta Mathematica},
doi = {10.1007/s00229-010-0415-8}
}

@book {hodges,
    AUTHOR = {Hodges, Wilfrid},
     TITLE = {Model theory},
    SERIES = {Encyclopedia of Mathematics and its Applications},
    VOLUME = {42},
 PUBLISHER = {Cambridge University Press, Cambridge},
      YEAR = {1993},
     PAGES = {xiv+772},
      ISBN = {0-521-30442-3},
   MRCLASS = {03-01 (03-02 03Cxx)},
  MRNUMBER = {1221741},
MRREVIEWER = {J.\ M.\ Plotkin},
       DOI = {10.1017/CBO9780511551574},
       URL = {https://doi.org/10.1017/CBO9780511551574},
}

@book {poizmt,
    AUTHOR = {Poizat, Bruno},
     TITLE = {A course in model theory},
    SERIES = {Universitext},
      NOTE = {An introduction to contemporary mathematical logic,
              Translated from the French by Moses Klein and revised by the
              author},
 PUBLISHER = {Springer-Verlag, New York},
      YEAR = {2000},
     PAGES = {xxxii+443},
      ISBN = {0-387-98655-3},
   MRCLASS = {03Cxx (03-02 03C07 03C45)},
  MRNUMBER = {1757487},
       DOI = {10.1007/978-1-4419-8622-1},
       URL = {https://doi.org/10.1007/978-1-4419-8622-1},
}

@incollection {jf,
    AUTHOR = {Fehm, Arno and Jahnke, Franziska},
     TITLE = {Recent progress on definability of {H}enselian valuations},
 BOOKTITLE = {Ordered algebraic structures and related topics},
    SERIES = {Contemp. Math.},
    VOLUME = {697},
     PAGES = {135--143},
 PUBLISHER = {Amer. Math. Soc., Providence, RI},
      YEAR = {2017},
      ISBN = {978-1-4704-2966-9; 978-1-4704-4222-4},
   MRCLASS = {12J10 (03C60 12L12)},
  MRNUMBER = {3716069},
MRREVIEWER = {Franz-Viktor\ Kuhlmann},
       DOI = {10.1090/conm/697/14049},
       URL = {https://doi.org/10.1090/conm/697/14049},
}

@incollection {delon,
    AUTHOR = {Delon, Fran\c{c}oise},
     TITLE = {Types sur {${\bf C}((X))$}},
 BOOKTITLE = {Study {G}roup on {S}table {T}heories ({B}runo {P}oizat),
              {S}econd year: 1978/79},
     PAGES = {Exp. No. 5, 29},
 PUBLISHER = {Secr\'{e}tariat Math., Paris},
      YEAR = {1981},
      ISBN = {2-85926-274-1},
   MRCLASS = {03C60 (03C45 12L99)},
  MRNUMBER = {620032},
MRREVIEWER = {G.\ Cherlin},
note = {In French},
}

@book{ss, place={Cambridge}, series={New Mathematical Monographs}, title={Spectral Spaces}, publisher={Cambridge University Press}, author={Dickmann, Max and Schwartz, Niels and Tressl, Marcus}, year={2019}, collection={New Mathematical Monographs}}

@article{af,
  author  = {Anscombe, Sylvy and Fehm, Arno},
  title   = {Interpretations of syntactic fragments of theories of fields},
  journal = {Israel Journal of Mathematics},
  year    = {2026},
  date    = {2026-08-27},
  issn    = {1565-8511},
  doi     = {10.1007/s11856-026-2938-z},
  url     = {https://doi.org/10.1007/s11856-026-2938-z},
}

@misc{b,
      title={{NIP}n {CHIPS}}, 
      author={Blaise Boissonneau},
      year={2024},
      eprint={2401.04697},
      archivePrefix={arXiv},
      primaryClass={math.LO},
      url={https://arxiv.org/abs/2401.04697}, 
}

@article{h,
 ISSN = {00224812, 19435886},
 URL = {http://www.jstor.org/stable/44161512},
 author = {Jizhan Hong},
 journal = {The Journal of Symbolic Logic},
 number = {3},
 pages = {887--900},
 publisher = {[Association for Symbolic Logic, Cambridge University Press]},
 title = {SEPARABLY CLOSED VALUED FIELDS: QUANTIFIER ELIMINATION},
 urldate = {2024-12-06},
 volume = {81},
 year = {2016}
}

@article {r,
    AUTHOR = {Rideau, Silvain},
     TITLE = {Some properties of analytic difference valued fields},
   JOURNAL = {J. Inst. Math. Jussieu},
  FJOURNAL = {Journal of the Institute of Mathematics of Jussieu. JIMJ.
              Journal de l'Institut de Math\'{e}matiques de Jussieu},
    VOLUME = {16},
      YEAR = {2017},
    NUMBER = {3},
     PAGES = {447--499},
      ISSN = {1474-7480},
   MRCLASS = {03C10 (03C45 03C60 12H10 32B05 32P05)},
  MRNUMBER = {3646280},
MRREVIEWER = {G. Cherlin},
       DOI = {10.1017/S1474748015000183},
       URL = {https://doi.org/10.1017/S1474748015000183},
}

@article {a,
    AUTHOR = {Anscombe, Sylvy},
     TITLE = {Existentially generated subfields of large fields},
   JOURNAL = {J. Algebra},
  FJOURNAL = {Journal of Algebra},
    VOLUME = {517},
      YEAR = {2019},
     PAGES = {78--94},
      ISSN = {0021-8693},
   MRCLASS = {12L12 (03C60 12J10)},
  MRNUMBER = {3869267},
MRREVIEWER = {Ricardo Bianconi},
       DOI = {10.1016/j.jalgebra.2018.09.021},
       URL = {https://doi.org/10.1016/j.jalgebra.2018.09.021},
}

@article {fvktf,
    AUTHOR = {Kuhlmann, Franz-Viktor},
     TITLE = {The algebra and model theory of tame valued fields},
   JOURNAL = {J. Reine Angew. Math.},
  FJOURNAL = {Journal f\"{u}r die Reine und Angewandte Mathematik. [Crelle's
              Journal]},
    VOLUME = {719},
      YEAR = {2016},
     PAGES = {1--43},
      ISSN = {0075-4102},
   MRCLASS = {03C60 (12J10)},
  MRNUMBER = {3552490},
MRREVIEWER = {Sylvy Anscombe},
       DOI = {10.1515/crelle-2014-0029},
       URL = {https://doi.org/10.1515/crelle-2014-0029},
}

@book {fj,
    AUTHOR = {Fried, Michael D. and Jarden, Moshe},
     TITLE = {Field arithmetic},
    SERIES = {Ergebnisse der Mathematik und ihrer Grenzgebiete. 3. Folge. A Series of Modern Surveys in Mathematics},
    VOLUME = {11},
PUBLISHER = {Springer, Cham},
      YEAR = {2023},
      ISBN = {978-3-031-28019-1; 978-3-031-28020-7},
   MRCLASS = {12E30 (03B25 03C10 03C20 03C60 11Rxx)},
  MRNUMBER = {4647281},
       DOI = {10.1007/978-3-031-28020-7},
       URL = {https://doi.org/10.1007/978-3-031-28020-7},
}

@book {ep,
    AUTHOR = {Engler, Antonio J. and Prestel, Alexander},
     TITLE = {Valued fields},
    SERIES = {Springer Monographs in Mathematics},
 PUBLISHER = {Springer-Verlag, Berlin},
      YEAR = {2005},
     PAGES = {x+205},
      ISBN = {978-3-540-24221-5; 3-540-24221-X},
   MRCLASS = {12J20 (12F05 12J10 12J12 12J15)},
  MRNUMBER = {2183496},
MRREVIEWER = {Niels Schwartz},
}

@article {ksw,
    AUTHOR = {Kaplan, Itay and Scanlon, Thomas and Wagner, Frank O.},
     TITLE = {Artin-{S}chreier extensions in {NIP} and simple fields},
   JOURNAL = {Israel J. Math.},
  FJOURNAL = {Israel Journal of Mathematics},
    VOLUME = {185},
      YEAR = {2011},
     PAGES = {141--153},
      ISSN = {0021-2172,1565-8511},
   MRCLASS = {03C45 (12J20 12L12)},
  MRNUMBER = {2837131},
MRREVIEWER = {Assaf\ Hasson},
       DOI = {10.1007/s11856-011-0104-7},
       URL = {https://doi.org/10.1007/s11856-011-0104-7},
}

@InProceedings{duret,
author="Duret, Jean-Louis",
title="Les corps faiblement algebriquement clos non separablement clos ont la propriete d'independance",
booktitle="Model Theory of Algebra and Arithmetic",
year="1980",
publisher="Springer Berlin Heidelberg",
address="Berlin, Heidelberg",
pages="136--162",
isbn="978-3-540-38393-2"
}

@book {simon,
    AUTHOR = {Simon, Pierre},
     TITLE = {A guide to {NIP} theories},
    SERIES = {Lecture Notes in Logic},
    VOLUME = {44},
 PUBLISHER = {Association for Symbolic Logic, Chicago, IL; Cambridge
              Scientific Publishers, Cambridge},
      YEAR = {2015},
     PAGES = {vii+156},
      ISBN = {978-1-107-05775-3},
   MRCLASS = {03-02 (03C45 03C60 03C64 68Q32)},
  MRNUMBER = {3560428},
MRREVIEWER = {Alf Onshuus},
       DOI = {10.1017/CBO9781107415133},
       URL = {https://doi.org/10.1017/CBO9781107415133},
}

@article {js,
    AUTHOR = {Jahnke, Franziska and Simon, Pierre},
     TITLE = {N{IP} henselian valued fields},
   JOURNAL = {Arch. Math. Logic},
  FJOURNAL = {Archive for Mathematical Logic},
    VOLUME = {59},
      YEAR = {2020},
    NUMBER = {1-2},
     PAGES = {167--178},
      ISSN = {0933-5846,1432-0665},
   MRCLASS = {03C45 (12L12)},
  MRNUMBER = {4054852},
MRREVIEWER = {G.\ Cherlin},
       DOI = {10.1007/s00153-019-00685-8},
       URL = {https://doi.org/10.1007/s00153-019-00685-8},
}

@article{gs,
author = {Gurevich, Yuri and Schmitt, P.},
year = {1984},
month = {07},
pages = {},
title = {The Theory of Ordered Abelian Groups does not have the Independence Property},
volume = {284},
journal = {Transactions of The American Mathematical Society - Trans. Amer. Math. Soc.},
doi = {10.2307/1999281}
}

@article {kap,
    AUTHOR = {Kaplansky, Irving},
     TITLE = {Maximal fields with valuations},
   JOURNAL = {Duke Math. J.},
  FJOURNAL = {Duke Mathematical Journal},
    VOLUME = {9},
      YEAR = {1942},
     PAGES = {303--321},
      ISSN = {0012-7094},
   MRCLASS = {09.1X},
  MRNUMBER = {6161},
MRREVIEWER = {Saunders Mac Lane},
       URL = {http://projecteuclid.org/euclid.dmj/1077493226},
}

@article {afd,
    AUTHOR = {Anscombe, Sylvy and Dittmann, Philip and Fehm, Arno},
     TITLE = {Axiomatizing the existential theory of {$\mathbb{F}_q((t))$}},
   JOURNAL = {Algebra \& Number Theory},
  FJOURNAL = {Algebra \& Number Theory},
    VOLUME = {17},
      YEAR = {2023},
    NUMBER = {11},
     PAGES = {2013--2032},
      ISSN = {1937-0652,1944-7833},
   MRCLASS = {11U05 (03C60 11D88 11G25 12L05)},
  MRNUMBER = {4648855},
MRREVIEWER = {Franz-Viktor\ Kuhlmann},
       DOI = {10.2140/ant.2023.17.2013},
       URL = {https://doi.org/10.2140/ant.2023.17.2013},
}

@article{pop,
 ISSN = {0003486X, 19398980},
 URL = {http://www.jstor.org/stable/2118581},
 author = {Florian Pop},
 journal = {Annals of Mathematics},
 number = {1},
 pages = {1--34},
 publisher = {[Annals of Mathematics, Trustees of Princeton University on Behalf of the Annals of Mathematics, Mathematics Department, Princeton University]},
 title = {Embedding Problems Over Large Fields},
 urldate = {2026-01-29},
 volume = {144},
 year = {1996}
}

@article{aj,
   title={Characterizing {NIP} henselian fields},
   volume={109},
   ISSN={1469-7750},
   url={http://dx.doi.org/10.1112/jlms.12868},
   DOI={10.1112/jlms.12868},
   number={3},
   journal={Journal of the London Mathematical Society},
   publisher={Wiley},
   author={Anscombe, Sylvy and Jahnke, Franziska},
   year={2024},
   month=mar }

@book {oe,
    AUTHOR = {Endler, Otto},
     TITLE = {Valuation theory},
    SERIES = {Universitext},
      NOTE = {To the memory of Wolfgang Krull (26 August 1899--12 April
              1971)},
 PUBLISHER = {Springer-Verlag, New York-Heidelberg},
      YEAR = {1972},
     PAGES = {xii+243},
   MRCLASS = {12J20},
  MRNUMBER = {357379},
MRREVIEWER = {P. Ribenboim},
}

@book {tz,
    AUTHOR = {Tent, Katrin and Ziegler, Martin},
     TITLE = {A course in model theory},
    SERIES = {Lecture Notes in Logic},
    VOLUME = {40},
 PUBLISHER = {Association for Symbolic Logic, La Jolla, CA; Cambridge
              University Press, Cambridge},
      YEAR = {2012},
     PAGES = {x+248},
      ISBN = {978-0-521-76324-0},
   MRCLASS = {03C35 (03-01 03C45)},
  MRNUMBER = {2908005},
MRREVIEWER = {David Evans},
       DOI = {10.1017/CBO9781139015417},
       URL = {https://doi.org/10.1017/CBO9781139015417},
}

@article {fvkp,
    AUTHOR = {Kuhlmann, Franz-Viktor and Pal, Koushik},
     TITLE = {The model theory of separably tame valued fields},
   JOURNAL = {J. Algebra},
  FJOURNAL = {Journal of Algebra},
    VOLUME = {447},
      YEAR = {2016},
     PAGES = {74--108},
      ISSN = {0021-8693},
   MRCLASS = {03C60 (03C10 12J10 12L12)},
  MRNUMBER = {3427630},
MRREVIEWER = {G. Cherlin},
       DOI = {10.1016/j.jalgebra.2015.09.022},
       URL = {https://doi.org/10.1016/j.jalgebra.2015.09.022},
}

@book {ck,
    AUTHOR = {Chang, C. C. and Keisler, H. J.},
     TITLE = {Model theory},
    SERIES = {Studies in Logic and the Foundations of Mathematics},
    VOLUME = {73},
   EDITION = {Third},
 PUBLISHER = {North-Holland Publishing Co.},
   ADDRESS = {Amsterdam},
      YEAR = {1990},
     PAGES = {xvi+650},
      ISBN = {0-444-88054-2},
   MRCLASS = {03Cxx (03-01 03C57 03C95 03H05)},
  MRNUMBER = {1059055 (91c:03026)},
  BOEKCODE = {69F41},
}

@misc{dp5,
      title={Dp-finite fields {V}: {T}opological {F}ields of {F}inite {W}eight}, 
      author={Will Johnson},
      year={2020},
howpublished={Manuscript},
      eprint={2004.14732},
      archivePrefix={arXiv},
      primaryClass={math.LO},

}

@article {delscvf,
    AUTHOR = {Delon, Fran\c{c}oise},
     TITLE = {Id\'eaux et types sur les corps s\'eparablement clos},
   JOURNAL = {M\'em. Soc. Math. France (N.S.)},
  FJOURNAL = {M\'emoires de la Soci\'et\'e{} Math\'ematique de France.
              Nouvelle S\'erie},
    NUMBER = {33},
      YEAR = {1988},
     PAGES = {76},
      ISSN = {0037-9484},
   MRCLASS = {03C45 (03C60 12L12)},
  MRNUMBER = {986208},
MRREVIEWER = {John\ T.\ Baldwin},
note = {In French},
}

@article {ksh,
    AUTHOR = {Kaplan, Itay and Shelah, Saharon},
     TITLE = {Chain conditions in dependent groups},
   JOURNAL = {Ann. Pure Appl. Logic},
  FJOURNAL = {Annals of Pure and Applied Logic},
    VOLUME = {164},
      YEAR = {2013},
    NUMBER = {12},
     PAGES = {1322--1337},
      ISSN = {0168-0072,1873-2461},
   MRCLASS = {03C45 (03C98)},
  MRNUMBER = {3093393},
MRREVIEWER = {Alf\ Onshuus},
       DOI = {10.1016/j.apal.2013.06.014},
       URL = {https://doi.org/10.1016/j.apal.2013.06.014},
}

@article {tdf,
    AUTHOR = {Johnson, Will},
     TITLE = {Type-definable {NIP} fields are {A}rtin-{S}chreier closed},
   JOURNAL = {Fund. Math.},
  FJOURNAL = {Fundamenta Mathematicae},
    VOLUME = {260},
      YEAR = {2023},
    NUMBER = {3},
     PAGES = {251--261},
      ISSN = {0016-2736,1730-6329},
   MRCLASS = {03C60 (03C45)},
  MRNUMBER = {4538330},
MRREVIEWER = {Piotr\ L.\ Kowalski},
       DOI = {10.4064/fm149-8-2022},
       URL = {https://doi.org/10.4064/fm149-8-2022},
}

@article {pz,
    AUTHOR = {Prestel, Alexander and Ziegler, Martin},
     TITLE = {Model-theoretic methods in the theory of topological fields},
   JOURNAL = {J. Reine Angew. Math.},
  FJOURNAL = {Journal f\"{u}r die Reine und Angewandte Mathematik. [Crelle's
              Journal]},
    VOLUME = {299(300)},
      YEAR = {1978},
     PAGES = {318--341},
      ISSN = {0075-4102,1435-5345},
   MRCLASS = {54H13 (03C60 12J99)},
  MRNUMBER = {491852},
MRREVIEWER = {G.\ Cherlin},
       DOI = {10.1515/crll.1978.299-300.318},
       URL = {https://doi.org/10.1515/crll.1978.299-300.318},
}

@article {ajlring,
    AUTHOR = {Anscombe, Sylvy and Jahnke, Franziska},
     TITLE = {Henselianity in the language of rings},
   JOURNAL = {Ann. Pure Appl. Logic},
  FJOURNAL = {Annals of Pure and Applied Logic},
    VOLUME = {169},
      YEAR = {2018},
    NUMBER = {9},
     PAGES = {872--895},
      ISSN = {0168-0072,1873-2461},
   MRCLASS = {12J10 (03C60 12J25 12L12 13J15)},
  MRNUMBER = {3808399},
MRREVIEWER = {Luc\ B\'{e}lair},
       DOI = {10.1016/j.apal.2018.04.008},
       URL = {https://doi.org/10.1016/j.apal.2018.04.008},
}

\end{document}